\documentclass[leqno,12pt]{amsart}

\usepackage[T1]{fontenc}
\usepackage[english]{babel}

\usepackage{mathtools}
\usepackage{amssymb, amsthm, mathrsfs}
\usepackage{stmaryrd} %mapsfrom

\usepackage{comment}

\usepackage[all]{xy}
\usepackage{graphicx}
\usepackage{xcolor}
\usepackage[margin=1in]{geometry}
\usepackage[onehalfspacing]{setspace}
\usepackage{multicol}
\usepackage{tabularx}
\usepackage{enumitem}
\usepackage[normalem]{ulem}

\usepackage{CJKutf8}

\numberwithin{equation}{section}
\usepackage[
  colorlinks,
  allcolors=blue,
  pagebackref,
  bookmarks=false,
  pdfpagelabels,
  plainpages=false
]{hyperref}

\usepackage[nameinlink,noabbrev]{cleveref}

\usepackage[alphabetic,backrefs]{amsrefs}

\DeclareRobustCommand{\gobblefive}[5]{}

\newcommand{\la}{\langle}
\newcommand{\ra}{\rangle}

\newcommand{\HPI}{\mathcal{HPI}}
\newcommand{\wHPI}{\widetilde{\mathcal{HPI}}}
\newcommand{\uom}{{\underline \omega}}
\newcommand{\ula}{{\underline \lambda}}
\newcommand{\umu}{{\underline \mu}}
\newcommand{\uth}{{\underline \Theta}}

\newcommand{\Imm}{{\rm Imm} (\Sigma, M)}
\newcommand{\ImmU}{{\rm Imm}_{\mathcal{U}} (\Sigma, M)}
\newcommand{\wImm}{\widetilde {\rm Imm} (\Sigma, M)}

\newtheorem*{definition*}{Definition}
\newtheorem*{example*}{Example}
\newtheorem*{fact*}{Fact}
\newtheorem*{lemma*}{Lemma}
\newtheorem*{proposition*}{Proposition}
\newtheorem*{remark*}{Remark}

\newcounter{dummy}
\makeatletter
\newcommand\myitem[1][]{\item[#1]\refstepcounter{dummy}\def\@currentlabel{#1}}
\makeatother

\newcommand\rl{\rm{l}}

\newcommand{\cI}{\mathcal{I}}

\newcommand{\cP}{\mathcal{P}}

\newcommand{\cU}{\mathcal{U}}

\newcommand{\fg}{{\mathfrak g}}
\newcommand{\fh}{{\mathfrak h}}

\newcommand{\fv}{{\mathfrak v}}

\newcommand{\su}{\mathfrak{su}}

\newcommand{\gl}{\mathfrak{gl}}

\newcommand{\Sp}{{\rm Spin} (7)}

\newcommand{\GL}{\mathrm{GL}}

\renewcommand{\det}{\mathop\mathrm{det}\nolimits}
\newcommand{\Diff}{\mathrm{Diff}}
\newcommand{\End}{{\mathrm{End}}}

\newcommand{\Hol}{\mathrm{Hol}}

\newcommand{\id}{\mathrm{id}}

\def\cx{\mathbb{C}}

\def\rl{\mathbb{R}}

\def\Z{\mathbb{Z}}

\def\End{\mathrm{End}}
\def\GL{\mathrm{GL}}

\def\SU{\mathrm{SU}}

\def\vol{\mathrm{vol}}
\def\Vol{\mathrm{Vol}}

\def\su{\mathfrak{su}}

\def\<{\mathopen{}\left<}
\def\>{\right>\mathclose{}}
\def\({\mathopen{}\left(}
\def\){\right)\mathclose{}}

\renewcommand{\epsilon}{\varepsilon}

\renewcommand{\i}{{\sqrt{-1}}}

\newtheorem{theorem}{Theorem}[section]

\newtheorem{corollary}[theorem]{Corollary}

\newtheorem{lemma}[theorem]{Lemma}
\newtheorem{proposition}[theorem]{Proposition}
\theoremstyle{definition} \newtheorem{definition}[theorem]{Definition}
\newtheorem{example}[theorem]{Example}

\newtheorem{remark}[theorem]{Remark}

\crefname{theorem}{Theorem}{Theorems}						% label for Theorems
\creflabelformat{theorem}{#2{#1}#3}							% label format for 'theorem'
\crefname{Mtheorem}{Main Theorem}{Main Theorems}			% label for the Main Theorems
\creflabelformat{Mtheorem}{#2{#1}#3}						% label format for 'Mtheorem'
\crefname{lemma}{Lemma}{Lemmata}							% label for Lemmata
\creflabelformat{lemma}{#2{#1}#3}							% label format for 'lemma'
\crefname{corollary}{Corollary}{Corollaries}				% label for Corollaries
\creflabelformat{corollary}{#2{#1}#3}						% label format for 'corollary'
\crefname{proposition}{Proposition}{Propositions}			% label for Propositions
\creflabelformat{proposition}{#2{#1}#3}						% label format for 'proposition'
\crefname{ineq}{inequality}{inequalities}					% label for inequalities
\creflabelformat{ineq}{#2{\upshape(#1)}#3}					% label format for 'ineq'
\crefname{cond}{condition}{conditions}						% label for conditions
\creflabelformat{cond}{#2{\upshape(#1)}#3}					% label format for 'cond'
\crefname{hypoth}{Hypothesis}{Hypotheses}					% label for Hypotheses
\creflabelformat{hypoth}{#2{#1}#3}							% label format for 'hypoth'
\crefname{def}{Definition}{Definitions}						% label for Definitions
\creflabelformat{def}{#2{#1}#3}								% label format for 'def'
\crefname{appsec}{Appendix}{Appendices}

\crefname{sec}{Section}{Sections}

\begin{document}
%\begin{CJK}{UTF8}{min}
%Japanese

%\author{Kotaro Kawai}

\title{Anisotropic calibrations, Fueter maps\\and mirror symmetry}

\author{Kotaro Kawai}
\address{Beijing Institute of Mathematical Sciences and Applications, 
No. 544, Hefangkou Village, Huaibei Town, Huairou District, Beijing, 101408, 
China}
\address{
Department of Mathematics, Osaka Metropolitan University, 3-3-138, Sugimoto, Sumiyoshi-ku, Osaka, 558-8585, Japan}
\email{\href{mailto:kkawai@bimsa.cn}{kkawai@bimsa.cn}}

\author{Tommaso Pacini}
\address{Department of Mathematics, University of Torino, 
via Carlo Alberto 10, 10123 Torino, Italy}
\email{\href{mailto:tommaso.pacini@unito.it}{tommaso.pacini@unito.it}}

\keywords{Calibrated geometry, anisotropic submanifolds, mirror symmetry, gauge theory, $G_2$ geometry, Fueter maps}
\subjclass[2020]{53C38, 53C10, 58E15, 53C42, 53C07}

% 53C38 Calibrations and calibrated geometries
% 53C10 G-structures
% 58E15 Variational problems concerning extremal problems in several variables; Yang-Mills functionals [See also 81T13], etc.
% 53C42 Differential geometry of immersions (minimal, prescribed curvature, tight, etc.) [See also 49Q05, 49Q10, 53A10, 57R40, 57R42]
% 53C07 Special connections and metrics on vector bundles (Hermite-Einstein, Yang-Mills) [See also 32Q20]

\begin{abstract}
Let $(M,g)$ be a Riemannian manifold. Choose a pair $(\alpha,H)$, where $\alpha$ is a calibration and $H$ is a calibrated distribution. Using these data, we define a 1-parameter family of forms $\alpha_\epsilon$ and study its adiabatic limit as $\epsilon\rightarrow 0$. We show that (i) the limit is a calibration in a generalized sense, (ii) under the usual closedness assumptions, the adiabatic calibrated submanifolds are anisotropic minimal in the classical sense defined in the Calculus of Variations/PDE theory.

We apply this construction to $G_2$-manifolds endowed with an associative distribution. Here, one can also define the notion of Fueter maps. We prove that, in the case of isometric immersions, adiabatic calibrated submanifolds coincide with Fueter maps: this is a first-order analogue of the classical relationship between minimal submanifolds and harmonic maps. We provide explicit examples and prove local analytic existence theorems for adiabatic calibrated submanifolds. Applying mirror symmetry as described by the real Fourier--Mukai transform in the standard ``toy model'' situation, the picture is as follows: adiabatic limits correspond to large radius limits, calibrated (associative) submanifolds correspond to deformed Donaldson--Thomas connections, adiabatic calibrated submanifolds correspond to $G_2$-instantons. 
\end{abstract}
\maketitle

\setcounter{tocdepth}{1}
\tableofcontents

%%%%%%%%%%%%%%%%%%%%%%%%%%%%%%%%%%%%%%%%%%%%%%%%%%%%%%%%%%%%%%%%%%%%

%\setcounter{page}{0}\pagenumbering{arabic}

\section{Introduction}\label{s:intro}

One of the main intuitions that string theory has contributed to geometry is the concept of mirror symmetry. According to \cite{SYZ}, \cite{GYZ}, the fundamental structures underlying this theory are (i) manifolds with special holonomy, (ii) calibrated fibrations. 

Generally speaking, the motivation behind this paper is to take a step back with respect to both structures, developing a geometry based on (i') Riemannian manifolds with a calibration, (ii') calibrated distributions. A number of features come to light, which seem to have previously been either unnoticed or studied only in very specific settings.

The paper is split into two parts: Sections \ref{sec:weighted_cal}-\ref{sec:mirror general} cover the general theory, Sections \ref{sec:G2algebra}-\ref{sec:G2 CK} concern its application to a specific class of manifolds related to special holonomy: $G_2$-manifolds.

\subsubsection*{Note:}Sections that can be omitted upon first reading are marked by ($\star$).

\subsection*{Summary of Part 1.}It may be useful to briefly review two keywords.

\subsubsection*{Adiabatic limit.}
In differential geometry the term ``adiabatic limit''
goes back at least to the 1980s, in work by Witten, Bismut-Cheeger, Bismut-Freed, Mazzeo-Melrose et al. The underlying idea is to study objects depending on a parameter $\epsilon$ which, in different directions, scale at different speeds. This construction is generally implemented in the context of fiber bundles $M\rightarrow B$, with the idea of distinguishing horizontal and vertical directions and, intuitively, of collapsing the vertical fibers. 
In certain situations the ``formal limit'', defined by letting $\epsilon\rightarrow 0$, is a new, interesting object in its own right, defined on the base space $B$.

\subsubsection*{Calibration.}Let $(M,g)$ be a Riemannian manifold. Assume we want to prove that a given submanifold minimizes the volume functional. A classical technique is to look for an appropriate ``calibration'', i.e. a differential form $\alpha$ satisfying certain conditions. The role of $\alpha$ is to select, within the Grassmannian of planes, a certain subset of ``calibrated'' planes: if the tangent spaces of the given submanifold are calibrated, the properties of $\alpha$ will imply that the submanifold is a minimizer. Calibrated geometry, as theorized in \cite{HL}, concerns the reverse problem: given a calibration, find/study its calibrated submanifolds. Any such submanifold is minimizing, thus ``minimal'', i.e. it is a critical point of the volume functional.\\

The starting point for our paper is the following, very simple, idea. Let us choose a Riemannian manifold $(M,g)$ and (i) a calibration $\alpha$, (ii) a distribution $H$ whose planes are calibrated. Let $V$ denote the orthogonal distribution. We may then apply the construction of adiabatic limits to the object $\alpha$: inserting the parameter will scale the two directions $H$, $V$ differently, producing a form $\alpha_\epsilon$. We shall study its geometry for $\epsilon\rightarrow 0$. 

As usual in calibrated geometry, the main geometric objects of interest to us are the immersed submanifolds. Let us introduce another keyword.

\subsubsection*{Anisotropic minimal submanifold.}
The standard volume functional can be generalized via a choice of weighting function that assigns different weights to different directions in $M$. In this (also very classical) context, critical points are known as ``anisotropic minimal'' submanifolds.\\

We can now state the main goals of Part 1.\\

(i) Introduce the notion of anisotropic calibrations, parallel to the classical notion but whose corresponding submanifolds minimize a generalized volume functional, and are therefore anisotropic minimal (rather than minimal). This is discussed in Section \ref{sec:weighted_cal}. See, in particular, Proposition \ref{prop:minimize Volf} and Remark \ref{rem:anisotropic minimal}.

(ii) Show that, in the very general setting $(M,g,\alpha,H)$ discussed above, a certain adiabatic limit of the forms $\alpha_\epsilon$ generated by $\alpha$ is an anisotropic calibration on $M$. This is contained in Section \ref{sec:relative_distr}. See, in particular, Theorem \ref{thm:from cali to rel cali}.

(iii) Study the geometric features of the adiabatic calibrated submanifolds, both from the variational and the mirror-symmetric viewpoint. See, in particular, Corollaries \ref{cor:rel cali minimize VE} and \ref{cor: minimize VE+VolH} and Section \ref{sec:mirror general}.

\ 

An obvious source of examples of distributions $H$ is the case where $M$ has the structure of a fibration $M\rightarrow B$. It is then very natural to let $V$ be the (integrable) distribution defined by the fibers and $H$ be the orthogonal distribution. It is also natural to concentrate on the special class of immersions defined as sections $B\rightarrow M$. 

In general, however, neither of our distributions will be integrable. This implies that there does not exist a ``base manifold'' $B$ defined by projection along fibers. Correspondingly, there does not exist a notion of sections.

This has two consequences: (i) Contrary to other studies of adiabatic limits, all limiting objects must be defined within $M$ itself. (ii) The notion of sections must be replaced by the notion of ``projectable'' immersions $\Sigma\rightarrow M$, defined with respect to the distribution $H$. The abstract manifold $\Sigma$ fixes the topological type of these submanifolds.
The intermediate Sections \ref{sec:Preliminaries} and \ref{sec:functionals} discuss distributions, projectable immersions and the generalized volume functionals of interest in our paper.

\subsection*{Summary of Part 2.}Here, we apply the general theory of Part 1 to the class of ``$G_2$-manifolds''. These manifolds are automatically endowed with a metric and a distinguished calibration, usually denoted $\varphi$. They are of special interest within the theory of special holonomy. We need one more keyword.

\subsubsection*{Fueter equation.}Recall the Cauchy--Riemann equation $f_x+if_y=0$ for functions $f:\mathbb{C}\rightarrow\mathbb{C}$. Strictly speaking, the Fueter equation $f_x+if_y+jf_z+kf_w=0$ for functions $f:\mathbb{H}\rightarrow \mathbb{H}$ is its quaternionic analogue. It can be generalized in many ways, e.g. \cite{Haydys}, \cite{Taubes}, \cite{Walp}. This serves also to emphasize the relation to the Dirac operator and hence to the Laplacian. 

The corresponding notion of Fueter maps plays a prominent role in the geometry of manifolds with special holonomy, e.g. in the attempt to define invariants for $G_2$-manifolds 
\cite{DS}, \cite{DT}, \cite{Haydys2}, \cite{HaydysG2}, \cite{WalpG2}, \cite{Walp}.\\

The main goals of Part 2 are the following.\\

(iv) Apply the adiabatic construction to the calibration $\varphi$ defined on a $G_2$-manifold. This requires a preliminary choice of distribution. According to the above, the calibrated submanifolds defined by the adiabatic limit will be anisotropic minimal: see Corollaries \ref{cor:fueter minimize VE} and \ref{cor:fueter minimize EV}. They also have several other interesting properties, listed in Section \ref{sec:G2Fueter}. 

(v) Use the $G_2$-structure to reformulate the anisotropic calibrating condition in terms of a Fueter-type differential operator; see Section \ref{s:Fueter maps}. As a result it is shown that anisotropic calibrated submanifolds are the geometric counterpart of Fueter maps in the same way that minimal submanifolds are related to harmonic maps: see Theorems \ref{prop:adiabatic vs F} and \ref{prop:minimal vs harmonic}.

(vi) Present both explicit examples of adiabatic calibrated submanifolds (cf. Section \ref{sec:G2 case examples})
and abstract local existence results (cf. Section \ref{sec:G2 CK}).

(vii) Show that these submanifolds are related, via mirror symmetry, to known gauge-theoretic objects. The results are summarized in Diagram \ref{eq:asso dDT Fueter G2inst}, which we reproduce here:

\begin{align*}
\begin{split}
\xymatrix{
\{ \mbox{associative submanifolds} \} \ar@{<->}[r]_-{mirror} 
\ar@{-->}[d]_-{adiabatic \ limit}
& \{ \mbox{dDT connections} \} 
\ar@{-->}[d]^-{large \ radius \ limit} \\
    \{ \mbox{adiabatic calibrated submanifolds} \} \ar@{<->}[r]_-{mirror}  
    & \{ \mbox{$G_2$-instantons} \}
}
\end{split}
\end{align*}
Understanding the full set of correspondences described in this diagram was actually one of the initial motivations for this paper.

\subsection*{Additional comments.}The above summary of our results glosses over an important issue: whether the initial calibration $\alpha$ is a closed form, and whether this holds also for its adiabatic limit. Interest in the $G_2$ class of examples is in part based on the fact that (i) the condition $d\varphi=0$ is geometrically extremely natural, (ii) this condition also allows us to bypass the analogous condition on its adiabatic limit, still obtaining similar results: compare the assumptions in Corollaries \ref{cor:fueter minimize VE}, \ref{cor:fueter minimize EV}.

\subsection*{Comparisons with the literature.}We are not aware of results in the literature similar to Part 1. Regarding Part 2, the idea of studying adiabatic limits in the context of manifolds with special holonomy is certainly not new: in particular, there is some overlap with work by Donaldson \cite{Donaldson} and by Li \cite{Li}. Their studies assume, however, the existence of a very strong fibration structure: $M\rightarrow B$ is a Riemannian submersion, and the fibers are K3 surfaces. The following remarks might provide some guidance in comparing our results with these, see also Example \ref{ex:Fueter_vs_Walpuski}.\\

1. One advantage of our more general framework is that, by simplifying the background geometry, it sheds light on some very general properties of the adiabatic process that, up to now, seem to have escaped attention: in particular, the fact that adiabatic calibrated submanifolds are anisotropic minimal.

2. \cite{Li} uses a different concept of mirror symmetry, based on Mukai K3 duality. This sometimes leads to results different from ours.

3. Our generalized volume functionals, applied to an immersion $\Sigma\rightarrow M$, require a metric on the image submanifolds. We obtain this from the metric on $M$ using the projectability assumption, rather than via the usual restriction process. On the other hand, in the special case where $M\rightarrow B$ is a Riemannian submersion, $\Sigma=B$ is endowed with a canonical metric. If we restrict our attention to immersions determined by sections $B\rightarrow M$, as in \cite{Li}, this metric coincides with the pullback of our metric. It follows that our results in Section \ref{sec:G2Fueter} include and extend those of \cite{Li}, Section 2.3.

4. Sections have a very special feature: they cannot be reparametrized. In the more general setting of manifolds with a distribution and immersions $\Sigma\rightarrow M$, our results would not be possible if we fixed a metric a priori on $\Sigma$ because the functional would then depend on the specific parametrization, leading to a theory similar to that of the energy functional and harmonic maps. The volume functional, calibrated geometry and all functionals in our paper are instead parametrization-independent. We will return to this point in Section \ref{ss:comments}. It is also relevant to Theorem \ref{prop:adiabatic vs F}.

5. Since our general setting does not include a base manifold $B$, the adiabatic construction takes on a different flavour compared, for example, to \cite{Donaldson}. This is true in at least two ways. 

(i) Since both the initial and the limiting objects belong to the same ambient $M$, the limits can be defined in the usual, topological, sense: there is no need to invoke any idea of  ``formal limit''. As often in the literature, however, it remains true that we make no claims concerning how solutions of an original equation may persist and converge to solutions of the limiting equation, as $\epsilon\rightarrow 0$.

(ii) We will find ourselves applying the adiabatic construction to several objects (tensors or equations) simultaneously; the main role of this construction is then to highlight which parts of each object scale at the same speed, suggesting relationships between them. Corollary \ref{cor:weighted eq prop dist}, showing the relationship between objects $\alpha_2$ and $\chi_1$, is a concrete instance of this.

\subsection*{Outlook.}The first notion discussed in this paper is that of anisotropic calibrations. As explained above, we have chosen to focus on its relationship to the theory of pairs $(\alpha,H)$ but its scope is wider: Sections \ref{sec:ex Herm}, \ref{sec:ex Sasaki} provide examples of how it relates to other research in the literature. It is natural to expect that it has a useful role to play within the theory of anisotropic minimal submanifolds.

The interplay between adiabatic calibrated submanifolds and Fueter maps also seems worthy of further attention. As mentioned, Fueter maps play an important role in the geometry of manifolds with special holonomy, and have thus been intensively studied. In \cite{Li} they appear as sections of a K3 fibration, but this is a very rigid situation. Our extension to the more generic setting of non-integrable distributions may be useful for the corresponding moduli space theory: this might be compared to the way that the theory of $J$-holomorphic maps, defined using non-integrable complex structures, contributed useful moduli spaces in symplectic topology. 

The class of $G_2$-manifolds is just one in Berger's list of manifolds with special holonomy \cite[Section 3]{Joyce}. In work in progress, the authors have applied the general theory of Part 1 to several other classes in this list. In each case, the corresponding adiabatic calibrated submanifolds have properties similar to those seen in Section \ref{sec:G2Fueter}, confirming that the general theory has a very broad scope. These geometries also admit analogues of Diagram \ref{eq:asso dDT Fueter G2inst}. These results suggest that adiabatic calibrated submanifolds might admit a theory parallel to that of the classical calibrated submanifolds, and might have an interesting role to play in mirror symmetry.

\subsection*{Acknowledgements.} KK is supported by 
JSPS KAKENHI Grant Number JP21K03231, 
by NSFC No. 12371048 and by 
International Scientists Project BJNSF \#IS25022. Both authors are happy to thank Nesin Mathematical Village, Turkey, and the organizer Mustafa Kalafat for providing hospitality and a beautiful context for the initial conversations in 2024 that eventually led to this project. They also thank A. De Rosa, J. Lotay and T. Walpuski for comments on the first draft of this paper, and A. Haydys for a very useful conversation.

The authors used ChatGPT to investigate certain ideas and examples in preliminary stages of this project. All mathematical content was verified by the authors, who take full responsibility for the manuscript.

\part{General theory.}\label{part1}

\section{Anisotropic calibrations}\label{sec:weighted_cal}

Let $M$ be a smooth manifold. Recall that a \textbf{Grassmannian geometry} on $M$ is determined by a subset $\mathcal{G}$ of the oriented Grassmannian of $k$-planes $G(k,M)$. Let $\Sigma$ be an abstract oriented $k$-manifold. We say that an immersion $\iota:\Sigma\rightarrow M$ is a \textbf{$\mathcal{G}$-submanifold} if, using the orientation on $\Sigma$ and for each $\sigma\in\Sigma$, $\iota_{*}(T_\sigma\Sigma)\in\mathcal{G}$.

An important source of Grassmannian geometries arises as follows.
 Let $(M,g)$ be a Riemannian manifold. Given any $\pi\in G(k,M)$, let $\vol_{\pi}$ denote the volume form on $\pi$ induced by the orientation and by $g$. We will say that a $k$-form $\alpha$ on $M$ is a \textbf{calibration} if it has the following properties: (i) for each $\pi\in G(k,M)$, $\alpha|_\pi\leq \vol_{\pi}$, (ii) $d\alpha=0$.

 Condition (i) means the following: if we write $\alpha|_\pi=\lambda\vol_\pi$, then $\lambda\leq 1$. This can be equivalently written as follows: for each $\pi\in G(k,M)$ and each positive orthonormal $k$-frame $e_1,\dots,e_k$ of $\pi$, $\alpha(e_1,\dots,e_k)\leq 1$.
 
\begin{remark}\label{rem:opposite or}
Given any $\pi\in G(k,M)$, let $-\pi$ denote the same $k$-plane endowed with the opposite orientation. 
Since condition (i) must apply also to $-\pi$, it can be rephrased as follows: for each non-oriented $k$-plane $\pi$ and orthonormal $k$-frame $e_1,\dots,e_k$ of $\pi$, $\alpha(e_1,\dots,e_k)\leq 1$. Changing the sign of one of the vectors leads to $|\alpha(e_1,\dots,e_k)|\leq 1$.

This formulation downplays the role of orientations, emphasizing that $\alpha$ is a calibration if and only if $-\alpha$ is a calibration.
\end{remark}

 We obtain a Grassmannian geometry by setting 
 $$\mathcal{G}(\alpha):=\{\pi\in G(k,M):\alpha|_\pi=\vol_\pi\}.$$ 
 The elements of $\mathcal{G}(\alpha)$ are known as \textbf{calibrated planes}. The corresponding $\mathcal{G}(\alpha)$-submanifolds are known as \textbf{calibrated submanifolds}. In this case $\iota^*\alpha$ coincides with the volume form induced by the given orientation on $\Sigma$ and by the metric $\iota^*g$.

 This class of Grassmannian geometries has the following additional, important, feature. Recall that, if $\Sigma$ is compact and oriented, any immersion defines a homology class in $M$. Let 
 $$\Imm=\{\iota:\Sigma\rightarrow M\}$$ 
 denote the set of immersions. The following, classical, result is the starting point of calibrated geometry.
 \begin{proposition} \label{prop:cali hom vol min}
If $\Sigma$ is compact, any calibrated immersion $\Sigma\rightarrow M$ minimizes the volume functional
$$\Vol:\Imm\rightarrow\mathbb{R}, \ \ \iota\mapsto\int_\Sigma\vol_{\iota^*g}$$ 
within its homology class. The minimum value of $\Vol$ is topologically determined.
\end{proposition}
\begin{proof}Given any immersion $\iota$, integrating condition (i) shows that $\int_\Sigma\iota^*\alpha\leq \Vol(\iota)$. This provides a lower bound for the volume functional. The lower bound is achieved if and only if $\iota$ is calibrated. Condition (ii) implies that this lower bound depends only on the homology class defined by $\iota$. 
\end{proof}

\begin{remark}
As already mentioned, $\alpha$ is a calibration if and only if $-\alpha$ is a calibration. The Grassmannian subsets $\mathcal{G}(\alpha)$, $\mathcal{G}(-\alpha)$ correspond to different choices of orientation on the same subset of non-oriented $k$-planes. In other words, $\alpha$-calibrated submanifolds are necessarily orientable and are canonically oriented by the form $\iota^*\alpha$. Choosing $-\alpha$ will define the same set of calibrated submanifolds, endowed with the opposite orientation.
\end{remark}

\begin{definition}\label{def:equality_property}
We will say that a calibration $\alpha$ satisfies the \textbf{equality property} if there exists a Riemannian vector bundle $(E,g_E)$ over $M$ and an $E$-valued $k$-form $\chi$ such that, for each $\pi\in G(k,M)$ and frame $v_1,\dots,v_k$ for $\pi$,
\begin{equation}\label{eq:eq prop}
|\alpha(v_1,\dots,v_k)|^2+|\chi(v_1,\dots,v_k)|^2=|v_1\wedge\dots\wedge v_k|^2.
\end{equation}
\end{definition}
Equation \eqref{eq:eq prop} also appears in \cite{LV}, under the name ``Harvey-Lawson's identity''. Interest in this property stems from the following fact: $\pi\in \mathcal{G}(\alpha)\cup\mathcal{G}(-\alpha)$ if and only if $\chi|_\pi=0$. This provides the following alternative description of calibrated submanifolds: an immersion $\iota:\Sigma\rightarrow M$ is calibrated (with respect to some orientation on $\Sigma$) if and only if $\iota^*\chi=0$. This formulation is ideal for proving local existence results for calibrated submanifolds via Cartan-K\"ahler theory, as in Section \ref{sec:G2 CK}.

\begin{remark}
The above is classical \cite{HL}. We refer to \cite{Morgan} 
for a history of calibrated geometry and a bibliography. An instance of the following generalization appears in \cite{Morgan1991}.
\end{remark}

We can generalize this construction as follows. Let $(M,g)$ be a Riemannian manifold. Choose a function
$f:\mathcal{U}\subseteq G(k,M)\rightarrow\mathbb{R}$. We will generally assume that $\mathcal{U}$ is open and connected, that it projects surjectively onto $M$, and that $f\geq 0$. 
We will say that a $k$-form $\alpha$ on $M$ is a \textbf{$f$-anisotropic calibration} if it has the following properties: (i) for each $\pi\in \mathcal{U}$, $\alpha|_\pi\leq f(\pi) \vol_{\pi}$, (ii) $d\alpha=0$. As before, we obtain a Grassmannian geometry, i.e. the \textbf{$f$-calibrated planes}, by setting 
$$\mathcal{G}(\alpha,f):=\{\pi\in \mathcal{U}:\alpha|_\pi=f(\pi)\vol_\pi\}.$$ 
The corresponding $\mathcal{G}(\alpha,f)$-submanifolds will be called \textbf{$f$-calibrated submanifolds}.

In this context we will be interested in the class of immersions 
$$\ImmU=\{\iota:\Sigma\rightarrow M:\iota_*(T\Sigma)\subset\mathcal{U}\}.$$
This definition requires $\Sigma$ to be oriented. If $\Sigma$ is also compact, we can define the functional
$$\Vol_f:\ImmU\rightarrow\rl, \ \ \iota\mapsto \int_\Sigma f(\iota_*(T\Sigma)) \vol_{\iota^*g}.$$

\begin{remark}\label{rem:extra hyps} Our assumptions on $\mathcal{U}$ and $f$ can be motivated as follows. 

(i) The assumption $\mathcal{U}$ open, when $\Sigma$ is compact, implies that $\ImmU$ is $C^1$-open in the space of all immersions.

(ii) If $\mathcal{U}$ has the property that $\pi\in\mathcal{U}$ if and only if $-\pi\in\mathcal{U}$, then the argument in Remark \ref{rem:opposite or} implies that we should at least assume $f\geq 0$. Notice: if, furthermore, $f>0$ then any $f$-calibrated submanifold is automatically orientable, and oriented by $\iota^*\alpha$.

(iii) Notice that, for $k<\mbox{dim}(M)$, $G(k,M)$ is connected and that it projects surjectively onto $M$ (i.e. it is non-empty above each point $x \in M$).

Our assumptions thus maximize the similarities with standard calibrated geometry.

Notice that, as an extreme example, our assumptions allow $\mathcal{U}=G(k,M)$ and $f\equiv 0$. In this case the condition $\alpha|_\pi\leq f(\pi)\,\vol_\pi$ implies $\alpha\equiv 0$, so $\mathcal{G}(\alpha, f)=G(k,M)$. Any submanifold is thus $f$-calibrated. This case is, of course, not very interesting.
\end{remark}

\begin{definition}\label{def:K happy}
We will say that an immersion $\iota\in \ImmU$ minimizes $\Vol_f$ \textbf{in its restricted homology class} if $\Vol_f(\iota)\leq \Vol_f(\iota')$, for each $\iota'\in \ImmU$ such that $\iota$ and $\iota'$ are homologous.
\end{definition}

 \begin{proposition} \label{prop:minimize Volf}
 Assume $\Sigma$ is compact. Then any $f$-calibrated immersion $\iota\in \ImmU$ minimizes the functional $\Vol_f$ in its restricted homology class. In particular, it is a $C^1$-local minimum for the functional. The minimum value $\Vol_f(\iota)$ is topologically determined.
\end{proposition}

\begin{remark}\label{rem:anisotropic minimal}
If $\mathcal{U}=G(k,M)$ then Proposition \ref{prop:minimize Volf} is exactly analogous to Proposition \ref{prop:cali hom vol min}. However, in all examples in this paper, $\mathcal{U}$ will be a proper subset. 

If $\Sigma$ is not compact, one can use perturbations with compact support to show that calibrated immersions are minimal; likewise, $f$-calibrated immersions satisfy the Euler-Lagrange equation for $\Vol_f$ so they are anisotropic minimal in the sense discussed, for example, in the survey \cite{DeRosa}. 

Notice that the general regularity theory for anisotropic minimal submanifolds actually requires $f>0$. Our focus is instead on specific examples, arising in specific geometric settings. In these examples we will see that $f$ can also vanish. This implies that one cannot rely on the known general regularity theorems, but it may be that the extra geometric features allow for ad hoc results in this direction.
\end{remark}

\begin{definition}
We will say that a $f$-anisotropic calibration $\alpha \in \Omega^k (M)$ satisfies the 
{\bf $f$-equality property} if 
there exists a Riemannian vector bundle $(E,g_E)$ over $M$ and 
an $E$-valued $k$-form $\chi \in \Omega^k(M,E)$ such that, for each $\pi\in\mathcal{U}$ and frame $v_1, \dots, v_k$ for $\pi$,
\begin{align}
|\alpha (v_1, \dots, v_k)|^2+|\chi (v_1, \dots, v_k)|^2
=f(\pi)^2 |v_1 \wedge \cdots \wedge v_k|^2.
\end{align}
\end{definition}
As before, an immersion is $f$-calibrated (with respect to some orientation on $\Sigma$) if and only if $\iota^*\chi=0$. 

\ 

In what follows, some results (related to the more linear-algebraic aspects of calibrations) will hold also for a weaker category of $k$-forms, as follows.

\begin{definition}\label{def:semi}
We will say that a $k$-form $\alpha$ is a ($f$-anisotropic) \textbf{semi-calibration} if it only satisfies condition (i): $\alpha|_\pi\leq \vol_\pi$ (or: $\alpha|_\pi\leq f(\pi)\vol_\pi$).
\end{definition}

In this situation one can still define calibrated submanifolds, but they will not have minimization properties.

\begin{remark}
Any manifold has infinitely many closed (possibly exact) $k$-forms. Any $C^0$-small closed form will satisfy the semi-calibration condition $\alpha|_\pi\leq\vol_\pi$; any closed form will satisfy the $f$-anisotropic semi-calibration condition, for some $f$. In this sense, calibrations and $f$-anisotropic calibrations exist in abundance. Whether or not a specific such form is of geometric interest will depend on additional factors. We will meet several such factors in the following sections: whether $f$ has geometric meaning, the structure of $\mathcal{G}(\alpha,f)$, whether the assumptions of Cartan-K\"ahler theory hold, etc.
\end{remark}

\subsection{($\star$) Example: Hermitian geometry} \label{sec:ex Herm}
Let $(M,J,g,\omega)$ be a Hermitian manifold of complex dimension $n$. Let $h=g-\i\omega$
denote the corresponding Hermitian metric. Let $\mathcal{U}\subseteq G(n,M)$ denote the subset of oriented totally real $n$-planes $\pi$, i.e. such that $J(\pi)\cap \pi=\{0\}$. It follows that, at any point, $\pi\oplus J(\pi)= T_xM$.

Given any $\pi\in\mathcal{U}$, choose a positive basis $v_1,\dots,v_n$. Let $v^1,\dots,v^n$ denote the complex-linear extension of its dual basis: $v^i(X_1+JX_2)=v^i(X_1)+\i v^i(X_2)$.
Let us define the \textbf{$J$-volume form} on $\pi$ as follows: 
$$\vol^J_\pi=\frac{v^1\wedge\dots\wedge v^n}{|v^1\wedge\dots\wedge v^n|_h}$$
This construction is independent of the choice of basis. By definition, each $\vol^J_\pi$ is a complex $(n,0)$-form on $T_xM$. It has unit length and it restricts to a real $n$-form on $\pi$. We can thus compare it to the standard volume form $\vol_\pi$.

\begin{lemma}\label{lem:TR}
Let $\pi$ be an oriented totally real $n$-plane and $e_1,\dots,e_n$ be a positive orthonormal basis. Then
$$\vol^J_\pi(e_1,\dots,e_n)=\sqrt{\det_\cx \left(h(e_i,e_j)\right)}=\sqrt{\vol_g(e_1,\dots,e_n,Je_1,\dots,Je_n)}.$$
It follows that the restriction of $\vol^J_\pi$ to $\pi$ satisfies $\vol^J_\pi=f(\pi)\vol_\pi$, where 
$$f:\mathcal{U}\rightarrow\rl, \qquad f(\pi)=\sqrt{\det_\cx \left(h(e_i,e_j)\right)}.$$
\end{lemma}

One can check that the pair $(\mathcal U,f)$ satisfies the assumptions required to define a $f$-anisotropic calibration. However, we still need to find a calibrating form.

Assume the canonical bundle $K_M$ is topologically trivial, so that one can choose a unit section $\Omega$: by definition, $\Omega$ is a complex $(n,0)$-form on $M$ and, for any $\pi$, $\Omega(\pi)=e^{\i\theta(\pi)}\vol^J_\pi$, for some $\theta(\pi)$. We can decompose $\Omega=\alpha+\i\beta$. Given $e_1,\dots,e_n$ as above,
$$|\alpha(e_1,\dots,e_n)|^2+|\beta(e_1,\dots,e_n)|^2=|\Omega(e_1,\dots,e_n)|^2=|\vol^J_\pi(e_1,\dots,e_n)|^2=f(\pi)^2$$
It follows that $\alpha|_\pi\leq f(\pi)\vol_\pi$, so $\alpha$ is a $f$-anisotropic semi-calibration which satisfies the $f$-equality property with respect to the $\rl$-valued $n$-form $\beta$. The $f$-calibrated planes, known as \textbf{special totally real}, are the oriented totally real $n$-planes $\pi$ such that $\beta(\pi)=0$.

To conclude, assume in addition that the chosen unit section $\Omega$ is holomorphic. Then $d\Omega=0$, so also $d\alpha=0$. It follows that $\alpha$ is a $f$-anisotropic calibration. Using Proposition \ref{prop:minimize Volf} and Remark \ref{rem:anisotropic minimal}, we may summarize as follows.

\begin{corollary}
Let $(M,J,g,\omega)$ be a Hermitian manifold and let $\Omega$ be a unit holomorphic volume form. Let $f$ be as in Lemma \ref{lem:TR}. Then any compact, oriented, special totally real submanifold is $f$-anisotropic minimal.
\end{corollary}

\begin{remark} The extra degree of freedom corresponding to the rotation factor implies that any oriented totally real $n$-plane $\pi$ is special for some unit holomorphic form. Fixing $\Omega$ a priori defines instead a subset $\mathcal{G}(\alpha,f)$ of calibrated totally real planes.

If an $n$-plane $\pi$ is Lagrangian, i.e. $\omega|_\pi=0$, it is totally real and $f(\pi)=1$. The notion of special totally real planes thus extends the more classical notion of special Lagrangian planes \cite{HL}, which appears in the theory of Calabi--Yau manifolds.

For further information on totally real and special totally real geometry see \cites{Borrelli, LPc, LP}, where 
$\Vol_f(\iota)=\int_\Sigma \vol^J_{\iota_*(T\Sigma)}$ 
is called the $J$-volume functional.
\end{remark}

\subsection{($\star$) Example: Sasakian geometry} \label{sec:ex Sasaki}
The odd-dimensional analogue of Section \ref{sec:ex Herm} is as follows. 
Let $(M^{2n+1},g,\eta,\xi,\phi)$ be a Sasakian manifold.
For $x\in M$, 
an oriented $n$-plane $\pi\subseteq T_xM$ is called {\bf transversely totally real}
if
\begin{equation}\label{eq:affine-Legendrian-plane}
T_xM = \pi \oplus \phi(\pi) \oplus \mathbb{R}\xi_x
\end{equation}
as a direct sum. 
(This is called affine Legendrian in \cite{Ksta}.)
Let $\mathcal{U}\subseteq G(n,M)$ denote the open subset of oriented 
transversely totally real $n$-planes.

Define $f: \cU \to \rl$ by 
\begin{equation*}
f(\pi)\ 
:=\ \sqrt{\mathrm{vol}_g(e_1,\dots,e_n,-\xi_x,\phi e_1,\dots,\phi e_n)}, 
\end{equation*}
where $\{ e_1,\dots, e_n \}$ is 
an orthonormal basis of $(\pi, g|_\pi)$. 
We have 
$0<f(\pi)\le 1$, with $f(\pi)=1$ if and only if $\eta|_\pi=0$ (i.e., $\pi$ is Legendrian);
see \cite[Lemma 3.14]{Ksta} for these properties.
Define the \textbf{$\phi$-volume form} on $\pi$ by
\[
\vol^\phi_\pi \ :=\ f(\pi)\,\vol_\pi. 
\]

Assume in addition that the Riemannian cone 
$(C(M), \bar g)=(\rl_{>0} \times M, dr^2+r^2g)$ is Calabi--Yau,
and fix a unit holomorphic $(n+1,0)$-form $\Omega$ on $C(M)$. 
Define a complex valued $n$-form $\psi$ on $M$
by the splitting
\begin{equation}\label{eq:cone-splitting-phi}
\Omega=(dr-\i\,\eta)\wedge r^n\psi
\end{equation}
(cf.\ \cite[(20)]{Ksta}).
Then for $\pi\in\mathcal{U}$ and a positive orthonormal basis 
$\{ e_1,\dots,e_n \}$ of $(\pi, g|_\pi)$, we have 
$$
\psi|_\pi=e^{\i\theta (\pi)}\vol^\phi_\pi
$$
for some $\theta (\pi)$ 
(see \cite[\S6.2.2]{Ksta}). 

Write $\psi=\alpha+\i\beta$ with $\alpha,\beta$ real $n$-forms on $M$. 
Given $\{ e_1,\dots,e_n \}$ as above,
\[
|\alpha(e_1,\dots,e_n)|^2+|\beta(e_1,\dots,e_n)|^2
=|\psi(e_1,\dots,e_n)|^2
=f(\pi)^2.
\]
Hence $\alpha|_\pi\le f(\pi)\,\vol_\pi$, so $\alpha$ is an $f$-anisotropic semi-calibration on $\mathcal{U}$, and it satisfies the $f$-equality property with respect to $\beta$.
The $f$-calibrated planes are precisely those $\pi\in\mathcal{U}$ such that $\beta|_\pi=0$.
We call them \textbf{special transversely totally real} planes.
(These planes are called special affine Legendrian in \cite[Definition 6.13]{Ksta}.)

\begin{remark}
In general $d\alpha\neq 0$, hence $\alpha$ defines only an $f$-anisotropic semi-calibration.
Consequently, $f$-calibrated submanifolds need not minimize the $\phi$-volume functional
$\Vol^\phi(\iota):=\int_\Sigma f\bigl(\iota_*(T\Sigma)\bigr)\,\vol_{\iota^*g}$ 
in their restricted homology class.
\end{remark}

\section{Immersions defined via a distribution} \label{sec:Preliminaries}

The examples of $f$-anisotropic calibrations of main interest to us arise in a very specific setting, related to a choice of distribution $H$ on a Riemannian manifold $(M,g)$. The goal of this section is to start constructing these examples by defining the underlying subset $\mathcal{U}\subseteq G(k,M)$ and the corresponding class $\ImmU$ of immersions. We will define the functions $f$ and the corresponding functionals in Section \ref{sec:functionals}.

\ 

Let $(M,g)$ be a $(q+r)$-dimensional Riemannian manifold 
with a $q$-dimensional distribution $H$. Denoting by $V$ the orthogonal complement of $H$, we have 
$$
TM = H \oplus V.
$$ 
Denote the corresponding projections by 
$p_H:TM \to H$, $p_V:TM \to V$.

At each point $x \in M$ we can think of $H_x$ as a distinguished element in the Grassmannian of non-oriented $q$-planes.

\begin{definition}\label{def:HPsubsp}
Let $\pi$ be an $s$-dimensional subspace of $T_x M$ for $x \in M$, 
where $s \leq q$. 
We will call $\pi$ {\bf horizontally projectable} if 
$$
p_H|_\pi:\pi \to H_x
$$
is injective. We will let $G_H(s,M)\subseteq G(s,M)$ denote the subset of oriented, horizontally projectable, $s$-planes. 
\end{definition}

Notice that $\pi\in G_H(s,M)$ if and only if $-\pi\in G_H(s,M)$ and that $G_H(s,M)$ retracts via $p_H$ onto the Grassmannian of oriented $s$-planes in $H$. We must thus distinguish two cases:
\begin{itemize}
\item If $s<q$ then the Grassmannian of oriented $s$-planes in $H$ is connected, so $G_H(s,M)$ is also connected.
\item If $s=q$ then the Grassmannian of oriented $s$-planes in $H$ coincides with the set of orientations on $H$, so it can be viewed as a $2:1$ bundle over $H$. If $H$ is orientable this implies that $G_H(q,M)$ is disconnected. 
\end{itemize}
We thus introduce the following additional notion.
\begin{definition}\label{def:pHPsubsp}
Assume $H$ is orientable, with a fixed orientation. Let $\pi\in G(q,M)$ be a $q$-dimensional subspace of $T_x M$ for $x \in M$. We will call $\pi$ {\bf positive horizontally projectable} if 
$$
p_H|_\pi:\pi \to H_x
$$
is an orientation-preserving isomorphism. We will let $G_H^+(q,M)\subseteq G(q,M)$ denote the subset of oriented, positive horizontally projectable, $q$-planes. The set $G_H^+(q,M)$ is connected. 
\end{definition}
We can transfer these notions to submanifolds as follows.

\begin{definition} \label{def:pos hori proj}
Let $\Sigma$ be an oriented $s$-dimensional manifold. An immersion $\iota: \Sigma \hookrightarrow M$ 
is said to be \textbf{horizontally projectable} if $\iota_*(T\Sigma)\subseteq G_H(s,M)$. In particular, we will say that $\iota$ is \textbf{horizontal} if 
$\iota_*(T\Sigma)\subseteq H$,
i.e. if the image submanifold is tangent to $H$.

If $s=q$ and $H$ is orientable with a fixed orientation, we will say that $\iota$ is \textbf{positive horizontally projectable} if $\iota_*(T\Sigma)\subseteq G_H^+(q,M)$.

For simplicity, we will set 
$$
\iota_*^H = p_H \circ \iota_*:T\Sigma \to \iota^* H, \qquad
\iota_*^V = p_V \circ \iota_*:T\Sigma \to \iota^* V.
$$

We will also set 
$$
\HPI=\{ \mbox{horizontally projectable immersions} \}=\ImmU, 
$$
where $\mathcal{U}=G_H(s,M)$. Alternatively, with the above assumptions,
$$
\HPI^+=\{ \mbox{positive horizontally projectable immersions} \}=\ImmU, 
$$
where $\mathcal{U}=G_H^+(q,M)$. 
\end{definition}

\begin{remark}\label{rem:topology}
The existence of a horizontal immersion $\iota$ should be seen as a strong assumption: it implies the integrability of $H$ along the image submanifold. 

When $s=q$ and $\iota$ is horizontally projectable, $\iota_*^H$ is an isomorphism. The existence of such immersions thus also has topological consequences: it implies that $\iota^*H$ is isomorphic to $T\Sigma$. Notice that horizontally projectable immersions can alternatively be defined in terms of the transversality of $\iota_*(T\Sigma)$ with respect to $V$.
\end{remark}

In the following sections it will be useful to use an alternative description of (non-oriented) horizontally projectable $q$-planes, as follows.

\begin{proposition} \label{lem:unori hp hom}
At any point $x\in M$, 
the subset of (non-oriented) horizontally projectable $q$-planes 
is diffeomorphic to the space $H^*_x\otimes V_x$
via the map 
\begin{equation*}
 \pi \mapsto T_\pi:=(p_V|_\pi) \circ (p_H|_\pi)^{-1}. 
\end{equation*}
Using the metric, $H^*_x\otimes V_x \cong H^*_x\otimes V^*_x$ so
we can alternatively identify $\pi$ with a 2-form 
$\beta_\pi:=g \left((p_V|_\pi) \circ (p_H|_\pi)^{-1}, \cdot \right) 
\in H^*_x\otimes V^*_x \subseteq \Lambda^2 T^*_xM$. 

In particular, an immersion $\iota\in\mathcal{HPI}$ defines a 2-form of this type along its image in $M$.
\end{proposition}

\begin{proof}
By definition, each (non-oriented) horizontally projectable $q$-plane is 
identified with the graph of the linear map $T_\pi: H_x \to V_x$. 
\end{proof}

We can describe $T_\pi$ and $\beta_\pi$ explicitly as follows. 
Let $\{ e_i \}_{i=1}^q$ be an arbitrary basis of $H_x$ 
with its dual $\{ e^i \}_{i=1}^q$. 
Since $p_H|_\pi:\pi \to H_x$ is an isomorphism, 
there exists a basis $\{ v_i \}_{i=1}^q$ of $\pi$ 
such that $p_H(v_i)=e_i$. Then 
\begin{align} \label{eq:Tpi beta pi basis}
T_\pi=\sum_{i=1}^q e^i \otimes p_V (v_i), 
\qquad
\beta_\pi=\sum_{i=1}^q e^i \wedge p_V (v_i)^\flat.  
\end{align}

\subsection{The horizontal inner product and volume form.}
Horizontally projectable planes have, of course, an inner product induced from $g$. This, however, is true for any plane. We will thus be more interested in an alternative inner product which encodes the horizontal projectability. It is defined as follows.

\begin{definition} \label{def:hori inn prod}
For a horizontally projectable subspace $\pi \subseteq T_x M$, 
define the \textbf{horizontal inner product} $g^H_\pi$ on $\pi$ by 
$$
g^H_\pi=(p_H|_\pi)^*g= g \left( (p_H|_\pi) (\cdot), (p_H|_\pi) (\cdot) \right).
$$
Analogously, given $\iota \in \HPI$, we may define the \textbf{horizontal metric} $g^H_\iota$ on $\Sigma$ by
$$g^H_{\iota} = g \left( \iota_*^H (\cdot), \iota_*^H (\cdot) \right).
$$
In other words, 
for any $\iota \in \HPI$ and $\sigma \in \Sigma$, we have 
$(g^H_\iota)_\sigma=\left(\iota_*|_{T_\sigma\Sigma}\right)^* g^H_{\iota_*(T_\sigma \Sigma)}$. 
\end{definition}

\begin{definition} \label{def:hori vol}
Any $\pi\in G_H(s,M)$ is oriented, so we can use the inner product $g^H_\pi$ to define the \textbf{horizontal volume form} $\vol^H_\pi$. Specifically, $\vol^H_\pi$ is the unique positive element of $\Lambda^s(\pi^*)$ 
with $|\vol^H_\pi|_{g^H_\pi}=1$.

Given $\iota\in\mathcal{HPI}$, we will denote the corresponding volume form by $\vol^H_\iota$. 
\end{definition}

\section{Functionals defined via a distribution}\label{sec:functionals}
We now want to define some functions $f$ on $\mathcal{U}:=G_H(s,M)$ or on 
$\mathcal{U}:=G_H^+(q,M)$ and, assuming $\Sigma$ is compact, the corresponding functionals $\Vol_f$. We will use these to define anisotropic calibrations in Section \ref{sec:relative_distr}. 

Recall that, in Section \ref{sec:weighted_cal}, $\Vol_f$ was defined via integration against $\vol_{\iota^*g}$. However, given any function $f:\mathcal{U}\rightarrow\rl$, we may write
\begin{equation}\label{eq:f vs g}
f(\pi)\vol_\pi=f(\pi)\frac{\vol_\pi}{\vol^H_\pi}\vol^H_\pi=g(\pi)\vol^H_\pi
\end{equation}
where $\pi\in \mathcal{U}$ and $g(\pi)=f(\pi)\frac{\vol_\pi}{\vol^H_\pi}=f(\pi)/\vol^H_\pi(e_1,\dots,e_s)$, where $\{ e_1,\dots,e_s \}$ is a positive orthonormal basis for $(\pi, g|_\pi)$. In this sense, functionals defined via integration against  $\vol_{\iota^*g}$ are equivalent to functionals defined by integration against $\vol^H_\iota$. From now on, we will go back and forth between these two points of view.

\

Set $\mathcal{U}:=G_H(s,M)$ so that $\ImmU=\mathcal{HPI}$. The choice $g:\mathcal{U}\rightarrow\rl$, $g\equiv 1$, leads to the \textbf{H-volume functional}
$$\Vol^H:\mathcal{HPI}\rightarrow\rl, \qquad 
\Vol^H(\iota):=\int_\Sigma \vol^H_\iota.$$
Given $\pi\in G_H(s,M)$, we can alternatively define the function
\begin{align} \label{eq:def of ve pi}
ve(\pi)=\frac{1}{2}\left| (p_V|_\pi) \right|_{g^H_\pi \otimes g}^2.
\end{align}
Given $\iota\in\mathcal{HPI}$, there is a corresponding function
$$ve(\iota)=\frac{1}{2}\left| (\iota_*^V) \right|_{g^H_\iota \otimes g}^2\in C^\infty(\Sigma).$$
Note that $ve(\iota)(\sigma)=ve(\iota_*(T_\sigma \Sigma))$ 
for $\sigma \in \Sigma$.
We can then define the \textbf{vertical energy functional}
$$VE:\mathcal{HPI}\rightarrow\rl, \ \ VE(\iota):=\int_\Sigma ve(\iota) \vol^H_\iota.$$

\begin{remark} 
We have the following identity:

\begin{align} \label{eq:norm d iota}
\begin{split}
|d \iota|_{g^H_{\iota} \otimes g}^2
= 
\sum_{i=1}^s \left|\iota_*(e_i) \right|^2 
=
\sum_{i=1}^s \left(\left|\iota^H_*(e_i) \right|^2 + \left|\iota^V_*(e_i) \right|^2 \right)
=
s + |\iota^V_*|_{g^H_{\iota} \otimes g}^2, 
\end{split}
\end{align}
where 
$\{e_i \}$ is a local orthonormal frame with respect to $g^H_{\iota}$. 
So the total energy defined by $g^H_{\iota}$ and $g$ is given by 
$$
\frac{1}{2} 
\int_\Sigma |d \iota|_{g^H_{\iota} \otimes g}^2\vol^H_\iota
=\frac{s}{2} \Vol^H(\iota)+ VE(\iota). 
$$
\end{remark}

\begin{remark}\label{rem:E vs VE}
Superficially, the total energy functional looks analogous to the standard energy functional $E(\iota)=(1/2)\int_\Sigma|d\iota|^2_{g_\Sigma\otimes g}\vol_\Sigma$.
However, the energy functional is defined with respect to a fixed choice $g_\Sigma$ of metric on $\Sigma$. Our metric and volume form vary with the image submanifold, so the integrands are well-defined on the Grassmannian. It follows that, contrary to the standard energy functional, our functionals are parametrization-independent. For example, $VE (\iota)=VE (\iota \circ f)$ 
for any $\iota \in \HPI$ and $f \in \Diff_+(\Sigma)$, so $VE$ descends to a map on $\HPI/\Diff_+(\Sigma)$. We can check this fact directly, as follows.

Let $\{ e_i \}_{i=1}^s$ be an orthonormal basis of $T_{f(\sigma)} \Sigma$ with respect to $g^H_\iota$ 
for $\sigma \in \Sigma$. 
Then 
$\{ f^{-1}_*(e_i) \}_{i=1}^s$ is an orthonormal basis of $T_{\sigma} \Sigma$ 
with respect to $g^H_{\iota \circ f}$, so 
\begin{align*}
|(\iota \circ f)^V_*|_{g^H_{\iota \circ f} \otimes g}^2 (\sigma)
=
\sum_{i=1}^s \left|(\iota \circ f)^V_* \left(f^{-1}_*(e_i) \right) \right|^2_g
=
\sum_{i=1}^s |\iota^V_* (e_i) |^2_g
=
|\iota^V_*|_{g^H_{\iota} \otimes g}^2 (f(\sigma)). 
\end{align*}
Hence 
$|(\iota \circ f)^V_*|_{g^H_{\iota \circ f} \otimes g}^2 = |\iota^V_*|_{g^H_{\iota} \otimes g}^2 \circ f$. 
We also have $\vol^H_{\iota \circ f}=f^*\vol^H_{\iota}$, so $VE (\iota)=VE (\iota \circ f)$.
\end{remark}

\subsection{Example: integrable distributions.}\label{ss:integrable dist}
Consider the following hierarchy of examples.
Each is a special case of the preceding one.

\begin{enumerate}
\item \textbf{Vertical foliations:} In this case $V$ is integrable, so $M$ is endowed with a foliation.
\item \textbf{Fiber bundles:} In this case $M$ is the total space of a fiber bundle $\pi:M\rightarrow B$ and $V=\ker(d\pi)$, $H=\ker(d\pi)^\perp$. We can view $H$ as a connection on the fiber bundle in the sense of Ehresmann.
\item \textbf{Riemannian submersions:} In this case $M$ is the total space of a Riemannian submersion $\pi:M\rightarrow B$, i.e. (i) $V=\ker(d\pi)$, $H=\ker(d\pi)^\perp$, (ii) $B$ has a Riemannian metric $g_B$, and (iii) for each $x \in M$, the restriction $d\pi_x|_{H_x}:H_x\rightarrow T_{\pi(x)}B$ is an isometry. 
\item \textbf{Products:} In this case $M$ is a Riemannian product $(M_1,g_1)\times (M_2,g_2)$ and $H=TM_1\oplus\{0\}$, $V=\{0\}\oplus TM_2$. In particular, both distributions are integrable and both projections define Riemannian submersions.
\end{enumerate}

One might also be interested in the following class.
\begin{enumerate}
\item[(5)]\textbf{Horizontal foliations:} In this case $H$ is integrable.
\end{enumerate}
Let us see how our definitions apply in some of these settings.

\begin{itemize}
\item Assume $H$ coincides with $TM$, so that $V=\{0\}$. Both distributions are integrable. We can think of this as a special case of the Riemannian product situation: $M=M\times\{0\}$. 
Any immersion $\iota:\Sigma\rightarrow M$ is horizontally projectable and the metric $g^H_\iota$ coincides with the standard induced metric. 
It follows that $\Vol^H=\Vol$.

\item Assume there is a fiber bundle structure $\pi:M\rightarrow B$. Any section $\iota:\Sigma=B\rightarrow M$ is a horizontally projectable immersion. In this setting we can not only project the tangent bundle $\iota_*(TB)$ onto $H$ using $p_H$, but also project the submanifold $\iota(B)$ onto the base $B$ using $\pi$. This provides an identification between $\iota(B)$ and $B$. Using this identification, the role of $\iota$ is basically to define varying metrics on 
$\Sigma=B$.

\item Assume $\pi:M\rightarrow B$ is a Riemannian submersion. Then given any section $\iota:\Sigma=B\rightarrow M$ the metric $g^H_\iota$ coincides with the metric on $B$. 
It follows that $\Vol^H(\iota)$ is constant for sections $\iota$. 

\item Consider the product case: $(M,g)=(M_1,g_1)\times (M_2,g_2)$. Choose $\Sigma=M_1$. Then any map $f:M_1\rightarrow M_2$ defines a horizontally projectable immersion $\iota(x)=(x,f(x))$.
\end{itemize}

\begin{remark}
Recall that the standard volume functional is elliptic modulo reparametrizations. Sections, however, are a very rigid class of maps, which cannot be reparametrized. In the context of Riemannian submersions, the fact that the functional $\Vol^H$ is constant on sections implies that its first variation is identically zero. It follows that, even modulo reparametrizations, we cannot expect it to have an elliptic linearization.
\end{remark}

As mentioned above, in the presence of a fiber bundle structure $\pi:M\rightarrow B$, any section is a horizontally projectable immersion. This particular class can be recognized by the fact that (i) the image submanifold is embedded, (ii) the projection map $\pi$ defines a diffeomorphism between the image submanifold and the base space $B$. 

Clearly, not all horizontally projectable immersions have these features. For example, assume that $\Sigma$ is compact and has the same dimension as $B$, and that $B$ is connected. Choose a positive horizontally projectable immersion $\iota:\Sigma\rightarrow M$. Then $\pi\circ\iota:\Sigma\rightarrow B$ is a proper local diffeomorphism. It is thus a covering map, rather than a diffeomorphism. Its ``inverse'' can be viewed as a multi-valued section of $B$, i.e. a map that associates to any $b\in B$ a subset of the fiber.

For example, assume $M=S^1\times \rl$ where $S^1=\rl/\Z=B$. Choose any section $\iota(x)=(x,s(x))$. Then the map $\Sigma=S^1\rightarrow S^1\times\rl$, $x\mapsto(2x,s(x))$, is horizontally projectable and projects to a 2:1 covering map of $B$.

\subsection{($\star$) Further comments on these functionals.}\label{ss:comments}
We collect here some additional remarks which may be useful to clarify the nature and properties of these functionals.

\ 

1. It is interesting to compare the two volume forms. 

\begin{lemma} \label{lem:volH ineq vol}
For any $\pi \in G_H(s,M)$, we have $\vol^H_\pi \leq \vol_\pi$. 
Equality holds if and only if $\pi$ is horizontal, i.e. $\pi \subseteq H$. 
\end{lemma}

\begin{proof}
Let $\{v_i \}_{i=1}^s$ be a positive orthonormal basis of $(\pi, g^H_\pi)$. Let $\{v^i \}_{i=1}^s$ denote the dual basis. Set 
$g_\pi= \sum_{i,j=1}^s g_{i j} v^i \otimes v^j$. 
Let $\{ \lambda_i \}_{i=1}^s$ be the eigenvalues of 
the symmetric matrix $\left( g_{i j} \right)$.
Since $g_\pi(v,v)\geq g^H_\pi(v,v)$ for any $v \in \pi$, 
we have $\lambda_i \geq 1$ for any $i$.
Then 
$$
\vol_\pi = \sqrt{\det \left( g_{i j} \right)} v^1 \wedge \cdots \wedge v^s
= \sqrt{\lambda_1 \cdots \lambda_s} \vol^H_\pi 
\geq \vol^H_\pi. 
$$
The equality holds if and only if $\lambda_i = 1$ for any $i$, i.e. 
$g_\pi=g^H_\pi$. This holds if and only if $\pi$ is horizontal. 
\end{proof}

This provides a geometric reason to be interested in $\Vol^H$, as follows.

\begin{proposition} \label{prop:variation horizontal}
On $\mathcal{HPI}$, the $\Vol^H$ functional provides a lower bound for the classical volume functional: $\Vol^H(\iota)\leq \Vol(\iota)$. Equality holds if and only if $\iota$ is horizontal.

In particular, at each horizontal $\iota$, the two functionals have the same first variation because $\Vol(\iota)-\Vol^H(\iota)$ attains its minimum there. It follows that a horizontal $\iota$ is $\Vol^H$-critical if and only if it is $\Vol$-critical, i.e. if and only if it defines a minimal submanifold.
\end{proposition}

Assume, for example, that $H$ is integrable. This proposition then allows us to study minimal submanifolds contained in the leaves of $H$ as critical points of the functional $\Vol^H$.

\ 

2. The choice of metric used in the definition implies that the function $ve$ does not admit a bounded extension to the space $G(s,M)$.

Indeed, for $\pi\in G_H(s,M)$, 
let $\{w_i \}_{i=1}^s$ be a positive orthonormal basis of $(\pi, g_\pi=g |_\pi)$ 
with its dual $\{w^i \}_{i=1}^s$.
Since $p_V|_\pi=\sum_{i=1}^s w^i \otimes p_V(w_i)$, we have 
\begin{align*}
ve(\pi)= \frac{1}{2}\sum_{i,j=1}^s (g^H_\pi)^* (w^i, w^j) \cdot g(p_V(w_i), p_V(w_j)), 
\end{align*}
where $(g^H_\pi)^*$ is the dual metric on $\pi^*$ induced from $g^H_\pi$. 
Since the matrix $\left( (g^H_\pi)^* (w^i, w^j) \right)$ 
is the inverse matrix of $W=\left( g^H_\pi (w_i, w_j) \right)$, we have 
\begin{align*}
(g^H_\pi)^* (w^i, w^j)
=\frac{{\rm adj} (W)_{i j}}{\det W}, 
\end{align*}
where ${\rm adj} (W)$ is the adjugate of $W$ and ${\rm adj} (W)_{i j}$ is its 
$(i,j)$ component. 
Since ${\rm adj} (W)_{i j}$ and $\det W$ are polynomials of 
the components of $W$ of degree $(s-1)$ and $s$, respectively, 
$ve(\pi)$ does not admit a bounded extension to $\pi\in G(s,M)\setminus G_H(s,M)$. 

Since $\vol^H_\pi=\sqrt{\det W} w^1 \wedge \cdots \wedge w^s
= \sqrt{\det W} \vol_\pi$, 
$f(\pi)=ve(\pi)(\vol^H_\pi/\vol_\pi)$ does not admit such an extension either. 
However, note that $\sqrt{ve(\pi)} \vol^H_\pi$
and $ve(\pi) (\vol^H_\pi)^{\otimes 2}$ can be extended.

\subsection{$(\star)$ The vertical energy hierarchy.} 
Vertical energy fits into a more general construction.

\begin{definition}
For a horizontally projectable subspace $\pi \subseteq T_x M$, 
define $ve_k(\pi)\in \rl$ for $k \geq 0$ by 
\begin{align} \label{eq:ve abstract}
\sqrt{\det \left(\id_\pi + \varepsilon {}^t (p_V|_\pi) \circ (p_V|_\pi) \right)}
= 
\sum_{k=0}^\infty ve_k(\pi) \varepsilon^k 
\end{align}
as a formal power series in $\varepsilon$.
Here, ${}^t(p_V|_\pi):V_x\to\pi$ is the adjoint of
$p_V|_\pi:\pi\to V_x$ with respect to $g|_{V_x\times V_x}$ and $g^H_\pi$.
\end{definition}

We can describe \eqref{eq:ve abstract} in a pointwise fashion as follows. 
Let $\{ v_i \}_{i=1}^s$ be an orthonormal basis of $(\pi, g^H_\pi)$ 
with dual basis $\{ v^i \}_{i=1}^s$. 
Since 
$p_V|_\pi= \sum_{i=1}^s v^i \otimes p_V (v_i)$, we have 
\begin{align*}
{}^t (p_V|_\pi) \circ (p_V|_\pi)
= 
\sum_{i,j} g^H_\pi \left(({}^t (p_V|_\pi) \circ (p_V|_\pi))(v_i), v_j \right) 
v^i \otimes v_j
= \sum_{i,j} g(p_V(v_i), p_V(v_j)) v^i \otimes v_j. 
\end{align*}
Hence, as a formal power series in $\varepsilon$, 
\begin{align} \label{eq:ve matrix}
\sqrt{\det \left( I_s + \varepsilon \left( g(p_V(v_i), p_V(v_j)) \right)_{i,j} \right)}
= \sum_{k=0}^\infty ve_k(\pi) \varepsilon^k, 
\end{align}
where $I_s$ is the identity matrix of dimension $s$. \\

We can describe $ve_k(\pi)$ via the following recursive formula. 

\begin{lemma} \label{lem:kth vertical energy explicit}
For $\pi \in G_H(s,M)$, we have 
\begin{align*}
ve_0(\pi)&=1, \\
ve_1(\pi)
&=ve(\pi)
= \frac{1}{2} \left| p_V|_\pi \right|_{g^H_\pi \otimes g}^2, \\
ve_k(\pi)
&= 
\frac{1}{2} \left( 
\frac{1}{k!} \left| (p_V|_\pi)^k \right|_{g^H_\pi \otimes g}^2
- \sum^{k-1}_{i=1} ve_i(\pi) ve_{k-i} (\pi)
\right)
\end{align*}
for $k \geq 2$, where 
$$
(p_V|_\pi)^k
=
\underbrace{(p_V|_\pi) \wedge \cdots \wedge (p_V|_\pi)}_{k} 
\in \Lambda^k \pi^* \otimes \Lambda^k V_x
$$
and $\left| (p_V|_\pi)^k \right|_{g^H_\pi \otimes g}^2$
is the squared norm of $(p_V|_\pi)^k$ measured by $g^H_\pi$ and $g$. 
Note that 
$(p_V|_\pi)^k=0$ for $k >s$. 
\end{lemma}

Notice that $ve_1(\pi) \geq 0$, but 
$ve_k(\pi)$ is not necessarily nonnegative 
for $k \geq 2$. 

\begin{proof}
For $\varepsilon >0$, define $c_k=c_k(\pi)$ by 
$$
\det \left(\id_\pi + \varepsilon {}^t (p_V|_\pi) \circ (p_V|_\pi) \right)
= 
\sum_{k=0}^s c_k \varepsilon^k. 
$$
We first show $c_k=\frac{1}{k!} \left| (p_V|_\pi)^k \right|_{g^H_\pi \otimes g}^2$. 

Since ${}^t (p_V|_\pi) \circ (p_V|_\pi):\pi \to \pi$ 
is positive semidefinite symmetric with respect to $g^H_\pi$, 
there exists an orthonormal basis $\{v_1, \dots, v_s \}$ of $(\pi,g^H_\pi)$ 
and $\lambda_i \geq 0$ 
such that 
$$
\left({}^t (p_V|_\pi) \circ (p_V|_\pi)\right) (v_i) = \lambda_i v_i 
\qquad \mbox{for} \quad i=1,\dots, s.
$$
Then 
\begin{align*}
\det \left(\id_\pi + \varepsilon {}^t (p_V|_\pi) \circ (p_V|_\pi) \right)
=
(1+\varepsilon \lambda_1) \cdots (1+\varepsilon \lambda_s)
=
\sum_{k=0}^s \left( \sum_{i_1< \cdots < i_k} \lambda_{i_1} \cdots \lambda_{i_k}\right)
\varepsilon^k. 
\end{align*}
So 
\begin{align*}
c_k= \sum_{i_1< \cdots < i_k} \lambda_{i_1} \cdots \lambda_{i_k}. 
\end{align*}

On the other hand, let $\{v^1, \dots, v^s \}$ be the dual basis of 
$\{v_1, \dots, v_s \}$. 
Since $p_V|_\pi=\sum_i v^i \otimes p_V(v_i)$, we have 
$$
(p_V|_\pi)^k= \pm \sum_{i_1, \dots, i_k} v^{i_1} \wedge \cdots \wedge v^{i_k} \otimes 
p_V(v_{i_1}) \wedge \cdots \wedge p_V(v_{i_k}). 
$$
So 
\begin{align*}
&\left| (p_V|_\pi)^k \right|_{g^H_\pi \otimes g}^2 \\
=& 
\sum_{i_1, \dots, i_k, j_1, \dots, j_k} 
\la v^{i_1} \wedge \cdots \wedge v^{i_k}, v^{j_1} \wedge \cdots \wedge v^{j_k} 
\ra_{g^H_\pi} 
\la p_V(v_{i_1}) \wedge \cdots \wedge p_V(v_{i_k}), 
p_V(v_{j_1}) \wedge \cdots \wedge p_V(v_{j_k})
\ra_g. 
\end{align*}
Since 
$g(p_V(v_{i_\ell}), p_V(v_{j_\ell}))
=
g^H_\pi \left(\left({}^t (p_V|_\pi) \circ (p_V|_\pi)\right) (v_{i_\ell}), v_{j_\ell} \right)
=
\lambda_{i_\ell} \delta_{i_\ell j_\ell}$, 
we obtain
\begin{align*}
\left| (p_V|_\pi)^k \right|_{g^H_\pi \otimes g}^2
=
\sum_{i_1, \dots, i_k} \lambda_{i_1} \cdots \lambda_{i_k} 
\left| v^{i_1} \wedge \cdots \wedge v^{i_k} \right|_{g^H_\pi}^2
=
k! \sum_{i_1< \cdots < i_k} \lambda_{i_1} \cdots \lambda_{i_k}, 
\end{align*}
which implies that 
$c_k=\frac{1}{k!} \left| (p_V|_\pi)^k \right|_{g^H_\pi \otimes g}^2$.\\

By squaring both sides of \eqref{eq:ve abstract}, we obtain 
\begin{align*}
\sum_{k=0}^s c_k \varepsilon^k
=
\left( \sum_{i=0}^\infty ve_i(\pi) \varepsilon^i \right)^2
= 
\sum_{i,j=0}^\infty ve_i(\pi) ve_j(\pi) \varepsilon^{i+j}. 
\end{align*}
So 
$$
c_k=\sum_{i+j=k} ve_i(\pi) ve_j(\pi). 
$$
By definition, 
$ve_0(\pi)=1$ and $ve_1(\pi)=c_1/2$. For $k \geq 2$, we have 
$$
c_k=2 ve_k(\pi)+ \sum_{i=1}^{k-1} ve_i(\pi) ve_{k-i}(\pi), 
$$
which implies the statement. 
\end{proof}

\begin{definition} \label{def:k th vert energy}

Given $\iota\in\mathcal{HPI}$, define $ve_k(\iota) \in C^\infty(\Sigma)$ for $k \geq 1$ by 
$$
(ve_k(\iota))(\sigma)=ve_k(\iota_*(T_\sigma \Sigma)). 
$$
The {\bf $k$-th vertical energy} of $\iota$ is given by 
$$
VE_k(\iota) 
= \int_{\Sigma} ve_k(\iota) \vol^H_{\iota}. 
$$
\end{definition}
Notice that $ve(\iota)=ve_1(\iota)$ and $VE(\iota)=VE_1(\iota)$.

\begin{remark}
Functionals similar to $VE_k (\iota)$ are defined in \cite[Sections 2.2 and 2.3]{RW}, 
where $VE_k$ is defined 
by the expansion of 
$\det \left(\id_\pi + \varepsilon {}^t (p_V|_\pi) \circ (p_V|_\pi) \right)$. 
For our purposes it is more useful to define $VE_k$ using the expansion of 
$\sqrt{\det \left(\id_\pi + \varepsilon {}^t (p_V|_\pi) \circ (p_V|_\pi) \right)}$, as in \eqref{eq:ve abstract}. 
\end{remark}

\section{Anisotropic calibrations via adiabatic limits}\label{sec:relative_distr}

Let us fix the following data:
\begin{itemize}
\item $(M,g)$ is a $(q+r)$-dimensional Riemannian manifold.
\item $\alpha$ is a semi-calibration $q$-form.
\item $H$ is a $q$-dimensional distribution calibrated, thus oriented, by $\alpha$.
\end{itemize} 

We can use the decomposition $TM=H\oplus V$ to obtain a decomposition of $\alpha$:
\begin{align} \label{eq:alpha decomp}
\alpha
=
\sum_{i=0}^q \alpha_i
\qquad \mbox{for} \qquad
\alpha_i \in \Gamma (\Lambda^{q-i} H^* \otimes \Lambda^i V^*). 
\end{align}
Notice that, by assumption, $\alpha_0$ restricts to the natural volume form on $H$. Furthermore, $\alpha_1=0$ by the so-called ``first cousin principle'' (\cite[Lemma 2.4]{HL2}).

\subsection{Adiabatic calibrations}\label{ss:adiabatic cals}

We want to show that the compatibility condition between $\alpha$ and $H$, expressed by the fact that $H$ is calibrated by $\alpha$, implies that one of the components of $\alpha$ is a semi-calibration, $f$-anisotropic with respect to the vertical energy function (which also depends on $H$).

To begin, consider the following rescalings of $\alpha$, $g$ in the direction of $V$: for $\varepsilon>0$, 
\begin{align} \label{eq:alpha eps g eps}
\alpha_\varepsilon:= \sum_{i=0}^q (\sqrt{\varepsilon})^i \alpha_i, 
\qquad
g_\varepsilon:=g|_H + \varepsilon g|_V. 
\end{align}
This construction is standard in the theory of adiabatic limits, mentioned in Section \ref{s:intro}.
Notice that $\alpha_2$ can be recovered as the limit of $(\alpha_\epsilon-\alpha_0)/\epsilon$, as $\epsilon\rightarrow 0$. 

\begin{proposition} \label{prop:alpha eps semi cali}
The $q$-form $\alpha_\varepsilon$ is a semi-calibration with respect to 
$g_\varepsilon$. 
\end{proposition}

\begin{proof}
We only have to show
\begin{align} \label{eq:alpha eps semi cali 1}
\alpha_\varepsilon (v_1, \dots, v_q) 
\leq \sqrt{\det \left( g_\varepsilon (v_i, v_j) \right)} 
\end{align}
for any $x \in M$ and $v_1, \dots, v_q \in T_x M$. 
Since $\alpha$ is a semi-calibration, we have 
\begin{align} \label{eq:alpha eps semi cali 2}
\alpha (w_1, \dots, w_q) \leq \sqrt{\det \left( g(w_i, w_j) \right)}
\end{align}
for any $x \in M$ and $w_1, \dots, w_q \in T_x M$. 
Set $v_i^\varepsilon := p_H(v_i) + \sqrt{\varepsilon} p_V (v_i)$. 
Choosing $w_i=v_i^\varepsilon$, 
\begin{align*}
\alpha (w_1, \dots, w_q)
&=\sum_{i=0}^q (\sqrt{\varepsilon})^i \alpha_i (v_1, \dots, v_q)
= \alpha_\varepsilon (v_1, \dots, v_q), \\
g(w_i, w_j)
&= g_\varepsilon (v_i, v_j)
\end{align*}
so \eqref{eq:alpha eps semi cali 2} implies \eqref{eq:alpha eps semi cali 1}. 
\end{proof}

\begin{theorem} \label{thm:from cali to rel cali}
For any $\pi \in G_H^+(q,M)$, we have
\begin{align} \label{eq:from cali to rel cali ineq}
\alpha_2|_\pi \leq ve (\pi) \vol^H_\pi, 
\end{align}
where $ve (\pi)$ is defined in \eqref{eq:def of ve pi}.
It follows that $\alpha_2$ is an anisotropic semi-calibration with respect to the function
\begin{align} \label{eq:from cali to rel cali fn}
f:\mathcal{U} \rightarrow \rl, \qquad  
f(\pi)= ve (\pi) \frac{\vol^H_\pi}{\vol_\pi},
\end{align}
where $\mathcal{U}:=G_H^+(q,M)$. 
\end{theorem}
\begin{proof}
Let $\{ v_1,\dots,v_q\}$ be a positive orthonormal basis for $(\pi, g^H_\pi)$. 
By Proposition \ref{prop:alpha eps semi cali}, we have 
\begin{align} \label{eq:cali ineq eps}
\alpha_\varepsilon (v_1, \dots, v_q) \leq 
\sqrt{\det \left( g_\varepsilon (v_i, v_j) \right)} 
\end{align}
for any $\varepsilon >0$. 
Since $p_H|_\pi:\pi \to H_x$ is orientation-preserving,
we have 
$\alpha_0 (v_1, \dots, v_q)
= \alpha (p_H(v_1), \dots, p_H(v_q)) =1. 
$
Thus, 
$$
\alpha_\varepsilon (v_1, \dots, v_q)
=
1 + 
\sum_{i=2}^q (\sqrt{\varepsilon})^i \alpha_i (v_1, \dots, v_q). 
$$
Since $g_\varepsilon(v_i, v_j)=\delta_{i j} 
+ \varepsilon g(p_V(v_i), p_V(v_j))$, 
\eqref{eq:ve matrix} implies that
$$
\sqrt{\det \left( g_\varepsilon (v_i, v_j) \right)}
= 1+ \sum_{k=1}^\infty ve_k(\pi) \varepsilon^k. 
$$
By subtracting $1$ from both sides of \eqref{eq:cali ineq eps}, 
dividing by $\varepsilon$, 
and then letting 
$\varepsilon \to 0$, we obtain \eqref{eq:from cali to rel cali ineq}. Note that $ve_1(\pi)=ve(\pi)$.
\end{proof}

By Proposition \ref{prop:minimize Volf}, we see the following. 

\begin{corollary}\label{cor:rel cali minimize VE}
Let $f$ be as in \eqref{eq:from cali to rel cali fn}.
If $d\alpha_2=0$ then any compact $f$-calibrated submanifold minimizes the functional $VE$ in its restricted homology class, in the sense of Definition \ref{def:K happy}.
\end{corollary}

We can formalize these facts using the following terminology.

\begin{definition}\label{def:ad cal}
Let $(M,g)$ and $(\alpha,H)$ be as above. Using the notation of \eqref{eq:alpha decomp}, we will say that $\alpha_2$ is the \textbf{(anisotropic) adiabatic semi-calibration} defined by $\alpha$. We will refer to its calibrated submanifolds as \textbf{adiabatic calibrated}.
\end{definition}

Notice that adiabatic calibrated submanifolds are a special class of (anisotropic) $f$-calibrated submanifolds, generated via the adiabatic process.

\subsection{The equality property}
Theorem \ref{thm:from cali to rel cali} introduces a class of $f$-calibrated planes, but provides no alternative way to check if a plane has this property. 
In addition to the assumptions at the beginning of Section \ref{sec:relative_distr}, 
let us now assume that $\alpha \in \Omega^q (M)$ 
satisfies the equality property \eqref{eq:eq prop} for a certain $\chi \in \Omega^q(M,E)$. In this case, we will find a close analogue of the equality property for $\alpha_2$ which can help us distinguish the $f$-calibrated planes. 
Decompose 
\begin{align} \label{eq:chi decomp}
\chi
=
\sum_{i=0}^q \chi_i
\qquad \mbox{for} \qquad
\chi_i \in \Gamma (E \otimes \Lambda^{q-i} H^* \otimes \Lambda^i V^*). 
\end{align}
Since $H$ is a calibrated distribution, 
\eqref{eq:eq prop} implies that $\chi_0=0$. 
For $\varepsilon>0$, set 
\begin{align}\label{eq:chi eps}
\chi_\varepsilon:= \sum_{i=0}^q (\sqrt{\varepsilon})^i \chi_i. 
\end{align}

Then $\alpha_\varepsilon$ also satisfies the equality property \eqref{eq:eq prop} 
for $\chi_\varepsilon$ as follows. 

\begin{proposition} \label{prop:eq prop alpha eps}
We have 
$$
|\alpha_\varepsilon (v_1, \dots, v_q)|^2
+|\chi_\varepsilon (v_1,\dots,v_q)|^2
=|v_1\wedge\cdots\wedge v_q|^2_{g_\varepsilon} 
$$
for any $x \in M$ and $v_1, \dots, v_q \in T_x M$, 
where $|v_1\wedge\cdots\wedge v_q|^2_{g_\varepsilon}$ 
is the squared norm of $v_1\wedge\cdots\wedge v_q$ 
with respect to $g_\varepsilon$. 
\end{proposition}

\begin{proof}
This is proved as in Proposition \ref{prop:alpha eps semi cali}. 
Setting $v_i^\varepsilon:= p_H(v_i) + \sqrt{\varepsilon} p_V (v_i)$, we have 
$$
\alpha_\varepsilon (v_1, \dots, v_q)
=
\alpha (v^\varepsilon_1, \dots, v^\varepsilon_q), 
\qquad
\chi_\varepsilon (v_1, \dots, v_q)
=
\chi (v^\varepsilon_1, \dots, v^\varepsilon_q), 
$$
and 
$$
|v^\varepsilon_1 \wedge \cdots \wedge v^\varepsilon_q|^2_{g}
=
\det \left( g (v_i^\varepsilon, v_j^\varepsilon) \right)
=
\det \left( g_\varepsilon (v_i, v_j) \right)
=
|v_1\wedge\cdots\wedge v_q|^2_{g_\varepsilon}. 
$$
The equality property \eqref{eq:eq prop} 
for $\alpha$ and $\chi$ thus implies the statement. 
\end{proof}

\begin{lemma}\label{prop:adi eq eps}
For any $v_1, \dots, v_q \in T_x M$, set 
$\vec v = v_1 \wedge \cdots \wedge v_q$. 
Decompose 
$$
\vec v = \sum_{i=0}^q \vec v_i
\qquad \mbox{for} \qquad
\vec v_i \in \Lambda^{q-i} H \otimes \Lambda^i V. 
$$
Set $\alpha_i=0, \chi_i=0$ and $\vec v_i=0$ for $i >q$.

\begin{enumerate}
\item 
For any $\ell \geq 1$, we have 
\begin{align}
\sum_{i+j=2\ell} \alpha_i (\vec v) \alpha_j (\vec v)
+ 
\sum_{i+j=2\ell} g_E \left(\chi_i (\vec v), \chi_j (\vec v) \right)
=&
|\vec v_\ell|^2, \label{eq:adi eq even}\\
\sum_{i+j=2\ell+1} \alpha_i (\vec v) \alpha_j (\vec v)
+ 
\sum_{i+j=2\ell+1} g_E \left(\chi_i (\vec v), \chi_j (\vec v) \right)
=&0. \label{eq:adi eq odd}
\end{align}

\item 
If $\{ v_1, \dots, v_q \}$ is a basis of $\pi \in G_H(q,M)$, we have 
$$
|\vec v_\ell|^2= \left|\vol^H_\pi (\vec v) \right|^2 \cdot 
\sum_{i+j=\ell} ve_i(\pi) ve_j(\pi). 
$$ 
\end{enumerate}
\end{lemma}

\begin{proof}
(1) By Proposition \ref{prop:eq prop alpha eps}, we have 
\begin{align*}
\sum_{i,j=0}^q (\sqrt{\varepsilon})^{i+j} \alpha_i (\vec v) \alpha_j (\vec v)
+
\sum_{i,j=0}^q (\sqrt{\varepsilon})^{i+j} 
g_E \left(\chi_i (\vec v), \chi_j (\vec v) \right)
= 
\sum_{i=0}^q \varepsilon^{i} 
|\vec v_i|^2, 
\end{align*}
where we use $g(\vec v_i, \vec v_j)=0$ for $i \neq j$. 
Comparing the coefficients of $(\sqrt{\varepsilon})^{2 \ell}$ 
and $(\sqrt{\varepsilon})^{2 \ell+1}$, 
we obtain 
\eqref{eq:adi eq even} and \eqref{eq:adi eq odd}, respectively.

(2) Let $\{ u_i \}_{i=1}^q$ be a positive orthonormal basis of $(\pi, g^H_\pi)$. 
Set $\vec u=u_1 \wedge \cdots \wedge u_q$.  
Then 
$$
\vec v= \vol^H_\pi (\vec v) \vec u, 
\qquad \mbox{and} \qquad
\sum_{k=0}^q \varepsilon^{k} 
|\vec v_k|^2
=
|\vec v|^2_{g_\varepsilon}
= 
|\vol^H_\pi (\vec v)|^2
|\vec u|^2_{g_\varepsilon}. 
$$
Since 
$g_\varepsilon (u_i, u_j)=\delta_{i j} + \varepsilon g(p_V(u_i), p_V(u_j))$, 
\eqref{eq:ve matrix} implies that 
$$
|\vec u|^2_{g_\varepsilon}
=
\left( \sum_{i=0}^\infty ve_i(\pi) \varepsilon^i \right)^2
= 
\sum_{i,j=0}^\infty ve_i(\pi) ve_j(\pi) \varepsilon^{i+j}. 
$$
By comparing the coefficients of powers of $\varepsilon$, 
we obtain the desired statement. 
\end{proof}

\begin{corollary}\label{cor:weighted eq prop dist}
Using the above notation:
\begin{enumerate}
\item 
If $\pi \in G_H(q,M)$ 
and $\{ v_i \}_{i=1}^q$ is a basis of $\pi$, then 
$$
\alpha_0 (\vec v) \alpha_2 (\vec v)
+ \frac{1}{2}|\chi_1 (\vec v)|^2 
= |\vol^H_\pi (\vec v)|^2 \cdot ve_1 (\pi), 
$$
where $\vec v = v_1 \wedge \cdots \wedge v_q$.

\item
If $\pi \in G_H^+(q,M)$ and 
$\{ v_i \}_{i=1}^q$ is a positive orthonormal basis of $(\pi, g^H_\pi)$, 
then 
\begin{align} \label{eq:adi eq 2 php}
\alpha_2 (\vec v)
+ \frac{1}{2}|\chi_1 (\vec v)|^2 
= 
ve_1 (\pi).
\end{align}
Hence the equality holds in \eqref{eq:from cali to rel cali ineq} 
if and only if $\chi_1 |_\pi =0$. 
\end{enumerate}
\end{corollary}

\begin{proof}
(1) follows from the case $\ell=1$ in \eqref{eq:adi eq even}, using $\alpha_1=0$ and $\chi_0=0$. 
(2) follows from 
$\alpha_0 (\vec v)=\alpha_0(p_H(v_1), \dots, p_H(v_q))=1$ 
since $\alpha_0$ restricts to the natural volume form on $H_x$ 
and $\{ p_H(v_i) \}_{i=1}^q$ is a positive orthonormal basis of $H_x$. 
\end{proof}

Note that when 
$\alpha=\alpha_0+\alpha_2$, i.e. $\alpha_i=0$ for $i \geq 3$, 
\eqref{eq:adi eq 2 php} can be written as 
\begin{align} \label{eq:adi eq 2 php 2}
\alpha (\vec v)
+ \frac{1}{2}|\chi_1 (\vec v)|^2 
= 
ve_1 (\pi)+1 
\end{align}
since $\alpha_0(\vec v)=1$. 
Then we immediately see the following.

\begin{corollary} \label{cor: minimize VE+VolH}
Suppose that $\alpha=\alpha_0+\alpha_2$. 
Let $f$ be as in \eqref{eq:from cali to rel cali fn}.
If $d\alpha=0$ then any compact $f$-calibrated submanifold minimizes the functional $VE+\Vol^H$ in its restricted homology class, in the sense of Definition \ref{def:K happy}.
\end{corollary}

\begin{remark} \label{rem:chi1 linear}
For $\pi \in G_H^+(q,M)$, set 
$T_\pi:=(p_V|_\pi) \circ (p_H|_\pi)^{-1} \in H^*_x\otimes V_x$
and 
$\beta_\pi:=g \left((p_V|_\pi) \circ (p_H|_\pi)^{-1}, \cdot \right) 
\in H^*_x\otimes V^*_x \subseteq \Lambda^2 T^*_xM$ 
as in Proposition \ref{lem:unori hp hom}. 
This construction endows $G_H^+(q,M)$ with a linear structure. The condition $\chi_1|_\pi=0$ then corresponds to a linear equation, which we shall denote $\cP(T_\pi)=0$ or $\cP(\beta_\pi)=0$. 

Indeed, 
let $\{ e_i \}_{i=1}^q$ be a positive orthonormal basis of $(H, g|_H)$ 
with its dual $\{ e^i \}_{i=1}^q$. 
Then $\{ v_i:=e_i +T_\pi(e_i) \}_{i=1}^q$ is a basis of $\pi$. 
Setting 
$\vec v = v_1 \wedge \cdots \wedge v_q$, we have 
$$
\chi_1 (\vec v)
= 
\chi (T_\pi(e_1), e_2, \dots, e_q)
+ 
\chi (e_1, T_\pi(e_2), \dots, e_q)
+
\cdots
+
\chi (e_1, \dots, e_{q-1}, T_\pi(e_q)). 
$$
We can also write 
$\chi_1 (\vec v)=
\left( \left. \frac{d}{dt} (e^{t T_\pi})^* \chi \right|_{t=0} \right) 
(e_1, \dots, e_q) 
=
\left( \sum_{i=1}^q e^i \wedge i(T_\pi(e_i)) \chi \right) 
(e_1, \dots, e_q)$. 

\end{remark}

\subsection{Relation to adiabatic limits} \label{sec:rel ad lim}
 
We shall now show that, in the above setting, the adiabatic calibrated equation, i.e. the equality in \eqref{eq:from cali to rel cali ineq}, is the ``adiabatic limit'' of the calibrated equation.

Recall from Section \ref{sec:weighted_cal} that $\pi \in G (q,M)$ is calibrated with respect to $\alpha_\varepsilon$ 
if $\alpha_\varepsilon |_\pi=\vol_{\varepsilon, \pi}$, 
where $\vol_{\varepsilon, \pi}$ is the volume form defined by the induced metric 
$g_\varepsilon|_\pi$. 
Using the equality property we can reformulate this as 
\begin{align}\label{eq:cali eps}
\chi_\varepsilon|_\pi=0,
\end{align}
together with the condition $\alpha_\varepsilon|_\pi >0$.

Dividing both sides of \eqref{eq:cali eps} by $\sqrt{\varepsilon}$ 
and taking the limit $\varepsilon \to 0$, 
we obtain $\chi_1 |_\pi = 0$. 
Note also that $\pi \in G_H^+(q,M)$ if and only if $\alpha_0|_\pi >0$, 
which can be considered as the limit of 
$\alpha_\varepsilon|_\pi >0$ as $\varepsilon \to 0$. 

Let us now consider the limit of the equation 
$\alpha_\varepsilon |_\pi=\vol_{\varepsilon, \pi}$ as $\varepsilon \to 0$. Using the proof of Theorem \ref{thm:from cali to rel cali},
we can describe this equation as follows. 

\begin{lemma} \label{lem:cali alpha eps}
For $\pi \in G_H^+(q,M)$, 
$\alpha_\varepsilon |_\pi=\vol_{\varepsilon, \pi}$ 
if and only if 
\begin{align} \label{eq:cali alpha eps}
\sum_{i=2}^q (\sqrt{\varepsilon})^i \alpha_i|_\pi
=
\sum_{k=1}^\infty ve_k(\pi) \varepsilon^k \cdot \vol^H_\pi,  
\end{align}
where $\vol^H_\pi$ is the volume form defined by $g^H_\pi$ in 
Definition \ref{def:hori vol}. 
\end{lemma}

Dividing both sides of \eqref{eq:cali alpha eps} by $\varepsilon$ 
and taking the limit $\varepsilon \to 0$, 
we obtain 
$\alpha_2|_\pi = ve_1 (\pi) \vol^H_\pi$. 
This is equivalent to $\chi_1|_\pi=0$ by 
Corollary \ref{cor:weighted eq prop dist} (2). 
We thus obtain the following diagram.

\begin{equation} \label{eq:adi lim general}
\begin{aligned}
\xymatrix{
\alpha_\varepsilon |_\pi=\vol_{\varepsilon, \pi} 
\ar@{<=>}[r]
\ar@{-->}[d]_-{\text{ad. lim.}} 
&
\chi_\varepsilon|_\pi=0
\ar@{-->}[d]_-{\text{ad. lim.}} 
&
\mbox{if} \quad \alpha_\varepsilon |_\pi>0
\ar@{-->}[d]_-{\text{ad. lim.}}
\\
\alpha_2|_\pi = ve_1 (\pi) \vol^H_\pi
\ar@{<=>}[r]
&
\chi_1|_\pi=0
&
\mbox{if} \quad \alpha_0|_\pi >0
}
\end{aligned}
\end{equation}
Here, “ad. lim.” denotes the corresponding adiabatic limits, including the appropriate normalizations discussed above. 
In this sense, we may say that the adiabatic calibrated equation is the adiabatic limit of the calibrated equation.

\subsection{($\star$) Some remarks about higher-order calibrations}

\begin{definition} \label{def:k-vanishing}
We say that $\pi \in G_H^+(q,M)$ is \textbf{$k$-vanishing} if 
$\chi_i|_\pi=0$ for all $1 \leq i \leq k$.
\end{definition}

We can generalize Corollary \ref{cor:weighted eq prop dist} as follows.

\begin{proposition} \label{prop:k-vanishing cali}
Choose any $\pi \in G_H^+(q,M)$.
Let $\{ v_i \}_{i=1}^q$ be a positive orthonormal basis of $(\pi, g^H_\pi)$ 
and set $\vec v = v_1 \wedge \cdots \wedge v_q$.
Then we have the following.

\begin{enumerate}
\item The subspace $\pi$ is $k$-vanishing 
if and only if 
\begin{align}
\alpha_{2 \ell} (\vec v)= ve_\ell (\pi), \qquad
\alpha_{2 \ell+1} (\vec v)=0
\qquad \mbox{for any} \quad 1 \leq \ell \leq k. 
\end{align}

\item 
If $\pi$ is $k$-vanishing, then 
\begin{align} \label{eq:k vanishing eq}
\alpha_{2k+2} (\vec v) + \frac{1}{2} |\chi_{k+1} (\vec v)|^2 = ve_{k+1} (\pi), 
\qquad
\alpha_{2k+3} (\vec v) + g_E \left(\chi_{k+1} (\vec v), \chi_{k+2} (\vec v) \right) = 0. 
\end{align}
\end{enumerate}
\end{proposition}

This leads to the following statements, which consider higher-order analogues of adiabatic calibrations and generalize
Theorem \ref{thm:from cali to rel cali} and Corollary \ref{cor:rel cali minimize VE}. 

\begin{corollary} \label{cor:rel cali k-vanishing}
\begin{enumerate}
\item
For any $k$-vanishing $\pi \in G_H^+(q,M)$, 
we have 
\begin{align} \label{eq:k+1 th cali ineq}
\alpha_{2k+2}|_\pi \leq ve_{k+1} (\pi) \vol^H_\pi.  
\end{align}

\item
The equality holds in \eqref{eq:k+1 th cali ineq} 
if and only if $\pi$ is $(k+1)$-vanishing. 
\end{enumerate}
\end{corollary}

\begin{remark}
Notice that the $k$-vanishing condition does not define an open subset of $G_H^+(q,M)$ 
and $ve_{k+1}(\pi)$ is not always $\ge 0$, so
we do not use the term ``anisotropic semi-calibration'' above. 
However, if $d \alpha_{2k+2}=0$, then any compact ``$(k + 1)$-vanishing'' submanifold will minimize $VE_{k+1}$ in an appropriate sense.
\end{remark}

\begin{proof}[Proof of Proposition \ref{prop:k-vanishing cali}]
We prove (1) and (2) simultaneously by induction on $k$.
Suppose $k=1$. By \eqref{eq:adi eq even} and \eqref{eq:adi eq odd} for $\ell=1$, 
we have 
\begin{align*}
\alpha_{2} (\vec v) + \frac{1}{2} |\chi_1 (\vec v)|^2 = ve_1 (\pi), 
\qquad
\alpha_{3} (\vec v) + g_E \left(\chi_{1} (\vec v), \chi_{2} (\vec v) \right) = 0, 
\end{align*}
where we use $\alpha_0(\vec v)=1, \alpha_1=0$ and $\chi_0=0$. 
Then we obtain (1). 
When $\pi \in G_H^+(q,M)$ is 1-vanishing, 
\eqref{eq:adi eq even} and \eqref{eq:adi eq odd} for $\ell=2$ imply that 
$$
2 \alpha_4 (\vec v) + \alpha_2 (\vec v)^2+|\chi_2 (\vec v)|^2=
2 ve_2 (\pi)+ve_1 (\pi)^2, 
\qquad
\alpha_{5} (\vec v) + g_E \left(\chi_{2} (\vec v), \chi_{3} (\vec v) \right) = 0. 
$$
Since $\alpha_2 (\vec v)= ve_1 (\pi)$ by (1), we obtain (2). \\

Next, suppose that (1) and (2) hold for $k$. 
We first show (1) for $k+1$. 
Suppose that $\pi \in G_H^+(q,M)$ is $(k+1)$-vanishing.
Since $\pi$ is $k$-vanishing, we have 
$\alpha_{2 \ell} (\vec v)= ve_\ell (\pi)$ and $\alpha_{2 \ell+1} (\vec v)=0$
for any $1 \leq \ell \leq k$. 
By (2) for $k$, we have 
$\alpha_{2k+2} (\vec v) = ve_{k+1} (\pi)$ and $\alpha_{2k+3} (\vec v)=0$. 

Conversely, suppose that 
$\alpha_{2 \ell} (\vec v)= ve_\ell (\pi)$ and $\alpha_{2 \ell+1} (\vec v)=0$
for any $1 \leq \ell \leq k+1$. 
Then from (1) for $k$, $\pi$ is $k$-vanishing. 
By (2) for $k$ and $\alpha_{2 (k+1)} (\vec v)= ve_{k+1} (\pi)$, 
we have $\chi_{k+1}(\vec v)=0$. \\

Next, we show (2) for $k+1$. 
For any $(k+1)$-vanishing $\pi \in G_H^+(q,M)$, 
\eqref{eq:adi eq even} and \eqref{eq:adi eq odd} 
for $\ell=k+2$ imply that 
\begin{align*}
2 \alpha_{2 k+4} (\vec v)
+ \sum_{i=1}^{2 k+3} \alpha_i (\vec v) \alpha_{2k+4-i} (\vec v)
+ |\chi_{k+2} (\vec v)|^2
&= 
2ve_{k+2}(\pi) + \sum_{i=1}^{k+1} ve_i(\pi) ve_{k+2-i}(\pi), \\
2 \alpha_{2k+5} (\vec v)
+ \sum_{i=1}^{2k+4} \alpha_i (\vec v) \alpha_{2k+5-i} (\vec v)
+ 2 g_E \left(\chi_{k+2} (\vec v), \chi_{k+3} (\vec v) \right)
&=0. 
\end{align*}
By (1) for $k+1$, we have 
$$
\sum_{i=1}^{2 k+3} \alpha_i (\vec v) \alpha_{2k+4-i} (\vec v)
= 
\sum_{i=1}^{k+1} \alpha_{2i} (\vec v) \alpha_{2k+4-2i} (\vec v)
=
\sum_{i=1}^{k+1} ve_i (\pi) ve_{k+2-i} (\pi), 
\quad 
\sum_{i=1}^{2k+4} \alpha_i (\vec v) \alpha_{2k+5-i} (\vec v)=0.
$$
Hence we obtain (2) for $k+1$. 
\end{proof}

\section{A gauge-theoretic interpretation}\label{sec:mirror general}
Recall from Proposition \ref{lem:unori hp hom} that any horizontally projectable immersion $\iota$ defines a 2-form $\beta_\iota$. The goal of this section is to show that, in a certain setting, this form is closely related to the curvature form of a certain connection. The adiabatic calibrated submanifolds introduced in Definition \ref{def:ad cal} can thus be interpreted as special connections. Although the setting is very restrictive, this construction provides an interesting bridge between calibrations and gauge theory. We will use this in Section \ref{sec:G2Fueter}. 

\subsection{Dual tori}
To begin, let us review the geometry of dual tori. Let
$$
T^r=\mathbb{R}^r/2\pi\mathbb{Z}^r, \qquad
(T^r)^*:=(\mathbb{R}^r)^*/(\mathbb{Z}^r)^*$$ 
be the $r$-dimensional torus and its dual torus, respectively. The relationship between these two tori (and the difference in normalization) can be explained as follows. 

Consider the trivial complex line bundle $(T^r)^*\times \mathbb{C}\to (T^r)^*$, endowed with the standard metric. Let $G_U=C^{\infty}((T^r)^*,S^1)$ be its unitary gauge group. Let $\mathcal{A}_{\mathrm{flat}}$ denote the space of flat Hermitian connections on this bundle. There is an induced action of $G_U$ on $\mathcal{A}_{\mathrm{flat}}$ defined, for $g \in G_U$, by
\[
g\cdot \nabla := g^{-1}\circ \nabla \circ g=\nabla+g^{-1}dg.
\] 
We define the \textbf{moduli space of flat Hermitian connections} as the quotient $\mathcal{M}_{\mathrm{flat}}:=\mathcal{A}_{\mathrm{flat}}/G_U$. General theory shows that it can be identified with $H^1((T^r)^*,\mathbb{R})/2\pi H^1((T^r)^*,\mathbb{Z}) \cong T^r$. In other words: points in $T^r$ parametrize flat connections on the trivial line bundle over $(T^r)^*$.

If we let $[\nabla]$ denote the equivalence class of $\nabla$ in this quotient, the above identification can be made explicit as follows:
$$
T^r \rightarrow \mathcal{M}_{\mathrm{flat}}, 
\qquad
z=[z^{q+1},\dots,z^{q+r}] \mapsto 
\left[ \nabla_z:=d+\sqrt{-1}\sum_{a=q+1}^{q+r} z^a\,dy^a \right], 
$$
where $(y^{q+1},\dots,y^{q+r})$ denote coordinates on $(T^r)^*$ induced from the dual basis of $(\mathbb{R}^r)^*$.

\begin{remark}Checking that this map defines an identification requires three steps.

(i) The correspondence is well-defined. Indeed, if $z'=z+2\pi m$ for some $m=(m_{q+1},\dots,m_{q+r})\in \mathbb{Z}^r$, then $\nabla_{z'}=g_m^{-1}\circ \nabla_z\circ g_m$, where
\[
g_m(y):=\exp\!\left(2\pi\sqrt{-1}\sum_{a=q+1}^{q+r} m_a y^a\right).
\]

(ii) The map is injective. Indeed, if $[\nabla_z]=[\nabla_{z'}]$ in $\mathcal{M}_{\mathrm{flat}}$, then $\sqrt{-1}\sum_{a=q+1}^{q+r}(z'^a-z^a)\,dy^a=g^{-1}dg$ for some $g\in G_U$. 
Since $[g^{-1}dg]\in 2\pi\sqrt{-1}\,H^1((T^r)^*,\mathbb{Z})
=2\pi\sqrt{-1}\sum_{a=q+1}^{q+r}\mathbb{Z}[dy^a]$ (see for example \cite[Lemma 4.1]{KYmirror}),
we see that $z'^a-z^a\in 2\pi\mathbb{Z}$ for all $a$.

(iii) It is simple to check that it is surjective.
\end{remark}

\subsection{The real Fourier--Mukai transform}The following presentation is based upon \cite{LYZ}, \cite[Appendix A]{Kmon}, \cite[Section 2]{KYFM}.

Let $B \subseteq \mathbb{R}^q$ be a contractible open set with coordinates $(x^1,\dots,x^q)$.  Set $M:=B\times T^r$ and $M^*:=B\times (T^r)^*$. 
Choose a smooth map $u:B\to T^r$. Let us continue to denote by $u$ the choice of a lift to a map $B\to\mathbb{R}^r$, with components $u^{q+1},\dots,u^{q+r}$. We shall interpret $\operatorname{graph}u$ as a smooth submanifold of $M$. 

 Applying the above construction to each fiber of $M$, i.e. setting $z=u(x)$, we obtain the \textbf{real Fourier--Mukai transform} of $u$: this is the Hermitian connection
\[
\nabla_u
:= \{ \nabla_{u(x)}\}_{x \in B}
:=d+\sqrt{-1}\sum_{a=q+1}^{q+r} u^a(x)\,dy^a
\]
on the trivial complex line bundle $L^*:=M^*\times \mathbb{C}\to M^*$. Its curvature $2$-form is given by
\begin{equation*}
F_{\nabla_u}=\sqrt{-1}\sum_{i=1}^q\sum_{a=q+1}^{q+r}\frac{\partial u^a}{\partial x^i}\,dx^i\wedge dy^a.
\end{equation*}

Let $\mathcal{A}_0$ denote the space of Hermitian connections on $L^*$, and let
\[
\mathcal{G}_U:=C^{\infty}(M^*,S^1)
\]
be the \textbf{unitary gauge group} acting on the sections of $L^*$. The induced action on $\mathcal{A}_0$ is defined by
\[
g\cdot \nabla := g^{-1}\circ \nabla \circ g=\nabla+g^{-1}dg.
\]
Notice that the connection $\nabla_u$, hence its coefficients, is well-defined only up to the action of $\mathcal{G}_U$ on $L^*$. Indeed, if $u'$ is another lift of the same map $B\to T^r$, then $u'-u:B\to 2\pi\mathbb{Z}^r$ is continuous. Since $B$ is connected, there exists $m\in \mathbb{Z}^r$ such that $u'^a-u^a=2\pi m_a$ for all $a$. Hence $\nabla_{u'}=g_m^{-1}\circ \nabla_u\circ g_m$, where
\[
g_m(x,y):=\exp\!\left(2\pi\sqrt{-1}\sum_{a=q+1}^{q+r} m_a y^a\right).
\]
This is well-defined because the coordinates $y^a$ on $(T^r)^*$ have period $1$. 

We may summarize as follows: the real Fourier--Mukai transform defines a map
\[
\{\text{graphical submanifolds of } B\times T^r\}\cong C^{\infty}(B,T^r)\longrightarrow \mathcal{A}_0/\mathcal{G}_U,
\qquad \operatorname{graph}u\longmapsto [\nabla_u]. 
\]
In other words, a section of $M\to B$ determines a family of flat connections on the fibers of $M^*\to B$. These connections assemble into a (not necessarily flat) Hermitian connection on $L^* \to M^*$. Notice that the curvature $F_{\nabla_u}$ is invariant under the $\mathcal{G}_U$-action.\\

We now want to study the properties of this map. This requires the following language.

Let $\mathrm{Aut}(L^*)$ denote the group of smooth complex line bundle automorphisms $\Phi:L^*\to L^*$ covering diffeomorphisms $\phi:M^*\to M^*$, namely satisfying $\pi\circ \Phi=\phi\circ \pi$, where $\pi:L^*\to M^*$ is the bundle projection. Let $\mathrm{Aut}_U(L^*)\subseteq \mathrm{Aut}(L^*)$ be the subgroup preserving the Hermitian metric. If $\Phi\in \mathrm{Aut}_U(L^*)$ covers $\phi$, then $\Phi$ induces a bundle isomorphism $\widehat{\Phi}:L^*\to \phi^*L^*$, $v_x\mapsto (x,\Phi(v_x))$. For $s\in \Gamma(M^*,L^*)$, let $\phi^*s\in \Gamma(M^*,\phi^*L^*)$ be the pullback section given by $(\phi^*s)(x):=(x,s(\phi(x)))$. The pullback connection $\phi^*\nabla$ on $\phi^*L^*$ is characterized by $(\phi^*\nabla)(\phi^*s)=\phi^*(\nabla s)$. In the fixed trivialization $L^*=M^*\times \mathbb{C}$, if $\nabla=d+\sqrt{-1}A$ for a real $1$-form $A$ on $M^*$, then $\phi^*\nabla=d+\sqrt{-1}\,\phi^*A$. We define the action of $\Phi$ on $\mathcal{A}_0$ by
\[
\Phi\cdot \nabla:=\widehat{\Phi}^{-1}\circ (\phi^*\nabla)\circ \widehat{\Phi}.
\]
For $g\in \mathcal{G}_U$, this recovers the ordinary gauge action $g\cdot \nabla=g^{-1}\circ \nabla \circ g$.

The \textbf{dual torus action} of $(T^r)^*$ on $M^*$ is given by fiberwise translations $\tau_t(x,y):=(x,y+t)$, $t\in (T^r)^*$, and its canonical lift to $L^*$ is $\widetilde{\tau}_t(x,y,z):=(x,y+t,z)$. We define the \textbf{extended unitary gauge group} $\widetilde{\mathcal{G}}_U$ to be the subgroup of $\mathrm{Aut}_U(L^*)$ generated by $\mathcal{G}_U$ and the automorphisms $\widetilde{\tau}_t$, $t\in (T^r)^*$. For $t\in (T^r)^*$, we simply write $t\cdot \nabla:=\widetilde{\tau}_t\cdot \nabla$. In the fixed trivialization, if
\[
\nabla=d+\sqrt{-1}\left(\sum_{i=1}^q a_i(x,y)\,dx^i+\sum_{a=q+1}^{q+r} b_a(x,y)\,dy^a\right),
\]
then
\[
t\cdot \nabla=d+\sqrt{-1}\left(\sum_{i=1}^q a_i(x,y+t)\,dx^i+\sum_{a=q+1}^{q+r} b_a(x,y+t)\,dy^a\right).
\]
Moreover,
\[
t\cdot (g\cdot \nabla)=(\tau_t^*g)\cdot (t\cdot \nabla), \qquad (\tau_t^*g)(x,y):=g(x,y+t).
\]
Therefore, if $\nabla_1\sim \nabla_2$ in $\mathcal{A}_0/\mathcal{G}_U$, then $t\cdot \nabla_1\sim t\cdot \nabla_2$ in $\mathcal{A}_0/\mathcal{G}_U$, and hence the $ (T^r)^*$-action descends to $\mathcal{A}_0/\mathcal{G}_U$.

Let $\pi_{(T^r)^*}:M^*\to (T^r)^*$ and set $F_B:=\ker(d\pi_{(T^r)^*})\subseteq TM^*$. For $\nabla\in \mathcal{A}_0$, we define the \textbf{$B$-direction partial connection} of $\nabla$ by restricting $\nabla$ to $F_B$:
\[
D_B^{\nabla}:=\nabla|_{F_B}:\Gamma(M^*,L^*)\to \Gamma(M^*,F_B^*\otimes L^*).
\]
If $\nabla=d+\sqrt{-1}\bigl(\sum_{i=1}^q a_i\,dx^i+\sum_{a=q+1}^{q+r} b_a\,dy^a\bigr)$, then $D_B^{\nabla}=d_B+\sqrt{-1}\sum_{i=1}^q a_i\,dx^i$, where $d_B$ denotes the trivial partial connection along $F_B$.

\begin{proposition}\label{prop:FM correspondence}
Let $\mathcal{A}_0$ denote the space of Hermitian connections on the trivial complex line bundle $L^*\to M^*$. The map
\[
\{\text{graphical submanifolds of } B\times T^r\}\cong C^{\infty}(B,T^r)\longrightarrow \mathcal{A}_0/\mathcal{G}_U,
\qquad \operatorname{graph}u\longmapsto [\nabla_u],
\]
where ${\rm graph}\,u:=\{(x,u(x))\mid x\in B\}$, 
is injective. Moreover, its image is precisely the subset of $\mathcal{A}_0/\mathcal{G}_U$ consisting of those gauge-equivalence classes $[\nabla]$ for which there exists a representative $\nabla'\in [\nabla]$ satisfying
\[
t\cdot \nabla'=\nabla' \quad \text{for all } t\in (T^r)^*, \qquad D_B^{\nabla'}=d_B.
\]
\end{proposition}

\begin{proof} 
We have already mentioned that the gauge equivalence class $[\nabla_u]$ is well-defined.

Suppose $[\nabla_u]=[\nabla_{u'}]\in \mathcal{A}_0/\mathcal{G}_U$. Choose lifts $u,u':B\to \mathbb{R}^r$ and representatives $\nabla_u,\nabla_{u'}$ as above. By assumption there exists $g\in \mathcal{G}_U$ such that $\nabla_{u'}=g\cdot \nabla_u$, hence
\[
\sqrt{-1}\sum_{a=q+1}^{q+r}(u'^a-u^a)\,dy^a=g^{-1}dg.
\]
Fix $x\in B$ and restrict the above equation to the fiber $F_x:=\{x\}\times (T^r)^*$. Writing $g_x:=g|_{F_x}\in C^{\infty}((T^r)^*,S^1)$, we obtain
\[
\sqrt{-1}\sum_{a=q+1}^{q+r}(u'^a(x)-u^a(x))\,dy^a=g_x^{-1}dg_x \qquad \text{on } (T^r)^*.
\]
Since $[g_x^{-1}dg_x]\in 2\pi\sqrt{-1}\,H^1((T^r)^*,\mathbb{Z})=2\pi\sqrt{-1}\sum_{a=q+1}^{q+r}\mathbb{Z}[dy^a]$ (see for example \cite[Lemma 4.1]{KYmirror}), we see that $u'^a(x)-u^a(x)\in 2\pi\mathbb{Z}$ for any $q+1\le a\le q+r$, which implies $u'(x)=u(x)$ in $T^r=\mathbb{R}^r/2\pi\mathbb{Z}^r$. Since this holds for all $x\in B$, we conclude $u=u'$ and the map is injective.

Next, let $u:B\to T^r$ be smooth. By construction, $t\cdot \nabla_u=\nabla_u$ for all $t\in (T^r)^*$ and $D_B^{\nabla_u}=d_B$. Hence every element in the image satisfies the stated condition.

Conversely, let $[\nabla]\in \mathcal{A}_0/\mathcal{G}_U$ and assume that it admits a representative $\nabla'$ such that $t\cdot \nabla'=\nabla'$ for all $t\in (T^r)^*$ and $D_B^{\nabla'}=d_B$. Write $\nabla'=d+\sqrt{-1}(\sum_{i=1}^q a_i(x,y)\,dx^i+\sum_{a=q+1}^{q+r} b_a(x,y)\,dy^a)$. The condition $t\cdot \nabla'=\nabla'$ implies that all coefficients $a_i$ and $b_a$ are independent of $y$. Since $D_B^{\nabla'}=d_B$, we have $a_i=0$ for all $i$, so $\nabla'=d+\sqrt{-1}\sum_{a=q+1}^{q+r} b_a(x)\,dy^a$. Define $u(x):=[b_{q+1}(x),\dots,b_{q+r}(x)]\in T^r$. Then $\nabla'=\nabla_u$, and therefore $[\nabla]=[\nabla_u]$. This proves the description of the image.
\end{proof}

\subsection{Calibrated submanifolds versus special connections}
Fix the standard volume form on $\rl^{q+r}$. Suppose we are given a calibration $\alpha \in \Lambda^q (\rl^{q+r})^*$ with constant coefficients and that 
$\rl^q=\rl^q \times \{ 0 \} \subseteq \rl^{q+r}$ is calibrated.
Suppose also that 
$\alpha$ satisfies the equality property \eqref{eq:eq prop} 
for a certain $\chi \in \Lambda^q (\rl^{q+r})^* \otimes E$ 
and vector space $E$. 

With a small abuse of notation, 
denote again by $\alpha$ 
 the calibration with the equality property on $B \times T^r$ 
induced from $\alpha$. 
Then 
the standard horizontal distribution $H$ induced by $B$ on $B \times T^r$ 
is calibrated. 
Denote by $V$ its orthogonal complement.

\begin{proposition} \label{prop:FM general}
Given a smooth map $u:B \to T^r$, define the immersion
$$
\iota:B \to B \times T^r, \qquad x \mapsto (x,u(x)). 
$$ 
This $\iota$ is positive horizontally projectable 
in the sense of Definition \ref{def:pos hori proj}. 

Define 
$\beta_\iota \in \Gamma (\iota^* \Lambda^2 T^*(B \times T^r))$ by 
$(\beta_\iota)_b:=\beta_{\iota_*(T_b B)}$, where we use the notation of 
Proposition \ref{lem:unori hp hom}.
Let $\nabla$ be 
the real Fourier--Mukai transform of $u$, and let $\Psi:B\times T^r\to B\times (T^r)^*$ be the diffeomorphism induced by the standard basis of $\mathbb{R}^r$, namely $\Psi(x,[z])=(x,[z/2\pi])$. 
Then, for every $b\in B$,
$$
\beta_{\iota_*(T_bB)}= -2\pi \sqrt{-1}\,(\Psi^*F_\nabla)_{\iota (b)}.
$$
Hence 
$\iota$ is $f$-calibrated if and only if $\cP(F_\nabla)=0$, where $f$ is defined in \eqref{eq:from cali to rel cali fn} and $\cP: H^* \otimes V^* \to E$ is the linear map referred to in Remark \ref{rem:chi1 linear}; in the latter condition, we use $\Psi^*$ to view $F_\nabla$ as an element of $(H^* \otimes V^*)\otimes \mathbb{C}$ and extend $\cP$ complex-linearly.
\end{proposition}

\begin{proof}
Since $\{ \partial/ \partial x^i \}_{i=1}^q$ 
is a basis of $H_{\iota (b)} (\cong T_b B)$ 
and 
$v_i:=(\partial/ \partial x^i, \partial u/ \partial x^i (b)) \in \iota_*(T_b B)$ 
satisfies 
$p_H(v_i)=\partial/ \partial x^i$, 
\eqref{eq:Tpi beta pi basis} implies that 
\[
\beta_{\iota_*(T_b B)}
=
\sum_{i=1}^q dx^i \wedge
\left( \frac{\partial u}{\partial x^i}  (b) \right)^\flat
=
\sum_{i=1}^q \sum_{a=q+1}^{q+r} \frac{\partial u^{a}}{\partial x^{i}}(b)\, dx^i \wedge dz^a,
\]
where $z^{q+1},\dots,z^{q+r}$ are the standard coordinates on $\mathbb{R}^r$, so that $dz^{q+1},\dots,dz^{q+r}$ denotes the induced invariant coframe on the torus factor of $M$.
On the other hand,
\[
-2\pi \sqrt{-1}\,\Psi^*F_\nabla
=
\sum_{i=1}^q \sum_{a=q+1}^{q+r} \frac{\partial u^{a}}{\partial x^{i}}\, dx^i \wedge dz^a,
\]
since $\Psi^*(dy^a)=\frac{1}{2\pi}dz^a$. This proves the first statement. 
Since $\iota$ is $f$-calibrated if and only if 
$\cP (\beta_{\iota_*(T_b B)})=0$ 
for any $b \in B$, we also obtain the second statement. 
\end{proof}

\begin{remark} \label{rem:FM gives new connection}
The correspondence in Proposition \ref{prop:FM correspondence} is often referred to as a ``mirror map''. The underlying definitions, and in particular the equation $\cP (F_\nabla)=0$ which appears in Proposition \ref{prop:FM general}, rely on the product structure of $B\times T^r$ and the notion of real Fourier--Mukai transform. 
In some instances this equation can be formulated using alternative geometric structures, and can thus be extended beyond this setting. 
In this situation, even though no specific correspondence is defined, solutions of the equation are often still called ``mirror'' connections. $G_2$-instantons (see Section \ref{sec:mirror}), $\Sp$-instantons 
and Hermitian--Yang--Mills connections are examples of this phenomenon. 

In this sense, the main role of the real Fourier--Mukai transform is simply to facilitate the discovery of new gauge-theoretic equations.
\end{remark}

As a final step, consider on $M$ the adiabatic rescaling
\[
g_\varepsilon \;=\; g|_H + \varepsilon g|_V,\qquad \varepsilon\to 0.
\]
The vertical directions correspond to torus fibers with radius of order $\sqrt{\varepsilon}$. The dual metric, on the dual torus fibers, rescales by the factor $\varepsilon^{-1}$ so dual fiber lengths (in particular, the radius of the dual tori) rescale by  $\varepsilon^{-1/2}$. This feature of radius inversion is part of the Buscher rules  \cite{Buscher1988PathIntegralQuantumDuality} describing T-duality; see also \cite{CavalcantiGualtieri2010GCGTDuality}. 

In particular, the adiabatic regime $\varepsilon\to 0$ corresponds, on the mirror (dual) side,
to a large-radius regime in which the dual torus fibers become arbitrarily large.

In Section \ref{sec:mirror} this large-radius regime will be realized by rescaling the ambient calibrated background geometry (thus the metric), leading to the familiar large-radius limits of the associated deformed gauge equations.

\part{Applications to $G_2$ geometry.}\label{part2}

\section{$G_2$-manifolds with associative distributions} \label{sec:G2algebra}

We now want to apply the results of Part \ref{part1} to a 7-dimensional manifold $M$ endowed with a $G_2$-structure $\varphi \in \Omega^3(M)$. As described below, this provides a Riemannian metric with respect to which $\varphi$ is a semi-calibration. The choice of a calibrated distribution $H$ then leads us to the data required by Section \ref{sec:relative_distr}. Specifically:
\begin{itemize}
\item $(M,g)$ is a (3+4)-dimensional Riemannian manifold.
\item $\varphi$ is a semi-calibration $3$-form.
\item $H$ is a $3$-dimensional distribution calibrated, thus oriented, by $\varphi$.
\end{itemize}
It follows that Theorem \ref{thm:from cali to rel cali} provides an adiabatic semi-calibration (see Definition \ref{def:ad cal}), anisotropic with respect to the vertical energy function. In turn, this generates a corresponding notion of adiabatic calibrated immersions.

\subsection{Review of $G_2$ geometry.}\label{ss:G2review}
Let us start with a quick review of the basic facts. 
For more details see, for example, \cite[Section 11]{Joyce}. 

Define a 3-form $\varphi_0$ on $\rl^7$ by 
\begin{align} \label{eq:varphi0}
\varphi_0 = dx_{123} + dx_{145} + dx_{167} + dx_{246}-dx_{257}-dx_{347}-dx_{356}, 
\end{align}
where $dx_{ijk}=dx_i \wedge dx_j \wedge dx_k$. 
The standard inner product $g_0$ on $\rl^7$, 
the volume form $\vol_0 := dx_{1 \cdots 7}$ and $\varphi_0$ are related by 
\begin{equation} \label{eq:form-1def}
g_0 (u,v)\; \vol_0 = \dfrac16 i(u) \varphi_0 \wedge i(v) \varphi_0 \wedge \varphi_0 
\end{equation}
for $u,v \in \rl^7$. 
It is known that the stabilizer at $\varphi_0$ of the $\GL (7, \rl)$-action on 
$\Lambda^3 (\rl^7)^*$ is the exceptional Lie group $G_2$. 
The Lie group $G_2$ canonically acts on each $\Lambda^k (\rl^7)^*$, leading to its decomposition into irreducible representations. In particular: 
\begin{equation} \label{eq:DiffForm-R7}
\begin{array}{rlrl}
\Lambda^1 (\rl^7)^* =& \Lambda^1_7 (\rl^7)^*, \\
\Lambda^2 (\rl^7)^* =& \Lambda^2_7 (\rl^7)^* \oplus \Lambda^2_{14} (\rl^7)^*, \\
\Lambda^3 (\rl^7)^* =& 
\Lambda^3_1 (\rl^7)^* \oplus \Lambda^3_7 (\rl^7)^* \oplus \Lambda^3_{27} (\rl^7)^*,
\end{array}
\end{equation}
where $\Lambda^k_\ell (\rl^7)^* \subseteq \Lambda^k (\rl^7)^*$ 
is $\ell$-dimensional. The Hodge isomorphism $\ast: \Lambda^k (\rl^7)^* \to \Lambda^{7-k} (\rl^7)^*$ (defined by $g_0$ and $\vol_0$)
is $G_2$-equivariant, so we see 
$\Lambda^k_\ell (\rl^7)^* \cong \Lambda^{7-k}_\ell (\rl^7)^*$, allowing us to complete the above list of decompositions. Furthermore, there are the following $G_2$-equivariant isometries
$\lambda^k: (\rl^7)^* \to \Lambda^k_7 (\rl^7)^*$ for $k=2,4,6$: 
\begin{align} \label{eq:lambdas G2}
\lambda^2(\alpha) := \frac{1}{\sqrt{3}} i (\alpha^\sharp) \varphi_0, \qquad
\lambda^4(\alpha) := \frac{1}{2} \alpha \wedge \varphi_0, \qquad
\lambda^6(\alpha) := \ast \alpha. 
\end{align} 
Together with the Hodge isomorphisms, we thus obtain explicit isomorphisms between all 7-dimensional representations in the above decompositions.

The pair $(M,\varphi)$ of a 7-manifold $M$ and a 3-form $\varphi \in \Omega^3(M)$ 
is called \textbf{a manifold with a $G_2$-structure} if 
there are linear isomorphisms $T_x: T_x M \to \rl^7$ such that 
$T_x^* \varphi_0 = \varphi_x$. 
A 7-manifold $M$ admits a $G_2$-structure if and only if 
its frame bundle reduces to a $G_2$-subbundle. 
Hence, considering the associated subbundles, 
we see that 
a Riemannian metric $g=g_\varphi$ 
and the volume form $\vol=\vol_\varphi$ 
are induced from a $G_2$-structure $\varphi$ 
by \eqref{eq:form-1def} and 
$\Lambda^* T^* M$ has the same decomposition as in \eqref{eq:DiffForm-R7}: 
$$
\Lambda^k T^*M = \oplus_\ell \Lambda^k_\ell T^*M, \qquad
\Omega^k_\ell (M) := \Gamma (\Lambda^k_\ell T^*M). 
$$
The volume form $\vol=\vol_\varphi$ agrees with the volume form induced by $g=g_\varphi$: 
$\vol_\varphi=\vol_{g_\varphi}$. 

The $G_2$-structure $\varphi$ is a semi-calibration with the equality property 
in the sense of Definition \ref{def:equality_property}.  
Indeed, define $\chi \in \Omega^3(M, TM)$ by 
\begin{align} \label{eq:def chi}
g (\chi (v_1, v_2, v_3), v_4) = (* \varphi) (v_1, v_2, v_3, v_4)
\qquad 
\mbox{for}
\quad v_1, \dots, v_4 \in TM. 
\end{align}
Then we have the following \textbf{associator equality} 
\begin{align*}
|\varphi (v_1, v_2, v_3)|^2+|\chi(v_1, v_2, v_3)|^2
=|v_1\wedge v_2 \wedge v_3|^2
\qquad 
\mbox{for any}
\quad v_1, v_2, v_3 \in TM. 
\end{align*}
This implies that $\varphi$  is a semi-calibration. Calibrated planes (or submanifolds) corresponding to $\varphi$ are called \textbf{associative} planes (or submanifolds). Up to orientation, associative 3-planes correspond to the condition $\chi|_\pi=0$. The orthogonal complements of associative planes are known as \textbf{coassociative} planes.

\begin{remark}
Coassociative planes can alternatively be viewed as follows. 

The Hodge dual $*\varphi$ is also a semi-calibration with the equality property. 
Indeed, define $\tau \in \Omega^4(M, TM)$ by 
\begin{align} \label{eq:def tau}
\tau = \varphi \wedge \id_{TM},  
\end{align}
where we consider the identity map $\id_{TM}$ of $TM$ 
as a $TM$-valued 1-form via $\Gamma(T^*M \otimes TM)=\Omega^1(M, TM)$. 
Then we have the following \textbf{coassociator equality} 
\begin{align*}
|*\varphi (v_1, v_2, v_3, v_4)|^2+|\tau (v_1, v_2, v_3, v_4)|^2
=|v_1\wedge v_2 \wedge v_3 \wedge v_4|^2
\qquad 
\mbox{for any}
\quad v_1, \dots, v_4 \in TM. 
\end{align*}
This implies that $* \varphi$ is also a semi-calibration. Up to orientation, coassociative planes are exactly those calibrated by $*\varphi$. They can alternatively be characterized via the condition $\tau|_\pi=0$ or the condition $\varphi|_\pi=0$. The corresponding submanifolds are known as coassociative submanifolds.
\end{remark} 

Let us define the \textbf{cross product} 
$\times: TM \times TM \to TM$ by 
\begin{align}
g (v_1 \times v_2, v_3) = \varphi (v_1, v_2, v_3)
\qquad 
\mbox{for}
\quad v_1, v_2, v_3 \in TM. 
\end{align}
In particular, assume $\pi$ is an associative plane. One can then check that the cross product restricts as follows: 
\begin{equation}\label{eq:cross}
\pi \times \pi \to \pi, \ \ \pi \times \pi^\perp \to \pi^\perp, \ \ \pi^\perp\times\pi^\perp\to\pi.
\end{equation}

The following is well-known.

\begin{lemma}\label{lem:associative}
Given any two linearly independent vectors $v_1, v_2$ in $TM$, the span of $\{v_1,v_2,v_1\times v_2\}$ is an associative 3-plane. In particular, any 2-plane is contained in a unique associative 3-plane.
\end{lemma}

Any closed manifold $M$ endowed with a $G_2$-structure admits an associative distribution. Indeed, \cite{Thomas} proves the existence of two linearly independent vector fields. As in Lemma \ref{lem:associative}, their cross product generates a third vector field such that their span is a 3-dimensional associative distribution (see also \cite[Lemma 1]{AS}). Since its orthogonal complement $V$ is a coassociative distribution, such distributions also exist.

\begin{remark}
As usual, the forms $\varphi$, $* \varphi$ are calibrations (rather than semi-calibrations) if they are closed. Both cases define interesting classes of $G_2$-manifolds and are the focus of much research.

The $G_2$-structure is \textbf{torsion-free} if 
$d \varphi=0$ and $d * \varphi=0$. It is known that,
for a Riemannian 7-manifold $(M, g)$, 
the holonomy group $\Hol (g)$ is contained in $G_2$ if and only if 
there is a torsion-free $G_2$-structure $\varphi$ such that $g=g_\varphi$.
\end{remark}

\subsection{The decompositions.}
Let $M$ be a manifold with a $G_2$-structure $\varphi$. Let us fix a $\varphi$-calibrated distribution $H$: in other words, $H_x \subseteq T_x M$ is an associative subspace for any $x \in M$. As in Section \ref{sec:relative_distr}, the $G_2$-structure $\varphi$ and its dual $*\varphi$ can be decomposed as follows. 

\begin{lemma} \label{lem:G2 str dist}
Let $\ula \in \Gamma (\Lambda^3 H^*)$ and $\umu \in \Gamma (\Lambda^4 V^*)$ 
be the volume forms on $H$ and $V$ induced from $g|_H$ 
and $g|_V$, respectively. Let us consider them as forms on $M$ by extending them to zero in the orthogonal directions. Then 
\begin{align*}
\varphi=\ula + \uom, \qquad 
* \varphi=\uth + \umu 
\end{align*}
for some $\uom \in \Gamma (H^* \otimes \Lambda^2 V^*)$ and 
$\uth \in \Gamma (\Lambda^2 H^* \otimes \Lambda^2 V^*)$. 
\end{lemma}

This notation follows \cite{Donaldson}. 
In the notation of Section \ref{sec:relative_distr}, 
$\alpha=\varphi$, $\alpha_0=\ula$, $\alpha_2=\uom$, 
and $\alpha_i=0$ for the other $i$.

\begin{proof}
Fix $x \in M$. 
Recall that $T_x M = H_x \oplus V_x$. 
Since $H_x$ is associative and $V_x$ is coassociative, we may assume that 
\begin{align} \label{eq:varphi *varphi at p}
\varphi_x= e^{123} + \sum_{i=1}^3 e^i \wedge \omega_i, 
\qquad
* \varphi_x = \eta^{4567} + \sum_{k \in \mathbb{Z}/3} e^{k, k+1} \wedge \omega_{k+2}, 
\end{align}
where 
$\{ e^1, e^2, e^3 \}$ 
is a positive orthonormal basis of $H_x^*$ and  
$$
\omega_1=\eta^{45} + \eta^{67}, \qquad  
\omega_2=\eta^{46} + \eta^{75}, \qquad  
\omega_3=-(\eta^{47} + \eta^{56})   
$$
for a positive orthonormal basis $\{ \eta^4, \eta^5, \eta^6, \eta^7 \}$ 
of $V_x^*$. 
Since $\ula_x=e^{123}$ and 
$\umu_x=\eta^{4567}$, 
$\uom:=\varphi-\ula$ and $\uth:=* \varphi- \umu$ satisfy 
\begin{align} \label{eq:uom uth at p}
\uom_x=\sum_{i=1}^3 e^i \wedge \omega_i, 
\qquad
\uth_x=\sum_{k \in \mathbb{Z}/3} e^{k, k+1} \wedge \omega_{k+2}. 
\end{align}
It follows that $\uom \in \Gamma (H^* \otimes \Lambda^2 V^*)$ and 
$\uth \in \Gamma (\Lambda^2 H^* \otimes \Lambda^2 V^*)$. 
\end{proof}

Again following Section \ref{sec:relative_distr}, we can also decompose $\chi$ as 
$$
\chi=\sum_{i=1}^3 \chi_i, 
\qquad \mbox{for} \qquad
\chi_i \in \Gamma (TM \otimes \Lambda^{3-i} H^* \otimes \Lambda^i V^*). 
$$
Fix $x \in M$ and 
let $\{ e_1, e_2, e_3 \}$ and $\{ \eta_4, \eta_5, \eta_6, \eta_7 \}$ 
be positive orthonormal bases of $H_x$ and $V_x$, respectively. 
Then 
\begin{align} \label{eq:chii at p}
\chi_1=- \sum_{a=4}^7 \eta_a \otimes i(\eta_a) \uth, \qquad 
\chi_2=- \sum_{i=1}^3 e_i \otimes i(e_i) \uth, \qquad
\chi_3=- \sum_{a=4}^7 \eta_a \otimes i(\eta_a) \umu.
\end{align}

The components of this decomposition can be characterized in the following two, alternative, ways.

\begin{definition} \label{def:Fueter plane}
Take any $\pi \in G_H^+(3,M)$, 
where $G_H^+(3,M)$ is defined in Definition \ref{def:pHPsubsp}. Set

$$
\mathcal{A}_H (\pi) := \sum_{i=1}^3 p_H \left( v_{i}\right) \times p_V \left( v_{i}\right) 
\in V, 
$$
where $\{ v_i \}_{i=1}^3$ is a positive orthonormal basis of $(\pi, g^H_\pi)$ and we use \eqref{eq:cross}.

For reasons which will become clear with Definition \ref{def:Fueter G2}, we shall refer to this map $\mathcal{A}_H:G_H^+(3,M)\rightarrow V$ as the \textbf{adiabatic calibrated operator}.
\end{definition}

Notice that $\mathcal{A}_H (\pi)$ is independent of the choice of the basis $\{ v_i \}_{i=1}^3$ 
because $\mathcal{A}_H (\pi)$ is given by the contraction of tensors $p_H \left( \cdot \right) \times p_V \left( \cdot \right)$ and $\sum_{i=1}^3 v_i \otimes v_i$, which is the dual of $g^H_\pi$.

\begin{proposition} \label{prop:chiiv at p}
Take any $\pi \in G_H^+(3,M)$. 
Let $\{ v_i \}_{i=1}^3$ be a positive orthonormal basis of $(\pi, g^H_\pi)$. 
Setting $\vec v=v_1 \wedge v_2 \wedge v_3$, we have 
\begin{align*}
\chi_1(\vec v) = \mathcal{A}_H (\pi), \qquad
\chi_2(\vec v) = - \sum_{i=1}^3 g\left( \mathcal{A}_H (\pi), p_V (v_i) \right) p_H (v_i), \qquad
\chi_3(\vec v) = \sum_{a=4}^7 \umu (\vec v, \eta_a) \eta_a. 
\end{align*}
\end{proposition}

\begin{proof}
Since $p_H|_\pi:\pi \to H_x$ is an orientation-preserving isomorphism for $x \in M$, 
$\{ e_i:= p_H(v_i) \}_{i=1}^3$ is a positive orthonormal basis of $H_x$. 
Letting $\{ e^i \}_{i=1}^3$ denote its dual, we can describe $\uth$ 
as in \eqref{eq:uom uth at p}. 
Then using $v_i=e_i+p_V(v_i)$ we compute 
\begin{align}\label{eq:computation uth v}
\begin{split}
\uth_x \left(\vec v, \cdot\right) 
=& 
\omega_1(p_V(v_1), \cdot) 
+ \omega_2(p_V(v_2), \cdot) 
+ \omega_3(p_V(v_3), \cdot) \\
&+ \left \{ 
\omega_2 (p_V(v_1), p_V(v_2)) 
+ \omega_3 (p_V(v_1), p_V(v_3)) 
\right \} e^1 \\
&+ \left \{ 
-\omega_1 (p_V(v_1), p_V(v_2)) 
+ \omega_3 (p_V(v_2), p_V(v_3)) 
\right \} e^2 \\
&+ \left \{ 
-\omega_1 (p_V(v_1), p_V(v_3)) 
- \omega_2 (p_V(v_2), p_V(v_3)) 
\right \} e^3 \\
=&
\sum_{i=1}^3 \varphi_x (e_i, p_V(v_i), \cdot) 
- 
\sum_{i,j=1}^3 \varphi_x (e_i, p_V(v_i), p_V(v_j)) e^j \\
=&
g_x (\mathcal{A}_H(\pi), \cdot) 
- 
\sum_{j=1}^3 g_x (\mathcal{A}_H(\pi), p_V(v_j)) e^j. 
\end{split}
\end{align}
We then obtain the equations for $\chi_1 (\vec v)$ and $\chi_2 (\vec v)$ 
by \eqref{eq:chii at p}.
The equation for $\chi_3 (\vec v)$ is immediate.
\end{proof}

We can alternatively describe $\chi_1(\vec v), \chi_2(\vec v), \chi_3(\vec v)$
as follows.

\begin{proposition}\label{prop:beta}
Use the notation of Proposition \ref{prop:chiiv at p} and its proof. 
Assume $\pi \in G_H^+(3,M)$ is identified with 
$\beta_\pi=g \left((p_V|_\pi) \circ (p_H|_\pi)^{-1}, \cdot \right) \in H^* \otimes V^* 
\subseteq \Lambda^2 T^*M$, as in Proposition \ref{lem:unori hp hom}.
Then 
\begin{align*}
\chi_1(\vec v)^\flat
&=*(\beta_\pi \wedge * \varphi)
= \sqrt{3} (\lambda^2)^{-1} \left(\pi^2_7 (\beta_\pi) \right), \\
\chi_2(\vec v)^\flat 
&= -2 (\lambda^4)^{-1}
\left(\pi^4_7 \left(\frac{1}{2}\beta_\pi^2 \right)\right), \\
\chi_3(\vec v)^\flat 
&= -* \left(\frac{1}{6} \beta_\pi^3 \right) 
= -(\lambda^6)^{-1}  
\left(\frac{1}{6} \beta_\pi^3 \right), 
\end{align*}
where $\pi^k_\ell: \Lambda^k T^*M \rightarrow \Lambda^k_\ell T^*M$ 
is the canonical projection 
and $\lambda^k$ is defined as in \eqref{eq:lambdas G2}. 
\end{proposition}
\begin{proof}
Since $\beta_\pi=\sum_{i=1}^3 e^i \wedge p_V(v_i)^\flat \in H^* \otimes V^*$, 
we compute 
$$
*(\beta_\pi \wedge * \varphi)
=
-\sum_{i=1}^3 i(e_i) i (p_V(v_i)) \underbrace{** \varphi}_{=\varphi}
=
\sum_{i=1}^3 \varphi(e_i, p_V(v_i), \cdot)
= 
\chi_1(\vec v)^\flat. 
$$
Since 
$\lambda^2(e^i) \in \Lambda^2 H^* \oplus \Lambda^2 V^*$ 
for $i=1,2,3$ 
and 
$\lambda^2(\eta^a) \in H^* \otimes V^*$ 
for $a=4,\dots,7$, we have 
$\la \beta_\pi, \lambda^2(e^i) \ra=0$ and 
\begin{align*}
\sqrt{3} \la \beta_\pi, \lambda^2(\eta^a) \ra
=
\left\la \beta_\pi, -\sum_{j=1}^3 e^j \wedge i(\eta_a) \omega_j 
\right\ra
=
\sum_{j=1}^3 \omega_j (p_V(v_j), \eta_a). 
\end{align*}
So 
$\sqrt{3} \pi^2_7(\beta_\pi)
=\sum_{j,a} \omega_j (p_V(v_j), \eta_a) \lambda^2(\eta^a)
=\sum_{a} g(\mathcal{A}_H (\pi), \eta_a) \lambda^2(\eta^a)
= \lambda^2(\mathcal{A}_H (\pi)^\flat)$ 
by \eqref{eq:computation uth v}. 

Since 
$\frac{1}{2} \beta_\pi^2
=-\frac{1}{2} \sum_{i,j=1}^3 e^{i j} \wedge p_V(v_i)^\flat \wedge p_V(v_j)^\flat 
=-\sum_{1\leq i<j \leq 3} e^{i j} \wedge p_V(v_i)^\flat \wedge p_V(v_j)^\flat 
\in \Lambda^2 H^* \otimes \Lambda^2 V^*,
$
$\lambda^4(e^i) \in \Lambda^2 H^*_x \otimes \Lambda^2 V^*_x$ 
for $i=1,2,3$ 
and 
$\lambda^4(\eta^a)
\in \left(H^*_x \otimes \Lambda^3 V^*_x \right)
\oplus \left(\Lambda^3 H^*_x \otimes V^*_x \right)
$
for $a=4,\dots,7$, we have 
$\la \beta_\pi^2, \lambda^4(\eta^a) \ra=0$. 
For $\ell=1,2,3$, 
we compute 
\begin{align*}
2 \left\la \frac{1}{2} \beta_\pi^2, \lambda^4(e^\ell) \right\ra 
=&
-\sum_{1 \leq i<j \leq 3} 
\left\la e^{i j} \wedge p_V(v_i)^\flat \wedge p_V(v_j)^\flat,  
e^\ell \wedge \sum_{k=1}^3 e^k \wedge \omega_k \right\ra\\
=&
-\omega_{\ell+1} (p_V(v_\ell), p_V(v_{\ell+1}))
+\omega_{\ell+2} (p_V(v_{\ell +2}), p_V(v_\ell)) \\
=&
\sum_{k=1}^3 \omega_k (p_V(v_k), p_V(v_\ell)) \\
=& -g (\chi_2 (\vec v), e_\ell). 
\end{align*}
Since 
$\frac{1}{6} \beta_\pi^3
=-e^{1 2 3} 
\wedge p_V(v_1)^\flat \wedge p_V(v_2)^\flat \wedge p_V(v_3)^\flat$, 
we compute 
\begin{align*}
\left\la \frac{1}{6} \beta_\pi^3, \lambda^6(\eta^a) \right\ra
= 
-\left\la p_V(v_1)^\flat \wedge p_V(v_2)^\flat \wedge p_V(v_3)^\flat \wedge \eta^a, 
\umu \right\ra 
=
-\umu (p_V(v_1), p_V(v_2), p_V(v_3), \eta_a). 
\end{align*}
\end{proof}

\subsection{The complex viewpoint.}\label{ss:G2alternative}

One of the key features of a $G_2$-structure $\varphi$ is that, given any associative plane $H_x$ at $x\in M$, it produces (up to a constant factor) an isometric identification between $H_x$ and the space of self-dual 2-forms on its coassociative complement $V_x:=H_x^\perp$:
$$h\in H_x\mapsto \varphi(h,\cdot,\cdot)|_{V_x}.$$
If $h$ is of unit length, we obtain by duality a complex structure $J_h$ on $V_x$:
$$g(J_h(\cdot),\cdot)|_{V_x}=\varphi(h,\cdot,\cdot)|_{V_x}.$$
The standard pointwise identities in $G_2$ geometry give 
\[
J_hJ_k+J_kJ_h=-2g(h,k)\id_{V_x}
\qquad
\text{for all }h,k\in H_x.
\]
We may conclude that positive orthonormal frames 
$\{e_1,e_2,e_3\}$ on $H_x$ are in one-to-one correspondence with ordered triples $\{J_1,J_2,J_3\}$ of $g|_{V_x}$-orthogonal complex structures on $V_x$ 
such that each $J_i$ induces the orientation on $V_x$ and $J_1 J_2 J_3=\id_{V_x}$.
This is similar to the usual quaternionic relations, but with the opposite sign convention: $J_1J_2=-J_3$, $J_2J_3=-J_1$, etc.

We may extend this construction to $\pi\in G_H^+(3,M)$ by associating to any positive orthonormal basis $\{ v_i \}_{i=1}^3$ of $(\pi, g^H_\pi)$ the complex structures $J_i$ on $V_x$ defined by the basis $\{ e_i:= p_H(v_i) \}_{i=1}^3$ on $H_x$. More explicitly,
$$
g(J_i(\cdot), \cdot)|_{V_x}=\varphi(p_H(v_i),\cdot,\cdot)|_{V_x}=\uom(p_H(v_i),\cdot,\cdot)|_{V_x}.
$$
Letting $\{ e^i \}_{i=1}^3$ denote the dual basis, we can describe $\uom$ 
as in \eqref{eq:uom uth at p}. 
Then 
$J_1= \eta^4 \otimes \eta_5 - \eta^5 \otimes \eta_4 
+ \eta^6 \otimes \eta_7 - \eta^7 \otimes \eta_6$, 
etc. In terms of matrices this becomes
\begin{align} \label{eq:cpx strs}
J_1=
\begin{pmatrix}
 & -1 & & \\
1 &  & & \\
 &  & & -1\\
 &  & 1& \\
\end{pmatrix}, 
\qquad
J_2=
\begin{pmatrix}
 & & -1 & \\
 &  & & 1\\
1 &  & & \\
 & -1 & & \\
\end{pmatrix}, 
\qquad
J_3=
\begin{pmatrix}
 &  & & 1\\
 &  & 1& \\
 &  -1& & \\
-1&  & & \\
\end{pmatrix}. 
\end{align}
We now compute 
\begin{align*}
g(\mathcal{A}_H (\pi), \cdot)
=
\sum_{i=1}^3 \varphi \left( 
p_H \left( v_{i}\right), p_V \left( v_{i}\right), \cdot
\right)
= 
\sum_{i=1}^3 \uom \left( 
p_H \left( v_{i}\right), p_V \left( v_{i}\right), \cdot
\right)
=
\sum_{i=1}^3 g \left( J_i \left( p_V \left( v_{i}\right) \right), \cdot \right). 
\end{align*}
Hence $\mathcal{A}_H (\pi) = \sum_{i=1}^3 J_i \left( p_V \left( v_{i}\right) \right)$. 

It is useful to rewrite this in terms of coordinates. Set $p_V(v_i)=\sum_{a=4}^7 v_{i a} \eta_a$. Then
$$
\mathcal{A}_H (\pi) = \sum_{i=1}^3 J_i \left( p_V \left( v_{i}\right) \right)
\cong
\begin{pmatrix}
-v_{15}-v_{26}+v_{37} \\
v_{14}+v_{27}+v_{36} \\
-v_{17}+v_{24}-v_{35} \\
v_{16}-v_{25}-v_{34}\\
\end{pmatrix}.
$$
This proves the following analogue of Lemma \ref{lem:associative}.

\begin{lemma}\label{lem:Fueter}
Let $v_1,v_2\in T_xM$ be such that $\{ e_i:=p_H(v_i)\}_{i=1,2}$ form an orthonormal pair. Set $e_3:=e_1\times e_2$, and let $J_i:=J_{e_i}$ for $i=1,2,3$. Then the equation
\[ \sum_{i=1}^3 J_i\bigl(p_V(v_i)\bigr)=0 \]
admits a unique solution $v_3\in T_xM$ satisfying $p_H(v_3)=e_3$. In particular, any horizontally projectable $2$-plane is contained in a unique $\pi\in G_H^+(3,M)$ such that $\mathcal{A}_H(\pi)=0$.
\end{lemma}

\begin{proof}
Since $J_3^2=-\id_V$, the equation uniquely determines $p_V(v_3)= J_3\left( J_1\left(p_V(v_1)\right) + J_2\left(p_V(v_2)\right) \right)$. Together with $p_H(v_3)=e_3$, this uniquely determines $v_3$.
\end{proof}

\subsection{The adiabatic calibration.}
Using Lemma \ref{lem:G2 str dist} we can rephrase Theorem \ref{thm:from cali to rel cali} as follows.
\begin{theorem} \label{thm:G2 rel cali}
For any $\pi \in G_H^+(3,M)$, we have 
\begin{align}
\uom|_\pi \leq ve (\pi) \vol^H_\pi. 
\end{align}
Thus $\uom$ is an anisotropic semi-calibration with respect to the function
$f(\pi)=ve_1 (\pi) \frac{\vol^H_\pi}{\vol_\pi}$. 
\end{theorem}

Adopting the terminology of Definition \ref{def:ad cal}, we will say that $\uom$ is an adiabatic semi-calibration, and that the corresponding planes inside $G_H^+(3,M)$ are called adiabatic calibrated planes. Using equation \eqref{eq:chii at p}, Propositions \ref{prop:chiiv at p} and \ref{prop:beta}, and 
Corollary \ref{cor:weighted eq prop dist}, 
we can describe them in several equivalent ways.

\begin{corollary} \label{cor:vanishing eq vanishing uth}
Given $\pi \in G_H^+(3,M)$, the following conditions are equivalent:
\begin{enumerate}
\item $\pi$ is $f$-calibrated by $\uom$, where 
$f(\pi)=ve_1(\pi)\frac{\vol_\pi^H}{\vol_\pi}$;
\item $\mathcal{A}_H(\pi)=0$;
\item $\chi_1|_\pi=0$, i.e. $\pi$ is 1-vanishing;
\item For any $\vec v \in \Lambda^3 \pi$, $\uth (\vec v, \cdot)=0$;
\item $\beta_\pi\in \Lambda^2_{14}$, i.e., $\beta_\pi \wedge * \varphi=0$;
\item $\beta_\pi \wedge \uth=0$.
\end{enumerate}
\end{corollary}

The equivalence of (5) and (6) follows from 
$* \varphi=\uth + \umu$ and 
$\beta_\pi \wedge \umu=0$, since $\beta_\pi \in H^* \otimes V^*$. 

\begin{remark}
By Proposition \ref{prop:chiiv at p}, 
$\pi \in G_H^+(3,M)$ is 1-vanishing if and only if it is 2-vanishing. 
Then by \eqref{eq:k vanishing eq} and the fact that $\alpha_4=\alpha_6=0$ for $\alpha=\varphi$, 
we see that 
$$
ve_2(\pi)=0, \qquad ve_3(\pi)=\frac{1}{2} |\chi_3 (\vec v)|^2 \geq 0
$$
for each adiabatic calibrated plane $\pi$. 
As we see below in Proposition \ref{prop:chi3 intersect H}, 
$ve_3(\pi)=0$ if and only if $\pi \cap H \neq \{ 0 \}$. 
\end{remark}

\subsection{($\star$) The calibrated Grassmannians.}

Recall that, pointwise, $G(3,M)$ can be identified with $SO(7)/(SO(3)\times SO(4))$ so it is smooth and has dimension 12.

As mentioned, the calibration $\varphi$ defines the subset of associative planes inside $G(3,M)$. It is known that, at any point, the group $G_2$ acts transitively on pairs of orthonormal vectors and hence, by Lemma \ref{lem:associative}, on associative planes. Using this one can show that, at any point, the associative Grassmannian can be identified with $G_2/SO(4)$. In particular, it is a smooth manifold of dimension 8.

The Grassmannian of adiabatic calibrated planes is contained in $G_H^+(3,M)$. This is an open subset of $G(3,M)$, so pointwise it has dimension 12. 
According to Lemma \ref{lem:Fueter}, adiabatic calibrated planes are uniquely determined by a choice of vectors $v_1,v_2$ whose horizontal projections are orthogonal. In practice, this corresponds to the choice of coordinates $v_{14},v_{15},v_{16},v_{17}, v_{24},v_{25},v_{26},v_{27}$. It follows that, pointwise, the Grassmannian of adiabatic calibrated planes is a smooth manifold of dimension 8.

\ 

The following result traces a relationship between these two Grassmannians.

\begin{proposition} \label{prop:chi3 intersect H}
Take any $\pi \in G_H^+(3,M)$. Let $\{ v_i \}_{i=1}^3$ be a positive orthonormal basis of $(\pi, g^H_\pi)$. Set $\vec v=v_1 \wedge v_2 \wedge v_3$. Then 
\begin{enumerate}
\item $\chi_3(\vec v) =0$ if and only if $\pi \cap H \neq \{ 0 \}$. 
\item 
When $\chi_3(\vec v) =0$,  
$\pi \in G_H^+(3,M)$ is associative 
with an appropriate orientation if and only if it is adiabatic calibrated. 

\item 
If $\pi$ is both associative and adiabatic calibrated with an appropriate orientation, 
then $\pi \cap H \neq \{ 0 \}$. 
\end{enumerate}
\end{proposition}

\begin{proof}
(1) By Proposition \ref{prop:chiiv at p}, $\chi_3(\vec v) =0$ if and only if 
$0=\vec v_3=p_V(v_1) \wedge p_V(v_2) \wedge p_V(v_3)$, where we use the notation in Lemma \ref{prop:adi eq eps}. 
This implies that $\{ p_V(v_i) \}_{i=1}^3$ is linearly dependent. 
Hence there exist $c_1,c_2,c_3 \in \rl$, at least one of which is nonzero, 
such that $0=\sum_{i=1}^3 c_i p_V(v_i)$. 
Thus $0 \neq \sum_{i=1}^3 c_i v_i \in \pi \cap H$. 
The converse is immediate. 

(2) Recall that $\pi \in G_H^+(3,M)$ is associative 
with an appropriate orientation if and only if 
$0=\chi (\vec v)=\chi_1 (\vec v) + \chi_2 (\vec v)$. 
By Proposition \ref{prop:chiiv at p}, 
$\chi_1 (\vec v) \in V$, $\chi_2 (\vec v) \in H$, 
and $\chi_1 (\vec v)=0$ implies $\chi_2 (\vec v)=0$. 
Hence we see the statement. 

(3) Under the assumption of (3), 
$\pi$ satisfies $\chi_1|_\pi=0$ and $\chi|_\pi=0$, 
which is equivalent to $\chi_1|_\pi=0$ and $(\chi_2+\chi_3)|_\pi=0$. 
By Proposition \ref{prop:chiiv at p}, $\chi_2|_\pi$, $\chi_3|_\pi$ 
are $H$, $V$-valued, respectively. So $\chi_2|_\pi=0$ and $\chi_3|_\pi=0$. 
Then by Proposition \ref{prop:chi3 intersect H} (1), $\pi \cap H \neq \{ 0 \}$. 
\end{proof}

\section{The $G_2$ case: Adiabatic calibrated immersions}\label{sec:G2Fueter}

Corollary \ref{cor:vanishing eq vanishing uth} allows us to describe the immersions calibrated by $\uom$ in several different ways. We formalize this as follows.

\begin{definition} \label{def:Fueter G2}
Let $(M,\varphi)$ be a manifold with a $G_2$-structure and let $H$ be an associative distribution. Let $\Sigma$ be a $3$-dimensional oriented manifold.

We say that a positive horizontally projectable immersion $\iota:\Sigma\to M$ is an \textbf{adiabatic calibrated immersion} if each 
$\iota_*(T_\sigma \Sigma) \in G_H^+(3,M)$ satisfies any of the conditions of Corollary \ref{cor:vanishing eq vanishing uth}. 
\end{definition}

This is of course compatible with our previous Definition \ref{def:ad cal}.

Corollary \ref{cor:vanishing eq vanishing uth} also suggests defining 
\begin{align}
\beta_\iota:=\beta_{\iota_*(T \Sigma)}
=\sum_{i=1}^3 \iota^H_* \left( e_{i}\right)^\flat \wedge  
\iota^V_* \left( e_{i}\right)^\flat
\in \Gamma(\iota^* (H^* \otimes V^*)) \subseteq \Gamma (\iota^* \Lambda^2 T^*M). 
\end{align}
Then $\iota$ is adiabatic calibrated if and only if $\beta_\iota \in \Gamma (\iota^* \Lambda^2_{14} T^*M)$, i.e. $\beta_\iota \wedge (* \varphi \circ \iota) =0 \in \Gamma (\iota^* \Lambda^6 T^*M)$. 

\ 

Let us now list some properties of adiabatic calibrated immersions.

\subsection{Minimization properties.}

Recall Definition \ref{def:K happy} 
and functionals in Section \ref{sec:functionals}. 
Corollary \ref{cor:rel cali minimize VE} can be written as follows.

\begin{corollary} \label{cor:fueter minimize VE}
Assume $\Sigma$ is compact. If $d\uom=0$ 
then any adiabatic calibrated immersion minimizes 
the vertical energy functional $VE$ in its restricted homology class.
\end{corollary}

For example, assume $H$ is integrable with compact leaves. Each leaf is then associative and is an absolute minimizer because $VE=0$. Any other adiabatic calibrated submanifold in the same restricted class must also have $VE=0$. This implies it is horizontal and hence associative.

We refer to Section \ref{sec:G2 case examples} for examples of manifolds with a $G_2$-structure and an associative distribution such that $d \uom=0$.

As in Corollary \ref{cor: minimize VE+VolH}, we can also see the following. 

\begin{corollary} \label{cor:fueter minimize EV}
Assume $\Sigma$ is compact. If $d\varphi=0$ 
then any adiabatic calibrated immersion minimizes $VE+\Vol^H$ in its restricted homology class.
\end{corollary}

\subsection{Adiabatic limits and mirror symmetry} \label{sec:mirror}

According to Section \ref{sec:rel ad lim} 
(in particular, in \eqref{eq:adi lim general}), 
the adiabatic calibrated equation can be viewed as 
the ``adiabatic limit'' of the associative equation.

In addition, for any $\varepsilon >0$, $\varphi_\varepsilon$ defined as in 
\eqref{eq:alpha eps g eps} is a $G_2$-structure 
which induces the metric $g_\varepsilon=g|_H + \varepsilon g|_V$.
The Hodge dual $*_\varepsilon \varphi_\varepsilon$ of $\varphi_\varepsilon$ 
with respect to $g_\varepsilon$ is given by 
$$
*_\varepsilon \varphi_\varepsilon = \varepsilon \uth + \varepsilon^2 \umu. 
$$
Defining $\chi_\varepsilon$ as in \eqref{eq:def chi} one finds that, for any $\pi \in G(3,M)$, 
\begin{align*}
\pi \text{ is associative with respect to } \varphi_\varepsilon
\quad &\Longleftrightarrow \quad
\chi_\varepsilon|_\pi = 0\\
\quad &\Longleftrightarrow \quad
(*_\varepsilon \varphi_\varepsilon)(\vec v,\cdot)=0 \text{ for any } \vec v \in \Lambda^3 \pi
\end{align*}
if $\varphi_\varepsilon|_\pi>0$. 
We see that the limit of the equation 
$(*_\varepsilon \varphi_\varepsilon) (\vec v, \cdot)/\varepsilon=0$
as $\varepsilon \to 0$ 
is 
$\uth (\vec v, \cdot)=0$, 
which is the 1-vanishing equation when $\pi \in G_H^+(3,M)$ 
by Corollary \ref{cor:vanishing eq vanishing uth}. 

Combining the above discussion with \eqref{eq:adi lim general}, 
we obtain the following diagram. 
$$
\xymatrix{
\varphi_\varepsilon |_\pi=\vol_{\varepsilon, \pi} 
\ar@{<=>}[r]
\ar@{-->}[d]_-{\text{ad. lim.}} 
&
\chi_\varepsilon|_\pi=0
\ar@{<=>}[r]
\ar@{-->}[d]_-{\text{ad. lim.}} 
&
(*_\varepsilon \varphi_\varepsilon) (\vec v, \cdot)=0
\ar@{-->}[d]_-{\text{ad. lim.}} 
&
\mbox{if} \quad \varphi_\varepsilon |_\pi>0
\ar@{-->}[d]_-{\text{ad. lim.}}
\\
\uom|_\pi = ve_1 (\pi) \vol^H_\pi
\ar@{<=>}[r]
&
\chi_1|_\pi=0
\ar@{<=>}[r]
&
\uth (\vec v, \cdot)=0
&
\mbox{if} \quad \ula|_\pi >0
}
$$
for any $\vec v \in \Lambda^3 \pi$. 
Here, as in \eqref{eq:adi lim general}, “ad. lim.” denotes the corresponding adiabatic limits, including the appropriate normalizations seen above.

Adiabatic calibrated immersions also have an interesting interpretation in the context of mirror symmetry, via Section \ref{sec:mirror general}.

\begin{theorem} \label{thm:FM Fueter G2inst}
Let $B\subset\mathbb{R}^3$ be a contractible open set. 
For a smooth map $f:B \to T^4$, define the immersion (section) 
$$
\iota:B \to B \times T^4, \qquad x \mapsto (x,f(x)). 
$$ 
Let $\nabla$ be 
the real Fourier--Mukai transform of $f$. 
On $B\times (T^4)^*$, let $\varphi^*$ denote the standard $G_2$-structure written in the coordinates $(x^1,x^2,x^3,y^4,\dots,y^7)$, and let $\uth^*$ be the corresponding $4$-form. Then the following are equivalent. 
\begin{enumerate}
\item 
$\iota$ is an adiabatic calibrated immersion 
with respect to the standard $G_2$-structure on $B \times T^4$ 
and the standard horizontal distribution $H$ induced by $B$ on $B \times T^4$. 

\item 
$F_\nabla \wedge * \varphi^* =0$ on $B\times (T^4)^*$, 
i.e. $\nabla$ satisfies the $G_2$-instanton equation. 

\item 
$F_\nabla \wedge \uth^* =0$ on $B\times (T^4)^*$.

\end{enumerate}
\end{theorem}
The equivalence of (1) and (3) follows from Proposition \ref{prop:FM general}, Definition \ref{def:Fueter G2}, and Corollary \ref{cor:vanishing eq vanishing uth}. The equivalence of (2) and (3) follows from that of (5) and (6) in Corollary \ref{cor:vanishing eq vanishing uth}.

Notice that the $G_2$-instanton equation makes sense on any $G_2$-manifold. The above theorem thus provides an instance of the situation mentioned in Remark \ref{rem:FM gives new connection}. Adopting that perspective, $G_2$-instantons can be considered as the ``mirror'' of adiabatic calibrated submanifolds.

\begin{remark} 
\cite{Li} provides a similar statement in the context of torsion-free $G_2$-manifolds fibered by K3 surfaces (see also Example \ref{ex:Fueter_vs_Walpuski} below). However, that statement is based on a different notion of duality, thus a different transform.

In the torus-bundle case discussed here, we use the correspondence ``a point on the torus $\leftrightarrow$ a flat line bundle on the dual torus'' on each torus fiber. Assembling this fiberwise corresponds to the real Fourier--Mukai transform, which turns a graphical submanifold into a connection on the dual fibration. 
Now consider the case of a K3 fibration. In certain situations it is known that the moduli space of HYM connections on a K3 surface is again a K3 surface: this defines the dual K3 and the correspondence ``a HYM connection on the fiber $\leftrightarrow$ a point on the dual fiber''. Assembling this fiberwise gives a transform that turns a family of HYM connections, parametrized by the base, into a section (hence a submanifold) of the dual fibration.

As mentioned, our condition (1) in Theorem \ref{thm:FM Fueter G2inst} coincides with the adiabatic associative sections defined in \cite[(7)]{Li}. Our condition (3) is called the adiabatic $G_2$-instanton equation in \cite[(8)]{Li}. 
It is shown in \cite[Theorems 1.2 and 1.6]{Li} that adiabatic $G_2$-instantons are equivalent to adiabatic associative sections (i.e., in our terminology, to adiabatic calibrated submanifolds). 

The equivalence of our conditions (2) and (3)
is due to the fact that $-\i F_\nabla \in \Gamma (H^* \otimes V^*)$. 
In the K3 fibration case this condition is not always satisfied, 
and only the equivalence of (1) and (3) is mentioned in \cite{Li}.
\\

A second instance of the situation discussed in Remark \ref{rem:FM gives new connection} is presented in \cite{LL}, which shows that the ``mirror'' of associative submanifolds is a certain class of connections, known as deformed Donaldson--Thomas (dDT) connections. They are defined by the equation 
$$(1/6)F_\nabla^3+F_\nabla\wedge*\varphi=0.$$
Let us now recall from \cite{Kobs} the following fact. Using a parameter $r\in\mathbb{R}^+$, we can rescale the $G_2$-structure as follows: $r\mapsto \varphi_r:=r^3\varphi$. The dDT equation then rescales as 
$$(1/6)F_\nabla^3+r^4 F_\nabla\wedge*\varphi=0,$$
so the ``large radius limit'' of the dDT equation---obtained by sending $r\to\infty$ and retaining only the leading-order terms in $r$---is the $G_2$-instanton equation.

\ 

We may summarize the above via the following diagram. 
\begin{align} \label{eq:asso dDT Fueter G2inst}
\begin{split}
\xymatrix{
\{ \mbox{associative submanifolds} \} \ar@{<->}[r]_-{mirror} 
\ar@{-->}[d]_-{adiabatic \ limit}
& \{ \mbox{dDT connections} \} 
\ar@{-->}[d]^-{large \ radius \ limit} \\
    \{ \mbox{adiabatic calibrated submanifolds} \} \ar@{<->}[r]_-{mirror}  
    & \{ \mbox{$G_2$-instantons} \}
}
\end{split}
\end{align}
\end{remark}
It is interesting to observe how the geometries differ on the two sides of the ``mirror''.
On the connection side, all notions (the large radius limit, the dDT equation and the $G_2$-instanton equation) are formulated intrinsically, without introducing any auxiliary structure.
On the submanifold side, by contrast, our definitions require a splitting of the tangent bundle, i.e. a distribution: adiabatic limits are taken with respect to such a choice, and the definition of adiabatic calibrated immersions also depends on it.
In particular, if we want to realise the mirror map via the real Fourier--Mukai transform, the vertical distribution must not only be integrable, but must actually arise as the vertical distribution of a fibration.

\subsection{($\star$) Further remarks on adiabatic limits and decompositions}
We collect here a few additional calculations and comments.

\ 

1. In Section \ref{sec:var cha Fueter} we will consider the equation $d \uth=0$. We shall now show that it is the adiabatic limit of 
the coclosed condition on $\varphi$.

Set $\Omega^{p,q}=\Gamma (\Lambda^p H^* \otimes \Lambda^q V^*)$. 
Notice that $d \Omega^{1,0}, d \Omega^{0,1} \subseteq \Omega^2 
= \Omega^{2,0} \oplus \Omega^{1,1} \oplus \Omega^{0,2}$. 
Furthermore, if $\alpha \in \Omega^{0,1}$ and $X_1, X_2 \in \Gamma (H)$ then
$$
(d \alpha)(X_1, X_2)
=-\alpha ([X_1, X_2]), 
$$
so the $(2,0)$-part of $d \Omega^{0,1}$ is given by an operator which is algebraic with respect to $\alpha$. 
The same computation holds for the $(0,2)$-part of $d \Omega^{1,0}$. 

More generally, by the graded Leibniz rule for $d$, 
the exterior derivative 
$d$ splits into types
\begin{align} \label{eq:deri decomp}
d = F_H + d_H +d_V + F_V: \Omega^{p,q} \to 
\Omega^{p+2,q-1} \oplus \Omega^{p+1,q} \oplus \Omega^{p,q+1} \oplus \Omega^{p-1,q+2},
\end{align}
where $F_H$ and $F_V$ are algebraic operators.

Since $\uom \in \Omega^{1,2}$ and $\ula \in \Omega^{3,0}$, we compute 
\begin{align*}
d \varphi_\varepsilon
= \varepsilon d \uom + d \ula
= \varepsilon \left( F_H \uom +d_H \uom +d_V \uom + F_V \uom \right)
+ 
\left( d_V \ula + F_V \ula \right). 
\end{align*}
Comparing degrees, we see that 
$d \varphi_\varepsilon=0$ if and only if 
\begin{align} \label{eq:d varphi eps 0}
\varepsilon F_H \uom + d_V \ula=0, \qquad
\varepsilon d_H \uom + F_V \ula=0, \qquad
d_V \uom=0, \qquad
F_V \uom=0. 
\end{align}
So the limit of the closedness equation above as $\varepsilon \to 0$ 
is 
\begin{align} \label{eq:d varphi adi 0}
d_V \ula=0, \qquad
F_V \ula=0, \qquad
d_V \uom=0, \qquad
F_V \uom=0. 
\end{align}

Next, we consider the coclosedness condition. 
Since $\uth \in \Omega^{2,2}$ and $\umu \in \Omega^{0,4}$, 
we compute 
\begin{align*}
d *_\varepsilon \varphi_\varepsilon
= \varepsilon d \uth + \varepsilon^2 d \umu
= \varepsilon \left( d_H \uth +d_V \uth + F_V \uth \right)
+ 
\varepsilon^2 \left( F_H \umu + d_H \umu \right). 
\end{align*}
Comparing degrees, we see that 
$d *_\varepsilon \varphi_\varepsilon=0$ if and only if 
\begin{align} \label{eq:d *varphi eps 0}
d_H \uth =0, \qquad
d_V \uth + \varepsilon F_H \umu=0, \qquad
F_V \uth + \varepsilon d_H \umu=0. 
\end{align}
So the limit of the coclosedness equation above as $\varepsilon \to 0$ 
is 
\begin{align} \label{eq:d *varphi adi 0}
d_H \uth =0, \qquad
d_V \uth =0, \qquad
F_V \uth =0, 
\end{align}
which is equivalent to $d \uth=0$.

\begin{remark}\label{rem:limit comp with K3}
The limits \eqref{eq:d varphi adi 0} and \eqref{eq:d *varphi adi 0} 
are slightly different from those in \cite[Lemma 5]{Donaldson}. 
In \cite{Donaldson}, it is assumed that the manifold admits a K3 fibration, so the vertical distribution $V$ is integrable. 
This shows that $d \Omega^{1,0} \subseteq \Omega^{2,0} \oplus \Omega^{1,1}$ 
and $F_V=0$ in \eqref{eq:deri decomp}. 

Then \eqref{eq:d varphi eps 0} becomes 
$\varepsilon F_H \uom + d_V \ula=0$, $d_H \uom=0$, $d_V \uom=0$, 
and we obtain 
$$
d_V \ula=0, \qquad d_H \uom=0, \qquad d_V \uom=0
$$
as $\varepsilon \to 0$. These agree with the first three equations 
in \cite[Lemma 5]{Donaldson}. 
The equation \eqref{eq:d *varphi eps 0} becomes 
$d_H \uth =0$, $d_V \uth + \varepsilon F_H \umu=0$, $d_H \umu=0$, 
and we obtain 
$$
d_H \uth =0, \qquad d_V \uth=0, \qquad d_H \umu=0
$$
as $\varepsilon \to 0$. These agree with the last two equations 
in \cite[Lemma 5]{Donaldson} and the equation in \cite[Lemma 6]{Donaldson}. 
\end{remark}

2. If $d \ula = 0$, then $V$ is involutive (hence integrable).
Indeed, for $X, Y \in \Gamma(V)$ and $h_1, h_2 \in \Gamma(H)$, we have 
$$
(d \ula)(X, Y, h_1, h_2) = - \ula( p_H([X,Y]), h_1, h_2 ).
$$
If $d \ula = 0$, we thus get $\ula( p_H([X,Y]), h_1, h_2 ) = 0$ for all $h_1, h_2$. Since $\ula|_H$ is a (nondegenerate) volume form on $H$, 
we have $p_H([X,Y])=0$. This implies that $V$ is integrable. 

The condition $d \uom = 0$ has no such consequence on $V$. Subsection \ref{ex:gen qH} provides manifolds with a $G_2$-structure 
admitting an associative distribution
such that $d \uom = 0$, but $V := \ker(e^1,e^2,e^3)$ is not integrable. 
Indeed, the structure equation \eqref{eq:structure-B} implies that 
$V$ is not integrable.

\begin{remark} 
We can generalize this as follows.
Let us define 
$\mathcal{F}_V\in\Gamma(\Lambda^2V^*\otimes H)$ and 
$\mathcal{F}_H\in\Gamma(\Lambda^2H^*\otimes V)$ 
by $\mathcal{F}_V(X,Y)=p_H([X,Y])$ for $X,Y\in\Gamma(V)$ and $\mathcal{F}_H(h_1,h_2)=p_V([h_1,h_2])$ for
$h_1,h_2\in\Gamma(H)$.  Then $V$ is involutive if and only if $\mathcal{F}_V=0$, and $H$ is involutive if and only if
$\mathcal{F}_H=0$.

By the definition of the algebraic pieces $F_H$ and $F_V$ in \eqref{eq:deri decomp}, we have the identities
\[
F_V\underline{\lambda}=-i(\mathcal{F}_V)\underline{\lambda}\in\Omega^{2,2},
\qquad
F_H\underline{\mu}=-i(\mathcal{F}_H) \underline{\mu}\in\Omega^{2,3},
\]
where $i(\mathcal{F}_V) \ula$ means: apply $\mathcal{F}_V:\Lambda^2V\to H$, then contract the
resulting $H$-vector into $\underline{\lambda}$ (and similarly for $i(\mathcal{F}_H) \underline{\mu}$).
Since contraction with a volume form is a fiberwise isomorphism,
\[
H\to\Lambda^2H^*,\quad h\mapsto i(h) \underline{\lambda},
\qquad
V\to\Lambda^3V^*,\quad v\mapsto i(v) \underline{\mu},
\]
we obtain the equivalences
\begin{align*}
\mbox{(2,2)-component of } d \ula
&=
F_V\underline{\lambda}=0 \Longleftrightarrow \mathcal{F}_V=0 \Longleftrightarrow V \ \text{is involutive (hence integrable)},\\
\mbox{(2,3)-component of } d \umu
&=
F_H\underline{\mu}=0 \Longleftrightarrow \mathcal{F}_H=0 \Longleftrightarrow H \ \text{is involutive (hence integrable)}.
\end{align*}
\end{remark}

\subsection{($\star$) A variational characterization} \label{sec:var cha Fueter}
There also exists a variational characterization of adiabatic calibrated submanifolds. To describe this, let us first recall the variational characterization of associative immersions.

Let $M$ be a manifold with a $G_2$-structure $\varphi$ 
and $\Sigma$ be a compact oriented 3-manifold. 
Let $\Imm$ be the space of immersions $\Sigma \to M$. 
Fixing $\iota_0 \in \Imm$, the universal cover $\wImm$ of $\Imm$ is given by 
$$
\wImm= \{ I:[0,1] \to \Imm \mid I(0)=\iota_0 \}/(\mbox{homotopy with fixed ends}). 
$$
The covering map $\pi_{\rm Imm}: \wImm \to \Imm$ is given by 
$\pi_{\rm Imm} ([I])=I(1)$. 
When $d * \varphi=0$, we have the well-defined functional 
$$
CS_{asso}: \wImm \to \rl, \qquad 
[I] \mapsto \int_{[0,1] \times \Sigma} I^* (*\varphi), 
$$
where we consider $I$ as a map $I:[0,1] \times \Sigma \to M$. 
It is known that 
$[I]$ is a critical point of $CS_{asso}$ if and only if 
$I(1)$ is an associative immersion. \\

The characterization of adiabatic calibrated immersions is analogous: as in Section \ref{sec:mirror}, it suffices to define a functional by replacing $* \varphi$ with $\uth$.

\begin{theorem}\label{thm:Fueter CS}
Let $M$ be a manifold with a $G_2$-structure $\varphi$ with an associative distribution $H$. Let $\Sigma$ be a compact oriented 3-manifold and 
$$
\wHPI^+
= \{ I:[0,1] \to \HPI^+ \mid I(0)=\iota_0 \}/(\mbox{homotopy with fixed ends})
$$
be the universal covering of $\HPI^+$, where we fix $\iota_0 \in \HPI^+$. 

If $d \uth=0$, 
$$
CS: \wHPI^+ \to \rl, \qquad 
[I] \mapsto \int_{[0,1] \times \Sigma} I^* \uth  
$$
is a well-defined functional, and 
$[I]$ is a critical point of $CS$ if and only if 
$I(1)$ is an adiabatic calibrated immersion. 
\end{theorem}

Note that a similar characterization is given in \cite[Section 2.3]{Li}
when $M$ admits the structure of a coassociative submersive fibration 
$M \to B$ over a contractible base $B$ 
and $\HPI^+$ is replaced with the space of sections of $M \to B$. 

We refer to Section \ref{sec:G2 case examples} for examples of manifolds with a $G_2$-structure and an associative distribution such that $d \uth=0$.

\begin{proof}
We first show that $CS$ is invariant under the homotopy with fixed ends. 
Take a family 
$\{ I_s: [0,1] \to \HPI^+ \}_{s \in [0,1]}$ such that 
$I_s(0)$ and $I_s(1)$ are fixed for any $s \in [0,1]$. 
Set $\cI: [0,1] \times [0,1] \times \Sigma \to M$ by 
$\cI (s,t,\sigma)=(I_s(t))(\sigma)$. 
Then by $d \uth =0$, we compute 
\begin{align*}
0=\int_{[0,1] \times [0,1] \times \Sigma} \cI^* d \uth 
= \int_{\partial ([0,1] \times [0,1] \times \Sigma)} \cI^* \uth
=& 
\int_{\{1 \} \times [0,1] \times \Sigma} \cI^* \uth
- 
\int_{\{0 \} \times [0,1] \times \Sigma} \cI^* \uth \\
=&
\int_{[0,1] \times \Sigma} I_1^* \uth
- 
\int_{[0,1] \times \Sigma} I_0^* \uth.
\end{align*}

Next, we consider the variation of $CS$. 
Fix $I: [0,1] \to \HPI^+$ and take a family 
$\{ I_s: [0,1] \to \HPI^+ \}_{s \in (- \varepsilon, \varepsilon)}$ 
with $I_0=I$ and $I_s(0)=I(0)=\iota_0$ for any $s$. 
Set 
$$
\widetilde Z= \left. \frac{d I_s}{ds} \right|_{s=0} \in \Gamma (I^* TM). 
$$
Then by a version of Cartan's formula in \cite[Remark 2.1]{Solomon}
and the assumption that $d \uth=0$, we have  
\begin{align*}
\left. \frac{d}{ds} I_s^* \uth \right|_{s=0}
=
d I^* \left(  i(\widetilde Z) (\uth \circ I)  \right). 
\end{align*}
So setting $i_t: \Sigma \hookrightarrow [0,1] \times \Sigma$ by 
$i_t(\sigma)=(t,\sigma)$ 
and 
$Z=\widetilde Z \circ i_1= \left. \frac{d I_s}{ds} (1) \right|_{s=0} 
\in \Gamma (I(1)^* TM)$, 
we have 
\begin{align*}
\left. \frac{d}{ds} \int_{[0,1] \times \Sigma} I_s^* \uth \right|_{s=0}
=
\int_\Sigma i_1^* I^* \left(  i(\widetilde Z) (\uth \circ I)  \right)
- 
\int_\Sigma i_0^* I^* \left(  i(\widetilde Z) (\uth \circ I)  \right)
=
\int_\Sigma I(1)^* \left(  i(Z) (\uth \circ I(1))  \right), 
\end{align*}
where we use 
$\widetilde Z \circ i_0 = \left. \frac{d I_s}{ds} (0) \right|_{s=0}=0$. 
Since $G_H^+(3,M) \subseteq G(3,M)$ is open, every sufficiently small variation of the endpoint immersion is again in $\HPI^+$, 
and hence $Z$ is arbitrary.
So 
$\left. \frac{d}{ds} \int_{[0,1] \times \Sigma} I_s^* \uth \right|_{s=0}=0$ 
if and only if 
$$
\uth \left( I(1)_*(x), I(1)_*(y), I(1)_*(z), \cdot \right)=0
$$ 
for any $x,y,z \in T \Sigma$. 
By Corollary \ref{cor:vanishing eq vanishing uth}, 
this is equivalent to $I(1)$ being adiabatic calibrated. 
\end{proof}

\section{The $G_2$ case: Fueter maps}\label{s:Fueter maps}

The goal of this section is to describe a relationship between adiabatic calibrated immersions and Fueter maps. As mentioned in the Introduction, the definition of Fueter map is formulated in different ways depending on the specific context. Two common ingredients are (i) the use of complex structures $\{J_1,J_2,J_3\}$ on the target bundle/manifold satisfying the quaternionic relations, (ii) a fixed metric $g_\Sigma$ on the domain $\Sigma$. The latter condition will mark an important difference with respect to the previous sections, where we instead relied on the pullback (horizontal) metric $g_\iota^H$.

To begin, let us describe the most elementary setting in which Fueter maps appear in $G_2$ geometry, following \cite{Donaldson} and \cite{Li}. This leads to the relationship in its simplest form.

\begin{example}\label{ex:Fueter_vs_Walpuski}Fix a manifold $(M,\varphi)$ with $G_2$-structure such that:
\begin{itemize}
\item $\varphi$ is torsion-free (or even just closed).
\item $M$ admits the structure of a Riemannian submersion $\pi:M\rightarrow B$, as in Section \ref{ss:integrable dist}.
\item Each fiber is a compact coassociative submanifold. 
\end{itemize}

In this case each fiber is known to admit a natural hyperk\"ahler structure \cite{Baraglia}. In particular, it has a natural orientation. Set $V:=\ker \pi_*$ and $H:=V^\perp$, so that $H$, with the orientation induced from $TM$ and $V$, is an associative distribution on $M$. It follows that $B$ also admits a natural orientation.

Let $g_B$ denote the metric on $B$ and $\{J_1,J_2,J_3\}$ denote the three complex structures on the fibers. We shall use the convention that $J_1J_2=-J_3$. As described in Section \ref{ss:G2alternative}, these data can be described, at each point, in terms of a positive $g_B$-orthonormal basis 
$\{e_1,e_2,e_3\}$ on $B$ (equivalently on $H$, because $\pi$ is a Riemannian submersion), with dual basis $\{e^1,e^2,e^3\}$: the $G_2$-structure can be written as
\[
\varphi=\pi^*(e^{123})+\sum_{i=1}^3 \pi^*(e^i)\wedge\omega_i,
\qquad \omega_i(\cdot,\cdot)=g(J_i\cdot,\cdot)\ \text{ on }V,
\]
where each self-dual 2-form $\omega_i$ on $V$ is extended to $M$ by setting it to 0 on $H$. 

Now set $\Sigma=B$ and, as above, let $s:\Sigma \to M$ be a section of $\pi: M \to B$. By the decomposition of $\varphi$, and using the fact that $s_*^H(e_i)$ is the horizontal lift of $e_i$ along $s$, we have 
$s_*^H(e_i)\times(\cdot)=J_i$ on $V$.
Define the \textbf{Fueter operator} on sections as
\[
\mathcal{F}(s):=\sum_{i=1}^3 J_i\bigl(s_*^V(e_i)\bigr)=\sum_{i=1}^3 s_*^H(e_i)\times s_*^V(e_i). 
\]
The relationship with adiabatic calibrated immersions is now obvious: according to Definition \ref{def:Fueter G2}, $\mathcal{F}(s)=0$ if and only if $s$ is an adiabatic calibrated immersion.
\end{example}

\begin{remark}
Since $TM=H\oplus V$, the distribution $H$ induces a covariant derivative $\nabla$ on sections $s:B\to M$: given $v\in T_bB$, it suffices to define 
\[
\nabla_v s := p_V\bigl(ds(v)\bigr)\in V_{s(b)}. 
\]
This construction corresponds to the fact that $H$ defines a connection on $M$ in the sense of Ehresmann. 
We can then, equivalently, write $\mathcal{F}(s):=\sum_{i=1}^3 J_i\bigl(\nabla_{e_i}s\bigr)$.
\end{remark}

The relationship described in Example \ref{ex:Fueter_vs_Walpuski} relies on the $G_2$ fact that we can use the horizontal metric to control complex structures on $V$ (see Section \ref{ss:G2alternative}). The additional fact that $\pi$ is a Riemannian submersion eliminates any difference between the metric on $\Sigma=B$ and the horizontal metric on its image. In other geometric situations, see for example \cite[Definition~1.1]{Walp}, the definition of $\mathcal{F}$ (in particular, of $J_i$) requires a preliminary choice of appropriate identifications. In our setting this takes the following form.

\begin{definition}\label{def:Fueter immersion}

Let $(M,\varphi)$ be a manifold with a $G_2$-structure and let $H$ be an associative distribution. Let $(\Sigma,g_\Sigma)$ be a $3$-dimensional oriented Riemannian manifold.

Given a map $\iota:\Sigma\rightarrow M$, assume $\iota^*H$ is trivial so that (using the Gram-Schmidt construction) there exists an orientation-preserving fiberwise isometry $\mathcal{I}=\mathcal{I}(\iota):(T\Sigma,g_\Sigma)\rightarrow (\iota^*H,g|_H)$. According to Section \ref{ss:G2alternative}, this defines an identification between positive orthonormal bases $\{e_1,e_2,e_3\}$ on $T\Sigma$ and quaternionic structures $(J_1,J_2,J_3)$ on $\iota^*V$ (using our usual convention $J_1J_2J_3=\id_V$).

Choosing such data, we may define the \textbf{Fueter operator} on maps by setting 
\[
\mathcal{F}_{\mathcal I}(\iota)
:=
\sum_{i=1}^3
J_i
\bigl(\iota_*^V(e_i)\bigr).
\]
We will say that $\iota$ is a \textbf{Fueter map} 
if $\mathcal{F}_{\mathcal I}(\iota)=0$.
\end{definition}
The condition that $\iota^*H$ is trivial is necessary, due to the well-known fact that $T\Sigma$ is trivial. It is a topological condition on $H$ which holds, for example, if $\iota$ is a  horizontally projectable immersion (see also Remark \ref{rem:topology}). The Fueter condition is independent of the choice of basis, but does depend on the choice of identifications $\mathcal{I}$ (more precisely, on the pair $(\iota,\mathcal{I})$).

Notice in Definition \ref{def:Fueter immersion}, as expected, the choice of a metric on $\Sigma$. In the special case where $\iota_*^H$ is an isometry with respect to this metric, as in
Example \ref{ex:Fueter_vs_Walpuski} with $\iota=s$, there is a canonical choice $\mathcal{I}(v):=\iota_*^H(v)$, for $v\in T\Sigma$. With this choice, Corollary \ref{cor:vanishing eq vanishing uth} immediately yields the following result, generalizing the conclusion of Example \ref{ex:Fueter_vs_Walpuski}.

\begin{theorem}\label{prop:adiabatic vs F}
Let $(M,\varphi,H)$ be as in the above definitions. Let $\iota:(\Sigma,g_\Sigma)\rightarrow M$ be a positive horizontally projectable immersion such that $\iota_*^H$ is an isometry.
Then $\iota$ is Fueter if and only if it is adiabatic calibrated. 
\end{theorem}

There is a second reason for interest in this result. To explain this, let us first emphasize that the dependence of the Fueter condition on the metric implies that it is not invariant under reparametrization. 
This is very similar to the situation for harmonic maps. Adiabatic calibrated immersions, like minimal immersions, are instead parametrization-independent. The following relationship is classical.

\begin{theorem}\label{prop:minimal vs harmonic}
Consider a map $\iota:(\Sigma,g_\Sigma)\rightarrow (M,g)$ between Riemannian manifolds. Assume it is an isometric immersion. Then $\iota$ is harmonic if and only if it is minimal.
\end{theorem}

Theorem \ref{prop:adiabatic vs F} can thus be seen as a first-order analogue of Theorem \ref{prop:minimal vs harmonic}.

\begin{remark}
In \cite{Walp}, Walpuski studies Fueter sections of a bundle of hyper-K\"ahler manifolds over $B$. In his setting there is usually no background $G_2$-structure, so the identifications are defined between unit vectors in $TB$ and (global) complex structures on the fibers: the global nature of these complex structures allows these identifications to be chosen independently of any section. Our vertical distribution is instead non-integrable, so our complex structures are defined pointwise: this requires our $\mathcal{I}$ to depend on $\iota$. 

In our setting, one might alternatively choose to study the class of maps $\iota:\Sigma\rightarrow M$ defined by the property $\sum_{i=1}^3
\iota_*^H(e_i)\times\iota_*^V(e_i)=0$,
where $\{e_1,e_2,e_3\}$ is a positive $g_\Sigma$-orthonormal basis. This equation is quadratic in $\iota_*$ and becomes degenerate when $\iota^H_*$ loses rank. Compared to the Fueter condition, it has the advantage of not requiring a preliminary choice of identifications.
\end{remark}

\section{The $G_2$ case: Explicit examples} \label{sec:G2 case examples}

The goal of this section is to exhibit explicit examples of (i) manifolds endowed with $G_2$-structures which have the properties discussed in the previous sections, (ii) examples of adiabatic calibrated submanifolds.

\subsection{Product-type hyperk\"ahler model}\label{ss:HK example}

Let $M=T^3\times X$, where $T^3=\rl^3/\Z^3$ and $X$ is hyperk\"ahler with K\"ahler forms $(\omega_1,\omega_2,\omega_3)$. Endow $T^3$ with the standard flat metric $g_{T^3}$ and the standard orientation.
Write $\eta^1,\eta^2,\eta^3$ for the standard closed $1$-forms on $T^3$. With $H=T(T^3)$ and $V=T(X)$, set
\[
\ula=\eta^{123},\qquad 
\uom=\eta^1\wedge\omega_1+\eta^2\wedge\omega_2-\eta^3\wedge\omega_3.
\]
Define $\varphi=\ula+\uom$. Then $d\ula=0$, $d\uom=0$, $d\varphi=0$. In this setting we can apply both Corollary \ref{cor:fueter minimize VE} and Corollary \ref{cor:fueter minimize EV}.

In order to build adiabatic calibrated submanifolds, let us concentrate on the simplest case: $X=T^4=\rl^4/\Z^4$ with the standard, flat, hyperk\"ahler structure. 
Recall that the almost complex structures $J_1,J_2,J_3$ induced on $V$ 
from $\omega_1, \omega_2, -\omega_3$ satisfy
\[
J_i^2=-\id_{TX},\qquad J_1J_2=-J_3,\qquad J_iJ_j+J_jJ_i=0\ \ (i\neq j).
\]
Let $\Sigma:=U\subset\mathbb{R}^3$ with coordinates $x=(x^1,x^2,x^3)$. 
Take a smooth map $\tilde u: U \to \rl^4$. 
The induced map 
\begin{align} \label{eq:imm prod}
\iota : U\to T^3\times T^4, 
\qquad
\iota (x) := (p(x), u(x)) := ([x], [\tilde u(x)])
\end{align}
is then a positive horizontally projectable immersion. 
The Fueter operator associated with $(J_1,J_2,J_3)$ is
\[
\mathcal{F}(\iota)=D \tilde u := J_1\frac{\partial \tilde u}{\partial x^1}+J_2\frac{\partial \tilde u}{\partial x^2}+J_3\frac{\partial \tilde u}{\partial x^3}.
\]
Using the relations above and 
setting $\partial_{i}=\partial/\partial x^i$, we obtain the identity 
\[D^{2}
= \sum_{i} J_i^{2}\,\partial_{i}^{2}
+ \sum_{i<j}\bigl(J_iJ_j + J_jJ_i\bigr)\,\partial_i \partial_j
= -\bigl(\partial_1^{2}+\partial_2^{2}+\partial_3^{2}\bigr)
= -\Delta, 
\]
where $\Delta=\sum_{i=1}^3\partial_i^2$ is the Euclidean Laplacian.
Notice that $D\tilde u=0$ implies $\Delta \tilde u=0$: in this setting, each component of an adiabatic calibrated immersion is harmonic.

\begin{remark}More generally, we could set $p(x)=[cx+d]$, where $c>0$ and $d\in\rl^3$. In this case the Fueter operator would use the metric $g^H_\iota=p^*g_{T^3}=c^2 \sum_{i=1}^3 (dx^i)^2$ to normalize the basis. This simply introduces a constant factor in the above calculations, so again $\iota$ 
is adiabatic calibrated if and only if $D\tilde u=0$.
\end{remark}

\begin{example}[Solutions via harmonic functions]\label{ex:flat harmonic}

Choose a smooth map $F:U\subseteq\rl^3\to \mathbb{R}^4$ such that each component is harmonic. Set

\[
\tilde u :=D(F)= J_1\frac{\partial F}{\partial x^1}+J_2\frac{\partial F}{\partial x^2}+J_3\frac{\partial F}{\partial x^3}:U\rightarrow \rl^4.
\]
Then $D\tilde u=-\Delta F=0$.
In particular, choosing a harmonic polynomial for each component 
gives smooth, globally defined, adiabatic calibrated immersions of $\mathbb{R}^3$.

On $U:=\mathbb{R}^3\setminus\{0\}$ one may instead consider the Newtonian potential $\phi(x)=(4\pi|x|)^{-1}$: then $\tilde u=D(\phi\,v_0)$, 
for $v_0 \in \rl^4$, provides more examples.  

\begin{remark}Recall that harmonic functions, thus adiabatic calibrated immersions, satisfy the maximum principle so one cannot hope to find similar examples on compact manifolds.
\end{remark}
\end{example}

\begin{example}[Affine solutions.]\label{ex:flat affine}

Let $\Sigma=\mathbb{R}^3$.
Let $A\in\mathrm{Mat}_{4\times 3}(\mathbb{R})$ and $b\in\rl^4$. Define $\tilde u:\mathbb{R}^3\to\mathbb{R}^4$ by $\tilde u(x)=Ax+b$.
Writing $a_i=\partial_i\tilde u$ for the columns of $A$, we have 
$$D\tilde u = J_1a_1+J_2a_2+J_3a_3.$$
If $A\in\mathrm{Mat}_{4\times 3}(\mathbb{Z})$, then $\tilde u(x+n)-\tilde u(x)\in\mathbb{Z}^4$ for all $n\in\mathbb{Z}^3$, hence $\tilde u$ descends to a map
$u:T^3=\mathbb{R}^3/\mathbb{Z}^3\to T^4=\mathbb{R}^4/\mathbb{Z}^4$. In particular, 
taking any $a_2,a_3\in\mathbb{Z}^4$ and setting  
$a_1:=-J_3 a_2+J_2 a_3$, we obtain a nontrivial adiabatic calibrated immersion
\[
\iota:T^3\to T^3\times T^4,\qquad \iota(x)=(x,u(x)).
\]
This map is a section of the projection $\pi:T^3\times T^4\to T^3$.
Using $\iota(x)=(cx,u(x))$ with $c\in\Z$ we obtain a multi-valued section, cf. Section \ref{ss:integrable dist}. For example, 
the choice $c=2$ defines an immersed $8$-fold multi-valued section.
\end{example}

\begin{proposition}Any smooth adiabatic calibrated section
\[
\iota:T^3\to T^3\times T^4,\qquad \iota(x)=(x,u(x))
\]
is of affine type, as above.
\end{proposition}
\begin{proof}
Lift $u:T^3\rightarrow T^4$ to a map $\tilde u:\rl^3\rightarrow\rl^4$. For $x \in \rl^3$ and $n\in\mathbb{Z}^3$, the difference $\tilde u(x+n)-\tilde u(x)$ lies in $\mathbb{Z}^4$. Since it is continuous with respect to $x$, it is independent of $x$; denote it by $c(n)\in\mathbb{Z}^4$.
Moreover $c(n+m)=c(n)+c(m)$, so there exists $A\in\mathrm{Mat}_{4\times 3}(\mathbb{Z})$ such that $c(n)=An$ for all $n$.
Set $b:=\tilde u(0)$ and $\tilde v(x):=\tilde u(x)-Ax-b$. Then $\tilde v$ is $\mathbb{Z}^3$-periodic, so it descends to a map $v:T^3\rightarrow\rl^4$. The corresponding section is Fueter, so by the maximum principle it is constant.
\end{proof}

\subsection{Examples via Lie groups.} \label{sec:ex via Lie gp}
Lie groups provide an important source of examples of manifolds with $G_2$-structures, via the following construction.

Let $\fg$ be a 7-dimensional Lie algebra. Choose any basis $e_1,\dots, e_7$. Let $e^1,\dots,e^7$ denote the dual basis. Consider the splitting $\fg=H\oplus V$, where $H$ is generated by $e_i$ for $i=1,2,3$ and $V$ is generated by $e_a$ for $a=4,5,6,7$. Define the 2-forms 
\begin{align} \label{eq:ex Lie gp omega i}
\omega_1=e^{45}+e^{67}, \qquad \omega_2=e^{46}-e^{57}, \qquad \omega_3=-(e^{47}+e^{56}).
\end{align}
Set 
\begin{align}\label{eq:ex Lie gp varphi}
\ula&:=e^{123}, & \uom&:=\sum_{i=1}^3 e^i\wedge\omega_i, 
& \varphi&:=\ula+\uom, \\
\uth&:=\sum_{k \in \mathbb{Z}/3} e^{k, k+1} \wedge \omega_{k+2}, 
& \umu&:=e^{4567}, 
& * \varphi&:=\uth + \umu. \label{eq:ex Lie gp varphi 2}
\end{align}
Now let $G$ be any Lie group whose Lie algebra is $\fg$. The data on $\fg$ then generate left-invariant data on $G$. In particular $M:=G$ is endowed with a $G_2$-structure $\varphi$ such that the $e_i$ are orthonormal, and distributions $H$, $V$ such that $H$ is associative.  The 4-form $*\varphi$ is the Hodge dual of $\varphi$ 
with respect to the induced metric and orientation.

Recall that the Lie algebra structure on $\fg$ is encoded by coefficients $c_{ij}^k$ defined by the equation $[e_i,e_j]=\sum_k c_{ij}^ke_k$. The fact that the Lie bracket is anti-symmetric and satisfies the Jacobi identity corresponds to certain conditions on these coefficients. Given any left-invariant 1-form $\alpha$ on $G$, the Cartan formula shows that $d\alpha(e_i,e_j)=-\alpha([e_i,e_j])$. Applying this to the left-invariant forms $e^k$, one finds $de^k=-\frac12\sum_{i,j}  c_{ij}^k e^i\wedge e^j$. Using the usual rules $de^{ij}=de^i\wedge e^j-e^i\wedge de^j$ etc., analogous formulae can be found for the higher-degree differentials. 

The bottom line is that the structure coefficients of $\fg$ encode the differentials of all left-invariant forms. Conditions such as $d\varphi=0$ or $d\uom=0$ can thus be written as equations on $c_{ij}^k$, in theory producing a complete list of Lie algebras whose Lie groups are endowed with left-invariant such structures. 

\begin{example}\label{ex:semidirect}
The simplest way to produce a 7-dimensional Lie algebra $\fg$ is via the direct sum of a 3-dimensional Lie algebra $\fh$ and a 4-dimensional Lie algebra $\fv$: $\fg=\fh\oplus\fv$. A more sophisticated construction replaces the direct sum with a semidirect product. This requires a representation $\rho:\fh\rightarrow \End(\fv)$. The vector space $\fh\oplus\fv$ is then endowed with the Lie bracket
\begin{align} \label{eq:semi dp bracket}
[X+u,Y+v]=[X,Y]_{\fh}+[u,v]_{\fv}+\rho(X)v-\rho(Y)u, \ \ X,Y\in\fh, \  u,v\in \fv.
\end{align}
The resulting Lie algebra is written $\fh\ltimes_\rho\fv$.
\end{example}

\subsection{Case study: $\SU(2)\ltimes_\rho\mathbb{C}^2$} \label{ex:left-inv semidp}

Let us concentrate on a concrete case of Example \ref{ex:semidirect}. Let $\fh:=\su(2)$ and $\fv:=\mathbb{C}^2$ (viewed as an abelian Lie algebra). Let $\rho:\su(2)\rightarrow \gl(2,\mathbb{C})$ denote the standard representation. Set $\fg:=\fh\ltimes_\rho\fv$.
Choose the basis for $\su (2)$ as follows: 
\begin{align} \label{eq:def ei su(2)}
e_1:= \left(\begin{matrix}0&1\\-1&0\end{matrix}\right), \qquad
e_2:= \left(\begin{matrix}0& \i \\ \i&0\end{matrix}\right), \qquad
e_3:= \left(\begin{matrix} \i &0\\0& -\i \end{matrix}\right).
\end{align}
Then $[e_1,e_2]=2e_3$ etc., so the structure coefficients for $i,j,k\in\{1,2,3\}$ are $c_{ij}^k=0$ if two or more indices coincide, otherwise $c_{ij}^k=\pm 2$ according to the sign of the permutation $(123)\mapsto (ijk)$. 

Under the identification
$\mathbb{C}^2\cong\mathbb{R}^4$ via 
$(z_1,z_2)=(x_1+\i y_1,x_2+\i y_2) \mapsto (x_1,y_1,x_2,y_2)$, 
$\rho$ is given, with respect to the
basis $e_4,e_5,e_6,e_7$ of $\mathbb{R}^4$, by
\[
\rho(e_1)=
\begin{pmatrix}
0&0&1&0\\
0&0&0&1\\
-1&0&0&0\\
0&-1&0&0
\end{pmatrix},
\qquad
\rho(e_2)=
\begin{pmatrix}
0&0&0&-1\\
0&0&1&0\\
0&-1&0&0\\
1&0&0&0
\end{pmatrix},
\qquad
\rho(e_3)=
\begin{pmatrix}
0&-1&0&0\\
1&0&0&0\\
0&0&0&1\\
0&0&-1&0
\end{pmatrix}.
\]
By \eqref{eq:semi dp bracket}, we have 
\[
[e_i,e_a]=\rho(e_i)e_a,\qquad [e_a,e_b]=0
\]
for $i=1,2,3$ and $a,b=4,5,6,7$. As explained in Section \ref{sec:ex via Lie gp}, this implies 
\[
de^1=-2 e^{23},\qquad de^2=-2 e^{31},\qquad de^3=-2 e^{12},
\]
and
\begin{align*}
de^4&=- (e^{16}-e^{27}-e^{35}),&
de^5&=- (e^{17}+e^{26}+e^{34}), \\
de^6&=- (-e^{14}-e^{25}+e^{37}),&
de^7&=- (-e^{15}+e^{24}-e^{36}).
\end{align*}
Using \eqref{eq:ex Lie gp omega i} 
we can directly check that $d\omega_1=d\omega_2=d\omega_3=0$. Using \eqref{eq:ex Lie gp varphi} and \eqref{eq:ex Lie gp varphi 2},
\begin{align*}
d\underline{\lambda}&=0,\\
d \varphi=d \uom&= 
\sum_{i=1}^3 d e^i \wedge\omega_i
= -2 e^{23} \wedge \omega_1
+2 e^{13} \wedge \omega_2-2 e^{12} \wedge \omega_3 
\neq 0, \\
d \umu &=0, \\
d * \varphi = d \uth &=
d(e^{12}) \wedge \omega_3
+ d(e^{23}) \wedge \omega_1
+ d(e^{31}) \wedge \omega_2
=0. 
\end{align*}

An example of a Lie group $G$ with this algebra is given by the semidirect product $\SU(2)\ltimes_\rho\mathbb{C}^2\cong \SU(2)\ltimes_\rho\mathbb{R}^4$. More specifically, $G$ is the manifold $\SU(2)\times\mathbb{R}^4$ endowed with the product
\begin{align} \label{eq:semi dp product}
(h,u)\cdot(k,v):=(hk,u+\rho(h)v).
\end{align}
Our calculations above show, in particular, that the $G_2$-structure on $G$ is coclosed.

We shall endow $\SU(2)\ltimes_\rho\mathbb{R}^4$ with the metric $g_\varphi$ induced from $\varphi$ and $\SU(2)$ with the metric 
$g_{\SU(2)}$ such that 
$\{ e_1,e_2,e_3\} \subseteq \su (2)$ forms an orthonormal frame.
Then 
\begin{align} \label{eq:proj sdp SU(2)}
\pi: (\SU(2)\ltimes_\rho\mathbb{R}^4, g_\varphi) \to (\SU(2), g_{\SU(2)})
\end{align}
is a Riemannian submersion.\\

Let us now construct examples of adiabatic calibrated immersions in $G$. Notice that, in this case, we should not expect them to have minimization properties because neither Corollary \ref{cor:fueter minimize VE} nor Corollary \ref{cor:fueter minimize EV} apply. However, these examples are relevant to Theorem \ref{thm:Fueter CS}.

Since $\SU (2) \cong {\rm Sp}(1)$, we see that 
$\rho(g)^*\omega_i=\omega_i$ for any $g \in \SU (2)$ and $i=1,2,3$, 
where $\omega_i \in \Omega^2(\rl^4)$ is defined as in \eqref{eq:ex Lie gp omega i}. Define $J_1, J_2, J_3 \in \End (\rl^4)$ by 
$\la J_i (\cdot), \cdot \ra= \omega_i$, 
where $\la \cdot, \cdot \ra$ is the standard inner product of 
$\rl^4 \cong \cx^2$. Then $J_1, J_2, J_3$ are given as in \eqref{eq:cpx strs} and
\begin{align} \label{eq:rho g Ji commute}
\rho(g)\circ J_i = J_i \circ \rho(g)
\qquad \mbox{for} \quad g\in \SU(2),\quad i=1,2,3.
\end{align}

As in Section \ref{sec:ex via Lie gp}, 
we extend $H,V, e_1,\dots, e_7, \omega_i, J_i$ to 
left-invariant tensors on $G$. 
By abuse of notation, we shall use the same symbol for the extended tensors. 

Let $I$ denote the identity element in $\SU(2)$, thus $(I,0)$ the neutral element in $G$. According to \eqref{eq:semi dp product}, left translation is given by
$$
(L_{(h,u)})_{* (I,0)} (X,v)
=
\left( (L_h)_{* I} X, \rho (h) v \right)
$$
for $(h,u) \in G$ and $(X,v) \in H_{(I,0)} \oplus V_{(I,0)}$. 
Hence 
\begin{align*}
H_{(h,u)}
&=(L_{(h,u)})_{* (I,0)} H_{(I,0)}
=\left \{ \left( (L_h)_{* I} X, 0 \right) \mid X \in H_{(I,0)}=\mathfrak{su} (2) \right \}
=T_h \SU (2) \times \{ 0 \}, \\
V_{(h,u)}
&=(L_{(h,u)})_{* (I,0)} V_{(I,0)}
=\left \{ \left( 0, \rho (h) v \right) \mid v \in V_{(I,0)}= \rl^4 \right \}
=\{ 0 \} \times \rl^4. 
\end{align*}

Let $U\subseteq \SU(2)$ be open and let $u:U\to \mathbb{R}^4$ be smooth.
Define
\[
\iota:U\longrightarrow G=\SU(2)\ltimes_{\rho}\mathbb{R}^4,
\qquad \iota(g)=(g,u(g)).
\]
Then 
$(\iota_*)_h(\widetilde X)=(\widetilde X, \widetilde X u) 
\in T_{\iota (h)} G$ 
for $h \in U$ and $\widetilde X \in T_h \SU (2)$, and 
\begin{align*}
p_H ((\iota_*)_h(\widetilde X))=(\widetilde X, 0) \in H_{\iota (h)}, 
\qquad
p_V ((\iota_*)_h(\widetilde X))=(0, \widetilde X u) \in V_{\iota (h)}. 
\end{align*}
This confirms that $\iota$ is positive horizontally projectable.

Since \eqref{eq:proj sdp SU(2)} is a Riemannian submersion, 
we have $g^H_\iota=g_{\SU (2)}|_U$, so the adiabatic calibrated condition (Definition \ref{def:Fueter G2}) becomes 
\[
\mathcal{F} (\iota)_h
=\sum_{i=1}^3 (J_i)_{\iota (h)} 
\bigl( p_V ((\iota_*)_h(e_i)) \bigr)
=\sum_{i=1}^3 
\left(
(L_{\iota (h)})_{* (I,0)} \circ (J_i)_{(I,0)} 
\circ 
(L_{\iota (h)})^{-1}_{* \iota (h)} 
\right)
(0, (e_i)_h u)
=0. 
\]
By the equation 
$(L_{\iota (h)})^{-1}_{* \iota (h)} (0, (e_i)_h u)
= 
(0, \rho(h)^{-1} (e_i)_h u)
$
and \eqref{eq:rho g Ji commute}, we see that 
\[
\iota\ \text{is adiabatic calibrated}
\qquad\Longleftrightarrow\qquad
D_{\SU(2)}u:=
\sum_{i=1}^3 (J_i)_{(I,0)} (e_i u)=0 \in C^\infty (U, \rl^4). 
\]
Since $(J_i)_{(I,0)}$ coincides with $J_i \in \End(\mathbb{R}^4)$, we henceforth denote $(J_i)_{(I,0)}$ simply by $J_i$.

\begin{lemma} \label{lem:D2 SU(2)}
We have the identity 
$$
D_{\SU(2)}^2=-\Delta_{\SU(2)}-2 D_{\SU(2)},
$$
where
$\Delta_{\SU(2)}:=\sum_{i=1}^3 e_i^2$
is the Laplace--Beltrami operator with respect to $(U, g^H_\iota)$.
\end{lemma}

\begin{proof}
Since the $J_i \in \End(\mathbb{R}^4)$ are constant endomorphisms of $\mathbb{R}^4$, we compute
\[
\begin{aligned}
D_{\SU(2)}^2u
&=\sum_{i,j=1}^3 J_i J_j \,e_ie_j u \\
&=\sum_{i=1}^3 J_i^2 e_i^2u
 +\sum_{i<j}\Bigl(J_i J_j e_ie_j u+J_j J_i  e_je_i u\Bigr) \\
&=-\sum_{i=1}^3 e_i^2u+\sum_{i<j}J_i J_j [e_i,e_j]u \\
&=-\Delta_{\SU(2)}u-2 D_{\SU(2)}u.
\end{aligned}
\]
\end{proof}

\begin{remark}
More generally,
let $\jmath:\Sigma^3\to G$ be a positive horizontally projectable immersion with $\Sigma$ closed and connected. As seen in Section \ref{ss:integrable dist}, $p:=\pi\circ\jmath:\Sigma\to \SU(2)$ is a covering map. Since $\SU(2)$ is simply connected, it is a diffeomorphism.
Thus, after reparametrization, every such immersion may be written in the graph form considered above.
So, for closed connected domains, the graph ansatz is no restriction.
\end{remark}

\begin{example}[Solutions via harmonic maps]
Let $F:U\to \mathbb{R}^4$ be a smooth map on an open set
$U\subseteq \SU(2)$ such that each component of $F$ is harmonic. Set
$u:=(D_{\SU(2)}+2)F$. According to Lemma \ref{lem:D2 SU(2)}, $D_{\SU(2)}(D_{\SU(2)}+2)=-\Delta_{\SU(2)}$ so
\[
D_{\SU(2)}u
=
D_{\SU(2)}(D_{\SU(2)}+2)F
=
-\Delta_{\SU(2)}F
=
0.
\]
In other words, every $\mathbb{R}^4$-valued local harmonic map $F$ produces a local Fueter solution $u=(D_{\SU(2)}+2)F$. Suitable choices of $F$ produce non-constant solutions.

For example, recall that 
$$
(\SU (2), g_{\SU (2)}) \rightarrow (S^3, g_{S^3}), 
\qquad
h \mapsto 
 h \begin{pmatrix} 1 \\ 0
\end{pmatrix}
$$
is an isometry by \eqref{eq:def ei su(2)}, 
where 
$S^3\subseteq \cx^2$ is the sphere of 
radius $1$ and 
$g_{S^3}$ is the round metric on 
$S^3$. 
According to \cite[Theorem 1, in the case $d=3$]{Cohl2011}, for each
$p\in \SU(2)$ the function
\[
f_p(h):=A \cot\!\bigl(r_p(h)\bigr) +B,
\qquad
r_p(h):=d_{g_{\SU(2)}}(p,h),
\]
is harmonic on $U_p:=\SU(2)\setminus\{p,\bar p\}$, 
where $\bar p$ is the antipodal point of $p$, 
$d_{g_{\SU(2)}}$ is the metric induced from $g_{\SU(2)}$, and $A, B \in \rl$. 
Here $r_p(h)\in(0,\pi)$ on $U_p$. 

In particular, when $A=\frac{1}{4\pi}$ and $B=0$, 
\[
\frac{1}{4\pi} \cot (r)=\frac{1}{4\pi r}+O(r)\qquad (r\to 0),
\]
has the same leading singular term as the Euclidean fundamental
solution in dimension $3$.

Setting 
\[
F_p:=f_p\,v_0:U_p\longrightarrow \mathbb{R}^4
\]
for $0\neq v_0\in\mathbb{R}^4$, 
$u_p:=(D_{\SU(2)}+2)F_p$ 
defines a non-constant Fueter solution on the punctured
domain $\SU(2)\setminus\{p,\bar p\}$.
\end{example}

\begin{remark}
If $u:\SU(2)\to\mathbb{R}^4$ is smooth and satisfies $D_{\SU(2)}u=0$, then
\[
0=D_{\SU(2)}^2u=-\Delta_{\SU(2)}u.
\]
Thus each component of $u$ is harmonic on the compact manifold $\SU(2)$, hence constant. It follows that every adiabatic calibrated immersion of graph type with domain $\SU(2)$ is, up to translation in the $\mathbb{R}^4$--factor, the inclusion of a slice
\[
\SU(2)\times\{v_0\}\subseteq \SU(2)\ltimes_{\rho}\mathbb{R}^4.
\]
In particular, these are the only affine-type solutions.
\end{remark}

\begin{remark}
Another concrete case of Example \ref{ex:semidirect} is Bryant's ``extremally Ricci pinched'' $G_2$-manifold \cite{Bryant}.
Using the above notation, it has the properties $d\varphi=0$, $d\ula=0$, $d\uom=0$. In particular, the $G_2$-structure is closed. On the other hand, $d*\varphi\neq 0$, $d\umu=0$, $d\uth\neq 0$. See, for example, \cite[Example 5.2]{PR} for a description along the above lines.
\end{remark}

\subsection{Case study: Generalized quaternionic Heisenberg groups} \label{ex:gen qH}

Consider $\mathbb{R}^7=\mathbb{R}^3\oplus\mathbb{R}^4$. Similarly to Section \ref{sec:ex via Lie gp}, choose a basis $\{e_1,e_2,e_3,\eta_4,\eta_5,\eta_6,\eta_7\}$. Let $\{e^1,e^2,e^3,\eta^4,\eta^5,\eta^6,\eta^7\}$ denote its dual basis. We can define a Lie algebra bracket $[\cdot,\cdot]$ on $\rl^7$ as follows. Choose any $3\times 3$ real matrix $B$. Consider the $2$-forms
\begin{align} \label{eq:omegai example}
\omega_1=\eta^{45}+\eta^{67},\qquad
\omega_2=\eta^{46}-\eta^{57},\qquad
\omega_3=-(\eta^{47}+\eta^{56}).
\end{align}

Given $y,y'\in\rl^4$, set $[y,y']=-\sum_{i,j=1}^3 B_{ij}\omega_j(y,y')e_i$. Make all other brackets vanish.

In particular, setting $H:=\rl^3$, $V:=\rl^4$, this implies 
$$[H,H]=0, \ \ [H,V]=0, \ \ [V,V]\subseteq H,$$
showing that, with this structure, $\rl^7$ is a 2-step nilpotent Lie algebra.

\begin{example}
Choosing $B=\mathrm{diag}(1,1,-1)$ and using the standard quaternionic conventions, one obtains $[y,y']=\operatorname{Im}(y\bar y')$.
\end{example}

Given any associated Lie group, the Cartan equations for the left-invariant 1-forms defined by $e^i,\eta^j$ imply that
\begin{align} \label{eq:structure-B}
d\eta^4=d\eta^5=d\eta^6=d\eta^7=0,\qquad
de^i=\sum_{j=1}^3 B_{ij}\,\omega_j.
\end{align}
These calculations show that when $B\neq0$ the left-invariant extension of $V$ is not integrable, so it is not the vertical distribution of a fibration. \\

Now define $\ula,\uom,\varphi$ and $\uth, \umu,*\varphi$ in analogy with equations \eqref{eq:ex Lie gp varphi} and \eqref{eq:ex Lie gp varphi 2}:
\[
\underline{\lambda}=e^{123},\qquad
\underline{\omega}=e^1\wedge\omega_1+e^2\wedge\omega_2+e^3\wedge\omega_3,\qquad
\underline{\Theta}=e^{12}\wedge\omega_3+e^{23}\wedge\omega_1+e^{31}\wedge\omega_2,
\qquad \umu=\eta^{4567}.
\]
Then 
$\varphi=\underline{\lambda}+\underline{\omega}$ 
is a $G_2$-structure 
with Hodge dual $*\varphi=\underline{\Theta}+\underline{\mu}$, and $\{e_i, \eta_a\}$ is orthonormal with respect to the induced metric. We compute 
\begin{align*}
d\underline{\mu}&=0,\\
d\underline{\omega}
&=2\,\mathrm{tr}(B)\,\underline{\mu},\\
d\underline{\lambda}
&=de^1\wedge e^{23}-e^1\wedge de^2\wedge e^3+e^{12}\wedge de^3,\\
d\underline{\Theta}
&=2\Bigl(
(B_{32}-B_{23})\,e^1+(B_{13}-B_{31})\,e^2+(B_{21}-B_{12})\,e^3
\Bigr)\wedge \underline{\mu}.
\end{align*}
The closedness conditions thus reduce to algebraic constraints on $B$, i.e. on the structure coefficients, as expected:
\begin{align*}
d\underline{\omega}=0 &\Longleftrightarrow \mathrm{tr}(B)=0,\\
d\underline{\lambda}=0 &\Longleftrightarrow de^1=de^2=de^3=0 \Longleftrightarrow B=0,\\
d(* \varphi)=d\underline{\Theta}=0 &\Longleftrightarrow B={}^t B, 
\end{align*}
where ${}^t B$ is the transpose of $B$. 
Moreover, since $d\underline{\lambda}\in \Lambda^2 H^* \wedge \Lambda^2 V^*$
and $d\underline{\omega}\in \Lambda^4 V^*$, 
we have 
\[
d\varphi=0 \qquad\Longleftrightarrow \qquad 
d\underline{\lambda}=0 \qquad \text{and}\qquad d\underline{\omega}=0 
\qquad\Longleftrightarrow \qquad
B=0. 
\]

Let $N_B=\mathbb{R}^3\times\mathbb{R}^4$ denote the simply-connected Lie group defined by our Lie algebra. Set
\[
\beta_B(y,y'):=B\,
\begin{pmatrix}
\omega_1(y,y')\\
\omega_2(y,y')\\
\omega_3(y,y')
\end{pmatrix}
\in\mathbb{R}^3.
\]
The multiplication law takes the form 
\begin{align} \label{eq:multiplication law}
(x,y)\cdot(x',y')=\Bigl(x+x'-\frac12\,\beta_B(y,y'),\ \ y+y'\Bigr),
\end{align}
where $x=(x^1,x^2,x^3)$ and $y=(y^4,y^5,y^6,y^7)$.

By \eqref{eq:multiplication law},
\[
(L_{(x,y)})_{*(0,0)}(X,Y)
=
\left(
X-\frac12\beta_B(y,Y),\,Y
\right).
\]
Hence the left-invariant distributions determined by
$H$, $V$ (defined above on the Lie algebra) are given by
\begin{align} \label{eq:HV gHb}
H_{(x,y)}
=
\mathbb{R}^3\times\{0\},
\qquad
V_{(x,y)}
=
\left\{
\left(
-\frac12\beta_B(y,Y),\,Y
\right)
\ \middle|\ 
Y\in\mathbb{R}^4
\right\}.
\end{align}
In particular, although
$N_B=\mathbb{R}^3\times\mathbb{R}^4$ as a manifold,
the left-invariant distribution $V$ is generally not the product
distribution $\{0\}\times\mathbb{R}^4$. Indeed, as mentioned, $V$ is integrable
if and only if $B=0$.

If we assume $B\in M_{3\times 3}(2\mathbb{Z})$ is an even integer matrix, so that $\beta_B(\mathbb{Z}^4,\mathbb{Z}^4)\subseteq 2\mathbb{Z}^3$, the subgroup
\[
\Gamma_B:=\mathbb{Z}^3\times\mathbb{Z}^4\subseteq \mathbb{R}^3\times\mathbb{R}^4
\]
is closed under the above multiplication. It follows that $\Gamma_B$ is a discrete cocompact subgroup (a lattice) in $N_B$. We thus obtain compact quotients of $N_B$ by setting
\[
M_B:=\Gamma_B\backslash N_B.
\]

\begin{lemma} \label{lem:H1 nil}
We have 
\[
H_1(M_B;\mathbb{Z})\cong \Gamma_B^{\mathrm{ab}}
\cong \mathbb{Z}^4\oplus\bigl(\mathbb{Z}^3/B\mathbb{Z}^3\bigr), 
\]
where $\Gamma_B^{\mathrm{ab}}$ 
is the abelianization of $\Gamma_B$. 
\end{lemma}

\begin{proof}
Since $H_1(M_B; \mathbb{Z})\cong \pi_1(M_B)^{\mathrm{ab}}$ 
and $\pi_1(M_B)\cong \Gamma_B$, we have 
$H_1(M_B;\mathbb{Z})\cong \Gamma_B^{\mathrm{ab}}
=\Gamma_B/[\Gamma_B,\Gamma_B]
$. 
By the group law \eqref{eq:multiplication law}, 
the subgroup $\mathbb{Z}^3\times\{0\}$ is central, and a direct computation shows that for $y,y'\in\mathbb{Z}^4$ 
\[
[(0,y),(0,y')]=( -\beta_B(y,y'),\,0)\in \mathbb{Z}^3\times\{0\}.
\]
Since $\omega_1,\omega_2,\omega_3$ given in \eqref{eq:omegai example} 
have integer coefficients, we see that the map
$$
\Lambda^2\mathbb{Z}^4\to\mathbb{Z}^3, \qquad
(y,y') \mapsto \left( \omega_1 (y,y'), \ \omega_2 (y,y'), \ \omega_3 (y,y')\right)
$$ 
is surjective. 
Hence 
$[\Gamma_B,\Gamma_B]=\bigl(B\mathbb{Z}^3\bigr)\times\{0\}\subseteq \mathbb{Z}^3\times\{0\}$ and 
\[
\Gamma_B^{\mathrm{ab}}
=\Gamma_B/[\Gamma_B,\Gamma_B]
\cong (\mathbb{Z}^3\times\mathbb{Z}^4)/\bigl((B\mathbb{Z}^3)\times\{0\}\bigr)
\cong \mathbb{Z}^4\oplus(\mathbb{Z}^3/B\mathbb{Z}^3).
\]
\end{proof}

This implies that there are infinitely many topological types satisfying $d\underline{\omega}=0$. Indeed, the set of trace-free even integer matrices is a free abelian group of rank $8$:
\[
\{B\in M_{3\times 3}(2\mathbb{Z})\mid \mathrm{tr}(B)=0\}\cong \mathbb{Z}^8.
\]
Consider the diagonal family
\[
B_n=\mathrm{diag}(2n,\,2,\,-2n-2),\qquad n\in\mathbb{Z}.
\]
Then $\mathrm{tr}(B_n)=0$, so $d\underline{\omega}=0$ holds on $M_{B_n}$.
We have 
\[
\mathbb{Z}^3/B_n\mathbb{Z}^3
\cong \mathbb{Z}/|2n|\mathbb{Z}\ \oplus\ \mathbb{Z}/2\mathbb{Z}\ \oplus\ \mathbb{Z}/|2n+2|\mathbb{Z},
\]
with the convention $\mathbb{Z}/0\mathbb{Z}\cong \mathbb{Z}$.
So the torsion subgroup of $H_1(M_{B_n};\mathbb{Z})$ has order
\[
|2n|\cdot 2\cdot |2n+2|=|\det(B_n)|=8|n(n+1)|, \qquad 
n\neq 0,-1. 
\]
Since $|\det(B_n)|$ takes infinitely many distinct values, 
the manifolds $M_{B_n}$ are pairwise non-homeomorphic for infinitely many $n$.

\

Similarly, there are infinitely many topological types satisfying $d\underline{\Theta}=0$.
The set of symmetric even integer matrices is a free abelian group of rank $6$:
\[
\{B\in M_{3\times 3}(2\mathbb{Z})\mid B={}^t B \}\cong \mathbb{Z}^6.
\]
For instance, 
\[
B_n=\mathrm{diag}(2n,\,0,\,0),\qquad n\in\mathbb{Z},
\]
are symmetric, hence $d\underline{\Theta}=0$ (equivalently $d(*\varphi)=0$) holds on $M_{B_n}$.
In this case $\mathbb{Z}^3/B_n\mathbb{Z}^3\cong \mathbb{Z}^2\oplus \mathbb{Z}/|2n|\mathbb{Z}$, so the torsion in
$H_1(M_{B_n};\mathbb{Z})$ has order $|2n|$ for $n \neq 0$, and again one obtains infinitely many pairwise non-homeomorphic compact
nilmanifolds.\\

In order to construct adiabatic calibrated immersions, let $J_1,J_2,J_3$ be the almost complex structures on $V$ induced by $\omega_1,\omega_2,\omega_3$ 
and the metric $g|_V=\sum_{a=4}^7 \eta^a \otimes \eta^a$. 
They are given by the constant matrices 
$J^0_1,J^0_2,J^0_3$ as in \eqref{eq:cpx strs}, 
with respect to $\{ \eta_a \}$. 
In particular, 
$J_1J_2=-J_3$ and $J_iJ_j+J_jJ_i=0$ for $i\neq j$.

Let $q:N_B\to M_B$ be the quotient map. 
Fix an open set $U\subseteq\mathbb{R}^3$ with coordinates
$x=(x^1,x^2,x^3)$, and let 
$\tilde u = (\tilde u^4,\dots,\tilde u^7):
U \to \mathbb{R}^4$ 
be a smooth map. Set 
\[
\widetilde\iota: U\to N_B, 
\qquad 
\widetilde\iota(x):=(x,\tilde u(x)),
\qquad
\iota:=q\circ\widetilde\iota:U\to M_B.
\]
Using the left-invariant trivializations of $H$ and $V$ induced by
$\{e_i\}_{i=1}^3$ and $\{\eta_a\}_{a=4}^7$, respectively,
\eqref{eq:HV gHb} gives for $x \in U$ and $X \in T_x U$ 
\[
\iota_{*x}^H(X)
=
X+\frac12\beta_B
\bigl(\tilde u(x),\tilde u_{*x}(X)\bigr),
\qquad
\iota_{*x}^V(X)
=
\tilde u_{*x}(X).
\]

Assume that $\iota$ is positive horizontally projectable. 
Then $\iota_{*x}^H:T_xU\to \mathbb{R}^3$ 
is an orientation-preserving isomorphism for every $x\in U$.
Set
\[
(E_i)_x:=(\iota_{*x}^H)^{-1}(e_i),
\qquad i=1,2,3.
\]
Then $\{(E_1)_x,(E_2)_x,(E_3)_x\}$ is a positive orthonormal basis of
$(T_xU,g^H_\iota)$, and the adiabatic calibrated equation becomes
\[
(D_B\tilde u)_x
:=
\sum_{i=1}^3
J_i^0\left(
\tilde u_{*x} \left((E_i)_x\right)
\right)
=
0.
\]
Thus $D_B$ is, in general, a nonlinear first-order operator.
When $B=0$, $D_B$ reduces to the
flat Fueter operator $D$ considered in Section \ref{ss:HK example}.

\begin{example}[Explicit affine solutions]
\label{ex:qH affine solution}
Suppose that $e_3$ is an eigenvector of $B$, and write
$Be_3=\kappa e_3$ for some $\kappa\in\mathbb{R}$.
Choose $v\in\mathbb{R}^4$ and set
\[
\tilde u(x)
:=
x^1v+x^2J_3^0v.
\]
Writing
$\partial_i:=\partial/\partial x^i$, set
\[
E_1
:=
\partial_1
+
\frac{\kappa|v|^2}{2}x^2\partial_3,
\qquad
E_2
:=
\partial_2
-
\frac{\kappa|v|^2}{2}x^1\partial_3,
\qquad
E_3
:=
\partial_3.
\]
A direct calculation gives $\iota_*^H(E_i)=e_i$. 
Since $\{E_1,E_2,E_3\}$ is positively oriented, $\iota$ is positive
horizontally projectable, and this frame is orthonormal with respect
to $g^H_\iota$. We have 
\[
\tilde u_*(E_1)=v,
\qquad
\tilde u_*(E_2)=J_3^0v,
\qquad
\tilde u_*(E_3)=0.
\]
Therefore
\[
D_B\tilde u
=
J_1^0v+J_2^0J_3^0v
=
0, 
\]
where we used $J_2^0J_3^0=-J_1^0$.
Thus, for $v\neq0$, the corresponding immersion is an explicit
non-horizontal adiabatic calibrated immersion.
\end{example}

%4293
\begin{remark}
The manifold $M_B$ admits a torus bundle structure. 
Indeed, the projection 
\[
\pi:M_B\longrightarrow \mathbb{Z}^4\backslash\mathbb{R}^4=T^4, 
\qquad
[(x,y)] \mapsto [y]
\]
is well-defined by \eqref{eq:multiplication law}. 
The fibers are $\mathbb{Z}^3\backslash\mathbb{R}^3=T^3$, and 
$M_B$ is a principal $T^3$-bundle over the $4$-torus $T^4$.
Notice that the tangent distribution to these fibers is $H$, not $V$.

For more details on the quaternionic Heisenberg group, see, for example, 
\cite{FIUV}. 
Note, however, that the sign convention of the structure equations in \cite{FIUV} 
is different from ours, although the underlying Lie algebra is the same. 
Indeed, setting
\[
B=
\begin{pmatrix}
0&2&0\\
0&0&2\\
2&0&0
\end{pmatrix}, 
\qquad
(\gamma^1,\gamma^2,\gamma^3,\gamma^4)=(\eta^4,\eta^5,\eta^6,-\eta^7),
\qquad
(\gamma^5,\gamma^6,\gamma^7)=\frac{1}{2} (e^3, e^1, e^2),
\]
our structure equations become 
\[
d\gamma^1=d\gamma^2=d\gamma^3=d\gamma^4=0,\qquad
d\gamma^5=\gamma^{12}-\gamma^{34},\quad
d\gamma^6=\gamma^{13}+\gamma^{24},\quad
d\gamma^7=\gamma^{14}-\gamma^{23},
\]
which agree with \cite[(3.1)]{FIUV}.
However, this change of coframe does not preserve the $G_2$-structure. The $G_2$-structure in \cite[(3.4)]{FIUV} is coclosed, whereas ours is not.
\end{remark}

\section{The $G_2$ case: Abstract existence results} \label{sec:G2 CK}

When a calibration satisfies the equality property, calibrated submanifolds are defined (up to the choice of an orientation) by the vanishing of the corresponding $E$-valued differential form. Locally, this is equivalent to the vanishing of its components, thus of a finite number of $\rl$-valued differential forms. There exists a tool specifically designed to prove the existence of submanifolds defined this way: Cartan-K\"ahler theory. Since all results are local, we can focus on embedded, rather than immersed, submanifolds. In the following subsections $\Sigma$ will thus denote an embedded submanifold in $M$, dropping any reference to the specific immersion $\iota$. 

Cartan-K\"ahler theory requires all data to be real analytic. A torsion-free $G_2$-manifold and all corresponding tensors $\varphi, g, *\varphi, \chi$ are automatically analytic. In the general case, this will instead be an additional assumption on $M$.

\subsection{Review of Cartan-K\"ahler theory} \label{sec:review CK}
We shall collect here some basic facts, in the form most useful for our applications. See, for example, \cite{Spivak} 
for a more complete presentation.

\begin{definition}\label{def:I}
Choose a set $I$ of differential forms on $M$.

We will say that a $s$-dimensional subspace $W\subseteq T_xM$ is \textbf{integral} if the forms in $I$ vanish at $x$ after restriction to $W$. 

We will say that a $s$-dimensional submanifold $\Sigma\subseteq M$ is \textbf{integral} if the forms in $I$ vanish after restriction to $\Sigma$.
\end{definition}

Notice that $W$ is integral with respect to $I$ if and only if it is integral with respect to the ideal generated by $I$.

On the other hand, a submanifold $\Sigma$ is integral with respect to $I$ if and only if it is integral with respect to the ideal $\mathcal{I}$ generated by the forms in $I$ and by their differentials. Such $\mathcal{I}$ is then a \textbf{differential ideal}, in the sense that it is closed with respect to the exterior derivative $d$. 

We will usually be concerned with the case where $I$ is a finite set of differential forms of the same degree $k$. 
In this case: 

1. If $s<k$, any $W$, $\Sigma$ are automatically integral. The integrality condition is thus relevant only when $s\geq k$.

2. $\mathcal{I}$ is generated by the $k$-forms in $I$ and by the ($k+1$)-forms defined by their differentials. It is thus a \textbf{homogeneous ideal}, in the sense that $\mathcal{I}=\oplus_j \mathcal{I}^j$, where $\mathcal{I}^j$ denotes the subspace of forms in $\mathcal{I}$ of degree $j$. One can check that $W$, $\Sigma$ are integral with respect to $\mathcal{I}$ if and only if they are integral with respect to $\mathcal{I}^s$.

\ 

The pair $(M,\mathcal{I})$, where $\mathcal{I}$ is a homogeneous differential ideal, is known as an \textbf{exterior differential system} (EDS). Cartan-K\"ahler theory studies such systems in the analytic category. It requires the following notions.

\begin{definition}\label{def:calI}
Let $\mathcal{I}$ be a homogeneous differential ideal.

Let $W\subseteq T_xM$ be an integral $s$-plane with respect to $\mathcal{I}$. Choose a basis $\{v_1,\dots,v_s\}$. Its \textbf{polar space} is the subspace
\begin{align*}
\mathcal{E}(W)&=\{X\in T_xM: \omega(v_1,\dots,v_s,X)=0, \forall\omega\in \mathcal{I}\}\\
&=\{X\in T_xM: \omega(v_1,\dots,v_s,X)=0, \forall\omega\in \mathcal{I}^{s+1}\}.
\end{align*}
We say that $W$ is \textbf{regular} if the dimension of $\mathcal{E}(W)$ is locally constant, with respect to small integral perturbations of $W$ and of $x$.
\end{definition}

By definition, $W\subseteq \mathcal{E}(W)$. The quotient contains the possible directions in which $W$ can be extended so as to remain integral. Roughly speaking, regularity can be rephrased by saying: the dimension of $\mathcal{E}(W)/W$ is locally constant, i.e. nearby points/subspaces have the same number of integral extensions. 

\begin{definition}\label{def:flag}
Let $\mathcal{I}$ be a homogeneous differential ideal.

Let $W\subseteq T_xM$ be an integral $s$-plane with respect to $\mathcal{I}$. We say that $W$ contains a \textbf{regular flag} if there exists a flag
$$W_0=\{0\}\subseteq W_1\subseteq\dots\subseteq W_{s-1}\subseteq W,$$
where $W_j$ has dimension $j$ and each $W_0,\dots,W_{s-1}$ is regular.
\end{definition}

Notice that $W$ itself is not required to be regular. The notion of regular flag will be used in the proof of Theorem \ref{thm:CKI}.

\begin{remark}
Notice the difference between Definition \ref{def:I}, which refers to an arbitrary set of forms $I$, and Definitions \ref{def:calI}, \ref{def:flag}, which are formulated for a homogeneous differential ideal $\mathcal{I}$. 
When $I$ consists of $k$-forms and $s=k$, the notion of an integral $s$-plane depends only on
the $k$-forms in $I$ (since $(k+1)$-forms vanish automatically on $s$-planes). 
However, the polar space $\mathcal E(W)$ (hence regularity and the regular flag condition) is defined using
the $(s+1)$-forms in $\mathcal I$, and these include the forms coming from $dI$. It follows that regularity must be
checked with respect to $\mathcal I$ rather than the generating set $I$ alone.
\end{remark}

One version of the Cartan-K\"ahler theorem is as follows.

\ 
 
\begin{theorem}[Cartan-K\"ahler, I]\label{thm:CKI}
Let $M$ be an analytic manifold and let $\mathcal{I}$ be an analytic EDS. Choose an integral $s$-plane $W\subseteq T_xM$. If $W$ contains a regular flag, then there exists a local analytic $s$-dimensional integral submanifold whose tangent plane at $x$ is $W$.
\end{theorem}

The proof is inductive. Starting from $x$ (whose tangent space is $W_0$), it iteratively builds integral $(j+1)$-dimensional extensions of $j$-dimensional integral submanifolds, with tangent space $W_{j+1}$ at $x$. Each new extension requires the regularity of the given $W_j$.

\ 

This statement is very simple, but it gives no indication of the number of possible integral submanifolds. 

\begin{theorem}[Cartan-K\"ahler, II]\label{thm:CKII}
Let $M$ be an analytic manifold and let $\mathcal{I}$ be an analytic EDS. Choose an integral $(s-1)$-plane $W\subseteq T_xM$. Assume (i) $W$ is regular, (ii) $W$ contains a regular flag, (iii) $\dim(\mathcal{E}(W))=s$. Then any $(s-1)$-dimensional analytic integral submanifold whose tangent space at $x$ is $W$ admits a unique analytic $s$-dimensional integral extension whose tangent space at $x$ is $\mathcal{E}(W)$.
\end{theorem}

We will see below that this statement provides infinitely many examples of associative and adiabatic calibrated submanifolds.

\begin{remark} 
Roughly speaking, an integral $s$-plane $W$ is called ordinary if, in a neighbourhood of $W$, the subset of integral elements is a smooth subset of the Grassmannian. This condition is also necessary for CK but it is implied by the existence of a regular flag, so it can be omitted in the CK statement. 
\end{remark}

\subsection{Existence of associative submanifolds}

The facts in this subsection are well-known, going back to \cite{HL}. We include them here for completeness.

Let $M$ be an analytic manifold and let $\varphi$ be an analytic $G_2$-structure. Locally $\chi$ can be viewed, via its components, as a collection of seven $3$-forms. Let $I$ be the set of these $3$-forms and let $\mathcal{I}$ denote the corresponding analytic EDS. By definition a $3$-plane $W$ is associative if and only if it is integral with respect to $I$, but any 4-form vanishes automatically on $W$. It follows that $W$ is associative if and only if it is integral with respect to $\mathcal{I}$. We can thus study the local existence of associative submanifolds in $M$ via the Cartan-K\"ahler theorems, applied to the EDS $\mathcal{I}$.

Let us check the regularity assumptions:

1. Since $\mathcal{I}$ is generated by forms of degree 3 and 4, any subspace $W$ of dimension $s=0,1$ is automatically integral with respect to $\mathcal{I}$. Its polar space has dimension 7: this is constant with respect to perturbations. 

2. For the same reason, any subspace $W$ of dimension $s=2$ is automatically integral with respect to $\mathcal{I}$. Applying the cross product shows that $W$ is contained in a unique associative $3$-plane, so its polar space has dimension 3: this is constant with respect to perturbations.

This shows that, for $s=0,1,2$, any $s$-plane is regular. Theorems \ref{thm:CKI}, \ref{thm:CKII} thus lead to the following existence results.

\begin{corollary}Assume $M$ and the $G_2$-structure $\varphi$ are analytic.
Choose $x \in M$, an associative 3-plane $W$ at $x$ and any flag $\{W_i\}$ in $W$. Then
\begin{enumerate}
\item There exists an analytic associative submanifold tangent to $W$ at $x$.
\item Any analytic 2-dimensional submanifold tangent to $W_2$ admits a unique analytic associative extension tangent to $W$ at $x$.
\end{enumerate}
\end{corollary}

The initial $2$-dimensional submanifolds can be written as follows. Choose local coordinates $(x_1,\dots,x_7)$ on $M$ such that $x\cong 0$ and $W_2$ is generated by $\partial x_1$, $\partial x_2$.
The initial submanifold can then be locally projected onto the $2$-plane generated by $\partial x_1$, $\partial x_2$, so it can be written as the graph of some analytic function $f(x_1,x_2)$ with $5$ components such that $f(0,0)=0$ and $\nabla f(0,0)=0$. Any such submanifold admits a unique extension to an associative submanifold tangent to $W$ at $x$. Since there are infinitely many such functions, there are infinitely many associative submanifolds tangent to $W$.

\begin{remark}
Cartan-K\"ahler theory provides the tools to be more precise regarding the parameters involved in such existence results, but this will not be necessary for our purposes.
\end{remark}

\subsection{Existence of adiabatic calibrated submanifolds}

The existence theory for adiabatic calibrated submanifolds works in exactly the same way. In this case, according to Corollary \ref{cor:vanishing eq vanishing uth}, we need to find local integral submanifolds for the EDS $\mathcal{I}$ generated by the set $I$ defined by the 4 components of the 3-form $\chi_1$. As before, a subspace $W\in G_H^+(3,M)$ is adiabatic calibrated if and only if it is integral with respect to $\mathcal{I}$.

The regularity assumptions hold in the same way as in the associative case: any projectable subspace of dimension $s=0,1$ is automatically integral and regular, and any projectable subspace of dimension $s=2$ is integral and regular thanks to Lemma \ref{lem:Fueter}.

\begin{corollary}\label{cor:Fueter CK}
Assume $M$, $H$ and the $G_2$-structure $\varphi$ are analytic.
Choose $x \in M$, an adiabatic calibrated 3-plane $W$ at $x$ and any flag $\{W_i\}$ in $W$. Then
\begin{enumerate}
\item There exists an analytic adiabatic calibrated submanifold tangent to $W$ at $x$.
\item Any analytic 2-dimensional submanifold tangent to $W_2$ admits a unique analytic adiabatic calibrated extension tangent to $W$ at $x$.
\end{enumerate}
\end{corollary}

%\end{CJK}
\end{document}